\documentclass[12pt]{article}
\usepackage{amsmath,amssymb,amsfonts,amsthm}
\newenvironment{myabstract}{\par\noindent
{\bf Abstract . } \small }
{\par\vskip8pt minus3pt\rm}
\newcounter{item}[section]
\newcounter{kirshr}
\newcounter{kirsha}
\newcounter{kirshb}
\newenvironment{enumroman}{\setcounter{kirshr}{1}
\begin{list}{(\roman{kirshr})}{\usecounter{kirshr}} }{\end{list}}
\newenvironment{enumarab}{\setcounter{kirshb}{1}
\begin{list}{(\arabic{kirshb})}{\usecounter{kirshb}} }{\end{list}}
\newenvironment{athm}[1]{\vskip3mm\par\noindent
{\bf #1 }. \slshape }
{\upshape\par\vskip10pt minus3pt}
\newtheorem{theorem}{Theorem}[section]

\newtheorem{lemma}[theorem]{Lemma}
\newtheorem{corollary}[theorem]{Corollary}
\newenvironment{demo}[1]{\noindent{\bf #1.}\upshape\mdseries}
{\nopagebreak{\hfill\rule{2mm}{2mm}\nopagebreak}\par\normalfont}
\theoremstyle{definition}

\newtheorem{definition}[theorem]{Definition}
\def\R{\mathbb{R}}

\def\C{{\mathfrak{C}}}

\def\At{{\bf At}}
\def\Nr{{\mathfrak{Nr}}}

\def\Sg{{\mathfrak{Sg}}}

\def\A{{\mathfrak{A}}}
\def\B{{\mathfrak{B}}}
\def\C{{\mathfrak{C}}}
\def\D{{\mathfrak{D}}}
\def\M{{\mathfrak{M}}}
\def\N{{\mathfrak{N}}}

\def\CA{{\bf CA}}

\def\Df{{\bf Df}}

\def\Lf{{\bf Lf}}

\def\RCA{{\bf RCA}}

\def\Rd{{\ Rd}}
\def\(R)RA{{\bf (R)RA}}

\def\RRA{{\bf RRA}}

\def\R{\mathbb{R}}

\def\c #1{{\cal #1}}
 \def\CA{{\sf CA}}
\def\B{{\sf B}}

\def\w{{\sf w}}
\def\y{{\sf y}}
\def\g{{\sf g}}

\def\r{{\sf r}}

 \def\Cm{{\mathfrak{Cm}}}
\def\Nr{{\mathfrak{Nr}}}

\def\restr #1{{\restriction_{#1}}}
\def\cyl#1{{\sf c}_{#1}}

\def\set#1{\{#1\} }

\def\Nr{{\mathfrak{Nr}}}
\def\Tm{{\mathfrak{Tm}}}
\def\A{{\mathfrak{A}}}
\def\B{{\mathfrak{B}}}
\def\C{{\mathfrak{C}}}
\def\D{{\mathfrak{D}}}

\def\A{{\mathfrak{A}}}
\def\B{{\mathfrak{B}}}
\def\C{{\mathfrak{C}}}
\def\D{{\mathfrak{D}}}

\def\L{{\mathfrak{L}}}
\def\Rd{{\mathfrak{Rd}}}

\def\At{{\mathfrak{At}}}
\def\L{{\mathfrak{L}}}

\def\CA{{\bf CA}}

\def\RRA{{\bf RRA}}
\def\RCA{{\bf RCA}}

\def\F{{\mathfrak{F}}}
\def\At{{\sf{At}}}
\def\N{\mathbb{N}}
\def\R{\mathfrak{R}}

\def\CCA{{\bf CCA}}

\def\cyl#1{{\sf c}_{#1}}

\def\c #1{{\cal #1}}

\def\pa{$\forall$}
\def\pe{$\exists$}

\def\nodes{{\sf nodes}}

\def\restr #1{{\restriction_{#1}}}

\def\Nr{{\mathfrak{Nr}}}
\def\Z{{\cal Z}}
\def\CA{{\bf CA}}
\def\RCA{{\bf RCA}}
\def\c#1{{\mathcal #1}}

\def\set#1{ \{#1\}}

\def\pe{$\exists$}
\def\pa{$\forall$}
\def\Cm{{\mathfrak Cm}}
\def\Sg{{\mathfrak Sg}}

\def\Rl{{\mathfrak Rl}}
\def\N{{\cal N}}

\def\At{{\sf At}}

\def\rng{{\sf rng}}
\def\dom{{\sf dom}}

\def\w{{\sf w}}
\def\g{{\sf g}}
\def\y{{\sf y}}
\def\r{{\sf r}}
\def\bb{{\sf b}} 

\def\cyl#1{{\sf c}_{#1}}

\def\ws{winning strategy}

\def\y{{\sf y}}
\def\g{{\sf g}}
\def\bb{{\sf b}}
\def\r{{\sf r}}
\def\w{{\sf w}}

\title{On Completions, neat atom structures, and omitting types}
\author{Tarek Sayed Ahmed \\
Department of Mathematics, Faculty of Science,\\ 
Cairo University, Giza, Egypt.
  }
\begin{document}
\maketitle
\begin{myabstract} This paper has a survey character, but it also contains several new results.
The paper tries to give a panoramic picture of the recent developments in algebraic logic.
We take a long magical tour in algebraic logic starting from classical notions due to Henkin Monk and Tarski
like neat embeddings, culminating in presenting sophisticated model theoretic constructions based on graphs, to solve problems on neat reducts.

We investigate several algebraic notions that apply to varieties of 
boolean algebras with operators in general $BAOs$, like canonicity, atom-canonicity and completions. We also show that in certain significant
special cases, when we have a Stone-like representabilty notion, like in cylindric, relation and polyadic algebras
such abtract notions turn out intimately related 
to more concrete notions, like complete representability, 
and the existence of weakly but not srongly representable atom structures.

In connection to the multi-dimensional  corresponding modal logic, we 
prove several omitting types theorem for the finite $n$ variable fragments of first order logic, 
the multi-dimensional modal logic corresponding to $CA_n$; the class of cylindric algebras of dimension $n$.

A novelty that occurs here is that
theories could be uncountable. Our constructions depend on deep model-theoretic results of Shelah. 

Several results mentioned in \cite{OTT} without proofs are proved fully here, one such result is:
There exists an uncountable atomic algebra in $\Nr_n\CA_{\omega}$ that is not completely representable. 
Another result: If $T$ is an $L_n$ theory (possibly uncountable), where $|T|=\lambda$, $\lambda$ is a regular cardinal, and $T$ admits elimination of 
quantifiers, then
$<2^{\lambda}$ non principal types {\it maximal} can be omitted. 

A central notion, that connects, representations, completions, complete representations for cylindric algebras is that of neat embedding,
which is an algebraic counterpart of Henkin constructions, and is a nut cracker in cylindric-like 
algebras for proving representation results and related results concerning various forms of the amalagmation property
for clases of representable algebras. 
For example, representable algebras are those algebras that have the neat embeding property, 
completey representable countable ones, are the atomic algebras that have the complete neat embedding 
property. We show that countability cannot be omitted which is sharp in view to our omitting types theorem mentiond above.
We show that the class $\Nr_n\CA_{\omega}$ is psuedo-elementary, not elementary, and its elementary closure is not finitely axiomatizable
for $n\geq 3$, We  characterize this class by games.

We give two constructions for weakly representable atom structures that is not 
strongly representable, that are simple variations on existing themes, and using fairly straightforward modificatons 
of constructons of Hirsch and Hodkinson, we show that
the latter class is not elementary for any reduct of polyadic algebras contaiting all cylindrifiers. 

We introduce the new notions of strongly neat, weakly neat and very weakly neat atom structures.
An $\alpha$ dimensional  atom strucure is  very weakly neat, $\alpha$
an ordinal,  if no algebra based on it in 
$\Nr_{\alpha}\CA_{\alpha+\omega};$ weakly neat if it has at least one algebra based on it that 
is in $\Nr_{\alpha}\CA_{\alpha+\omega}$, and finally  strongly neat if every algebra based on it is in 
$\Nr_{\alpha}\CA_{\alpha+\omega}.$ We give examples of the first two, show that they are distinct, 
and further show that the class of weakly 
neat atom structures is not elementary.
This is done for all dimensions $>2$, infinite included. For the third, we show that finite atom structures are strongly neat (in finite dimensions).

Modifying several constructions in the literature, as well as providing new ones, several 
results on complete representations and completions are presented, 
answering several questions posed by Robin Hirsch, and Ian Hodkinson, concerning relation algebras, and complete representabiliy
of reducts of plyadic algebras.
 
\end{myabstract}

\section{Introduction}

Atom canonicity, completions, complete representations and omitting types are  four notions that could appear at first glimpse unrelated. 
The first three, are algebraic notions that apply 
to varieties of Boolean algebras with operators $BAO$s. Omitting types is a metalogical notion that applies to the corresponding multi-modal logic. 

Canonicity is one of the most important concepts of modal 
completeness theory. From an algebraic perspective, canonical models are not abstract oddities, 
they are precisey the structure one is led to 
by underlying the idea in Stones representabilty theory for Boolean algebras.

The canonical extension of an algebra  has universe the power set algebra of 
the set of ultrafilters; that is its Stone space, and  the extra non-Boolean operations induced naturally from the original ones.
A variety is canonical if it is closed under taking canonical
extensions. 

This is typically a {\it persistence property}. Persistence properties refer to closure of a variety $V$  under 
passage from a given member in $V,$ to some 'larger' algebra .

The other  persistence property, namely, atom-canonicity, concerns the atom structure $\At\A$ of an atomic algebra $\A$. 
As the name suggests, $\At\A$ is a certain relational structure based on the set of atoms of $\A$. 
A variety $V$ is atom-canonical if it contains the complex algebra $\Cm\At(\A)$ whenever it contains $\A$.
If $\At V$ is the class of all atom structures of atomic algebras in $V$, 
then atom-canonicity amounts to the requirement that $\Cm\At V\subseteq V$.%

The (canonical) models of a multi-modal logic $\L_V$, corresponding to a canonical variety $V$,  are Kripke frames; 
these are the ultrafilter frames. The  atom structures, are special cases, we call these atomic 
models.

The canonical extension of an algebra is a complete atomic algebra that the original algebra embeds into, however it only preserves finite meets. 
Another completion is the Dedekind MacNeille completion, also known as the 
minimal completion. Every completely  additive $BAO$ has such a completion. It is uniquely 
determined by the fact that it is complete and the original algebra is dense in it; hence it preserves all meets existing in the original 
algebra.
The completion is atomic if and only if the original algebra is. The completion and canonical extension of an algebra 
only coincide when the algebra is finite.

Complete representations has to do with algebras that have a notion of representations involving - using jargon of modal logic - complex algebras 
of square frames or, using algebraic logic jargon, full set algebras having square units, and also having the 
Boolean operations of (concrete) intersections and complementation,
like relation algebras and cylindric algebras.  Unlike atom-canonicity, this notion is semantical.
Such  representations is a representation that carries existing (possibly infinite) meets to set theoretic 
intersections.

Atomic representability is also related to the Orey-Henkin omitting types theorem. Let $V$ be a variety of $BAO$'s which has a notion of representation,
like for example $CA_n$. The variety $CA_n$ corresponds to the syntactical part of $L_n$ 
first order logic restricted to the first $n$ variables, while $RCA_n$ corresponds to the semantical models 

Indeed given an $L_n$ theory $T$ and a model $\M$ of $T$, let $\phi^M$ denote the set of$n$-ary assignments satisfying
$\phi$ in $\M$, notationaly $\phi^{\M}=\{s\in {}^nM: \M\models \phi[s]\}$.
Then $\{\phi^M:\phi\in L\}$ is the universe of a cylindric set algebra $Cs_n$ with the operations of cyindrifiers corresponding to 
the semantics of existential quabtifiers
and diagonal elements to equality. The class of subdirect products of algebras in 
$Cs_n$ is the class $RCA_n$.

A set $\Gamma$ in the language of $T$ is said to be omitted by the model $\M$ of $T$, 
if $\bigcap \phi^{\M}=\emptyset$. This can be formulated algebraically as follows:
Let $\A\in \CA_n$ and  let there be given a family $(X_i: i\in I)$ of subsets of $\A$,  then there exists an injective homomorphism 
$f:\A\to \wp(V)$, $V$ a disjoint union of cartesian squares, that  omits the $X_i$'s, that is $\bigcap_{x\in X_i} f(x)=0$.

The Orey-Henkin omitting types theorem says that this always happens if $\A$ is countable and locally finite,
and when $I$ is countable, and  the $X_i'$ contain only finitely many 
dimensions (free variables). 
But it is clear that the above algebraic formulation lends itself to other contexts.

Note that if omitting types theorem holds for arbitrary cardinalities, then the 
atomic representable algebras are completely representable, by finding a representation that omits the 
non principal types $\{-x: x\in \At\A\}$. The converse is false. There are easy examples. Some will be provided below.

Now let us try to find a connection between such notions. 
Consider a variety$V$ of $BAO$s that is {\it not} atom-canonical. Ths means that there is an $\A\in V$, such that $\Cm\At\A\notin V$. 
If $V$ is completely aditive, then $\Cm\At\A$, is the completion of $\A$. So $V$ is not closed 
under completions. 

There are significant varieties that are not atom-canonical, like the variety of representable cylindric algebras 
$\RCA_n$ for $n\geq 3$ and representable relation algebras $\RRA$.
These have a notion of representability involving square units. 

Let $\A$ be an atomic representable such that $\Cm\At\A\notin \RCA_{n}$, for $n\geq 3$. Such algebras exist, the first of its kind was constructed
by Hodkinson. The term algera $\Tm\At\A$, which is the subalgebra of the complex algebra,
is contained in $\A$, because $\RCA_n$ is completely additiv; furthermore, it is 
representable but it {\it cannot} have a complete representation, for such a representation
would necessarily induce a representation of $\Cm\At\A$. That is to say, $\A$ is an example of an atomic theory in the cylindric modal logic of dimension
$n$, but it has no atomic model. 
In such a context, $\At\A$ is also an example of a {\it weakly representable} atom structure that is not {\it strongly representable.}

A weakly representable atom structure is an atom structure such that there is at least one algebra based on this atom structure that is 
representable \cite{w}. It is strongly representable, if every algebra based on it is representable.
Hodkinson, the first to construct a weakly representable cylindric atom structure that is not strongly representable, 
used a somewhat complicated construction depending on so called Rainbow 
constructions. His proof is model-theoretic.
In \cite{w} we use the same method to construct such an atom structure for both relation and cylindric algebras.
On the one hand, the graph we used substantially simplifies Hodkinson's construction, 
and, furthermore, we get our result for relation and cylindric algebras in one go.


Hirsch and Hodkinson show that the class of strongly representable atom structures of relation algebras 
(and cylindric algebras) is not elementary \cite{strong}.
The construction makes use of the pobabilistic method of  Erd\"os to show that there are finite graphs with arbitrarily large 
chromatic number and girth.   
In his pioneering paper of 1959, Erdos took a radically new approach to construct such graphs: for each $n$ he 
defined a probability space on the set of graphs with $n$ vertices, and showed that, for some carefully chosen probability measures, 
the probability that an $n$ vertex graph has these properties is positive for all large enough $n$.
This approach, now called the {\it probabilistic method} has since unfolded into a 
sophisticated and versatile proof technique, in graph theory and in other 
branches of discrete mathematics.
This method was used first in algebraic logic by Hirsch and Hodkinson to show that the class of strongly representable atom structures of cylindric 
and relation algebras is not elementary  
and that varieties of representable relation algebras are barely canonical \cite{hv}.
But yet again using these methods of Erd\"os in  \cite{ghv} it is shown 
that there exist continuum-many canonical equational classes of Boolean algebras 
with operators that are not generated by the complex algebras of any first-order 
definable class of relational structures. Using a variant of this construction  the authors resolve the long-standing question of Fine, 
by exhibiting a bimodal logic that is valid in its canonical frames, but is not sound and complete for any first-order definable class of 
Kripke frames. 

There is an ongoing interplay between algebraic logic on the one hand, and model theory and finite combinatorics particularly graph theory, 
on the other. 
Monk was the first to use Ramsey's theorems to construct what is known an Monk's algebras.
witnessing non finite axiomatizability for 
the class of representable cylindric algebras.
The key idea of the construction of a Monk's algebra is not so hard. 
Such algebras are finite, hence atomic, more precisely their Boolean reducts are atomic.
The atoms are given colours, and 
cylindrifications and diagonals are defined by stating that monochromatic triangles 
are inconsistent. If a Monk's algebra has many more atoms than colours, 
it follows from Ramsey's Theorem that any representation
of the algebra must contain a monochromatic triangle, so the algebra
is not representable. 

Later Monk-like constructions were substantially generalized by Andr\'eka N\'emeti \cite{ANT}, 
Maddux \cite{m}, and finally Hirsch and Hodkinson \cite{HHM}.
Constructing algebras from Erdos graphs have 
proved extremely rewarding \cite{strong}, \cite{hv}, \cite{ghv}. 
Another construction invented by Robin Hirsch and Ian Hodkinson is the so-called rainbow construction, which  is an ingenious technique 
that has been used to show that several classes are not elementary \cite{HH}, \cite{HHbook}, \cite{r}, and was used together with a 
lifting argument of Hodkinson of construction polyadic algebras from relation algebras to show that it is undecidable whether a finite relation or 
cylindric algebra is representable.
This shows that  ceratin important products modal logics are undecidable.
We will use the rainbow construction below to prove the $CA$ analogue of a deep result in \cite{r}.

Constructing cylindric algebras based on certain models satisfyig certain properties like homogeniouty, saturation, 
elimination of quantifiers, using model theory like in \cite {MLQ}, will be generalized below to answer a question of Hirsch \cite{r},
on relation algebra reducts of cylindric algebras.

Another model theoretic construction of Hodkinson, 
based on rainbow graphs, considerably simplified in \cite{w} will be further simplified here 
to prove  that several varieties approximating the class of representable cylindric algebras are not closed under completions.

\section*{ The main new results}

\begin{enumarab}

\item  Answering a question of Robin Hirsch in \cite{R} on complete representations for both relation and cylindric algebras 
using an example in the same paper. This example shows that in the characterization of countable completely representable algebras,
both relation and cylindric algebras, the condition of countability is necessary, it {\it cannot} be omitted.

\item Using an example by Andr\'eka et all, to show that unlike cylindric and polyadic equality algebras, 
atomic polyadic algebras, and Pinters substitution algebras, even without cylindrfiers, of 
dimension $2$ may  not be completely representable. However, the class of completely representable algebras 
is not so hard to characterize; it is finitely axiomatizable in first order logic.
This is contrary to a current belief for polyadic algebras, and is an answer to a result of Hodkinson for cylindifier free
Pinter's algebras.

\item Using the construction in Andr\'eka et all \cite{ANT}, 
showing that the omitting types theorem fails for finite first 
order definable extension of first order logic as defined by Tarski and Givant, and further pursued by others, like 
Maddux and Biro,  a result mentioned in the above cited paper without a proof.

\item Characterizing the class $\Nr_n\CA_{\omega}$ by games, and 
showing that the class $\Nr_n\CA_{\omega}$ is pseudo elementary, and its elementary closure is not finitely axiomatizable.

\item Characterizing the class of countable completely representable algebras of infinitely countable dimensions using weak representations 
(the question remains whether this class is elementary, 
the Hirsch Hodkinson example depending on a cardinality argument does not works when our units are weak spaces.)


\item Giving full proofs to three results mentioned in \cite{OTT} without proofs, referring to a pre-print, 
concerning omitting types in uncountable theories using finitely
many variables. This is the pre print, expanded, modified and polished containing proofs of these results and much more.
The results concerning omitting types depend on deep model-theoretic constructions of Shelah's.

\item We show that the class of weakly neat atom structures, as defined in the abstract, is not elementary for every dimension.

\item Unlike the cylindric case, we show that atomic polyadic algebras of infinite dimensions are completely representable.

\end{enumarab}

This paper also simplifies existing proofs in the literature, like the proof in \cite{ANT}, concerning complete representations of relation 
atom structures, having cylindric basis. Some classical results, like Monk's non-finitizability results for relation and cylindric algebras
are also re-proved.


In this preprint we ony deal with the notin of weakly neat aom structurs. A longer preprint contains all other results.

\section{Neat reducts, complete representations and games}

Next we characterize the class $\Nr_n\CA_{\omega}$ using games.
Our treatment in this part 
follows very closely \cite{R}. The essential difference is that we deal with $n$ 
dimensional networks and composition moves are replaced by cylindrifier moves in the games.

\begin{definition}\label{def:string} 
Let $n$ be an ordinal. An $s$ word is a finite string of substitutions $({\sf s}_i^j)$, 
a $c$ word is a finite string of cylindrifications $({\sf c}_k)$.
An $sc$ word is a finite string of substitutions and cylindrifications
Any $sc$ word $w$ induces a partial map $\hat{w}:n\to n$
by
\begin{itemize}

\item $\hat{\epsilon}=Id$

\item $\widehat{w_j^i}=\hat{w}\circ [i|j]$

\item $\widehat{w{\sf c}_i}= \hat{w}\upharpoonright(n\sim \{i\}$ 

\end{itemize}
\end{definition}
If $\bar a\in {}^{<n-1}n$, we write ${\sf s}_{\bar a}$, or more frequently 
${\sf s}_{a_0\ldots a_{k-1}}$, where $k=|\bar a|$,
for an an arbitary chosen $sc$ word $w$
such that $\hat{w}=\bar a.$ 
$w$  exists and does not 
depend on $w$ by \cite[definition~5.23 ~lemma 13.29]{HHbook}. 
We can, and will assume \cite[Lemma 13.29]{HHbook} 
that $w=s{\sf c}_{n-1}{\sf c}_n.$
[In the notation of \cite[definition~5.23,~lemma~13.29]{HHbook}, 
$\widehat{s_{ijk}}$ for example is the function $n\to n$ taking $0$ to $i,$
$1$ to $j$ and $2$ to $k$, and fixing all $l\in n\setminus\set{i, j,k}$.]
Let $\delta$ be a map. Then $\delta[i\to d]$ is defined as follows. $\delta[i\to d](x)=\delta(x)$
if $x\neq i$ and $\delta[i\to d](i)=d$. We write $\delta_i^j$ for $\delta[i\to \delta_j]$.

\begin{definition}
From now on let $2\leq n<\omega.$ Let $\C$ be an atomic $\CA_{n}$. 
An \emph{atomic  network} over $\C$ is a map
$$N: {}^{n}\Delta\to At\cal C$$ 
such that the following hold for each $i,j<n$, $\delta\in {}^{n}\Delta$
and $d\in \Delta$:
\begin{itemize}
\item $N(\delta^i_j)\leq {\sf d}_{ij}$
\item $N(\delta[i\to d])\leq {\sf c}_iN(\delta)$ 
\end{itemize}
\end{definition}
Note than $N$ can be viewed as a hypergraph with set of nodes $\Delta$ and 
each hyperedge in ${}^{\mu}\Delta$ is labelled with an atom from $\C$.
We call such hyperedges atomic hyperedges.
We write $\nodes(N)$ for $\Delta.$ But it can happen 
let $N$ stand for the set of nodes 
as well as for the function and the network itself. Context will help.

Define $x\sim y$ if there exists $\bar{z}$ such that $N(x,y,\bar{z})\leq {\sf d}_{01}$.
Define an equivalence relation
$\sim$ over the set of all finite sequences over $\nodes(N)$ by $\bar
x\sim\bar y$ iff $|\bar x|=|\bar y|$ and $x_i\sim y_i$ for all
$i<|\bar x|$.

(3) A \emph{ hypernetwork} $N=(N^a, N^h)$ over $\cal C$ 
consists of a network $N^a$
together with a labelling function for hyperlabels $N^h:\;\;^{<
\omega}\!\nodes(N)\to\Lambda$ (some arbitrary set of hyperlabels $\Lambda$)
such that for $\bar x, \bar y\in\; ^{< \omega}\!\nodes(N)$ 
\begin{enumerate}
\renewcommand{\theenumi}{\Roman{enumi}}
\setcounter{enumi}3
\item\label{net:hyper} $\bar x\sim\bar y \Rightarrow N^h(\bar x)=N^h(\bar y)$. 
\end{enumerate}
If $|\bar x|=k\in nats$ and $N^h(\bar x)=\lambda$ then we say that $\lambda$ is
a $k$-ary hyperlabel. $(\bar x)$ is referred to a a $k$-ary hyperedge, or simply a hyperedge.
(Note that we have atomic hyperedges and hyperedges) 
When there is no risk of ambiguity we may drop the superscripts $a,
h$. 

The following notation is defined for hypernetworks, but applies
equally to networks.  

(4) If $N$ is a hypernetwork and $S$ is any set then
$N\restr S$ is the $n$-dimensional hypernetwork defined by restricting
$N$ to the set of nodes $S\cap\nodes(N)$.  For hypernetworks $M, N$ if
there is a set $S$ such that $M=N\restr S$ then we write $M\subseteq
N$.  If $N_0\subseteq N_1\subseteq \ldots $ is a nested sequence of
hypernetworks then we let the \emph{limit} $N=\bigcup_{i<\omega}N_i$  be
the hypernetwork defined by
$\nodes(N)=\bigcup_{i<\omega}\nodes(N_i)$,\/ $N^a(x_0,\ldots x_{n-1})= 
N_i^a(x_0,\ldots x_{n-1})$ if
$x_0\ldots x_{\mu-1}\in\nodes(N_i)$, and $N^h(\bar x)=N_i^h(\bar x)$ if $\rng(\bar
x)\subseteq\nodes(N_i)$.  This is well-defined since the hypernetworks
are nested and since hyperedges $\bar x\in\;^{<\omega}\nodes(N)$ are
only finitely long.

For hypernetworks $M, N$ and any set $S$, we write $M\equiv^SN$
if $N\restr S=M\restr S$.  For hypernetworks $M, N$, 
and any set $S$, we write $M\equiv_SN$ 
if the symmetric difference $\Delta(\nodes(M), \nodes(N))\subseteq S$ and
$M\equiv^{(\nodes(M)\cup\nodes(N))\setminus S}N$. We write $M\equiv_kN$ for
$M\equiv_{\set k}N$.

Let $N$ be a network and let $\theta$ be any function.  The network
$N\theta$ is a complete labelled graph with nodes
$\theta^{-1}(\nodes(N))=\set{x\in\dom(\theta):\theta(x)\in\nodes(N)}$,
and labelling defined by 
$(N\theta)(i_0,\ldots i_{\mu-1}) = N(\theta(i_0), \theta(i_1), \theta(i_{\mu-1}))$,
for $i_0, \ldots i_{\mu-1}\in\theta^{-1}(\nodes(N))$.  Similarly, for a hypernetwork
$N=(N^a, N^h)$, we define $N\theta$ to be the hypernetwork
$(N^a\theta, N^h\theta)$ with hyperlabelling defined by
$N^h\theta(x_0, x_1, \ldots) = N^h(\theta(x_0), \theta(x_1), \ldots)$
for $(x_0, x_1,\ldots) \in \;^{<\omega}\!\theta^{-1}(\nodes(N))$.

Let $M, N$ be hypernetworks.  A \emph{partial isomorphism}
$\theta:M\to N$ is a partial map $\theta:\nodes(M)\to\nodes(N)$ such
that for any $
i_i\ldots i_{\mu-1}\in\dom(\theta)\subseteq\nodes(M)$ we have $M^a(i_1,\ldots i_{\mu-1})= 
N^a(\theta(i), \ldots\theta(i_{\mu-1}))$
and for any finite sequence $\bar x\in\;^{<\omega}\!\dom(\theta)$ we
have $M^h(\bar x) = 
N^h\theta(\bar x)$.  
If $M=N$ we may call $\theta$ a partial isomorphism of $N$.

\begin{definition}\label{def:games} Let $2\leq n<\omega$. For any $\CA_{n}$  
atom structure $\alpha$, and $n\leq m\leq
\omega$, we define two-player games $F_{n}^m(\alpha),$ \; and 
$H_{n}(\alpha)$,
each with $\omega$ rounds, 
and for $m<\omega$ we define $H_{m,n}(\alpha)$ with $n$ rounds.

\begin{itemize}
\item 
Let $m\leq \omega$.  
In a play of $F_{n}^m(\alpha)$ the two players construct a sequence of
networks $N_0, N_1,\ldots$ where $\nodes(N_i)$ is a finite subset of
$m=\set{j:j<m}$, for each $i$.  In the initial round of this game \pa\
picks any atom $a\in\alpha$ and \pe\ must play a finite network $N_0$ with
$\nodes(N_0)\subseteq  n$, 
such that $N_0(\bar{d}) = a$ 
for some $\bar{d}\in{}^{\mu}\nodes(N_0)$.
In a subsequent round of a play of $F_{n}^m(\alpha)$ \pa\ can pick a
previously played network $N$ an index $\l<n$, a ``face" 
$F=\langle f_0,\ldots f_{n-2} \rangle \in{}^{n-2}\nodes(N),\; k\in
m\setminus\set{f_0,\ldots f_{n-2}}$, and an atom $b\in\alpha$ such that 
$b\leq {\sf c}_lN(f_0,\ldots f_i, x,\ldots f_{n-2}).$  
(the choice of $x$ here is arbitrary, 
as the second part of the definition of an atomic network together with the fact
that $\cyl i(\cyl i x)=\cyl ix$ ensures that the right hand side does not depend on $x$).
This move is called a \emph{cylindrifier move} and is denoted
$(N, \langle f_0, \ldots f_{\mu-2}\rangle, k, b, l)$ or simply $(N, F,k, b, l)$.
In order to make a legal response, \pe\ must play a
network $M\supseteq N$ such that 
$M(f_0,\ldots f_{i-1}, k, f_i,\ldots f_{n-2}))=b$ 
and $\nodes(M)=\nodes(N)\cup\set k$.

\pe\ wins $F_{n}^m(\alpha)$ if she responds with a legal move in each of the
$\omega$ rounds.  If she fails to make a legal response in any
round then \pa\ wins.

\item
Fix some hyperlabel $\lambda_0$.  $H_{n}(\alpha)$ is  a 
game the play of which consists of a sequence of
$\lambda_0$-neat hypernetworks 
$N_0, N_1,\ldots$ where $\nodes(N_i)$
is a finite subset of $\omega$, for each $i<\omega$.  
In the initial round \pa\ picks $a\in\alpha$ and \pe\ must play
a $\lambda_0$-neat hypernetwork $N_0$ with nodes contained in
$\mu$ and $N_0(\bar d)=a$ for some nodes $\bar{d}\in {}^{\mu}N_0$.  
At a later stage
\pa\ can make any cylindrifier move $(N, F,k, b, l)$ by picking a
previously played hypernetwork $N$ and $F\in {}^{n-2}\nodes(N), \;l<n,  
k\in\omega\setminus\nodes(N)$ 
and $b\leq {\sf c}_lN(f_0, f_{l-1}, x, f_{n-2})$.  
[In $H_{n}$ we
require that \pa\ chooses $k$ as a `new node', i.e. not in
$\nodes(N)$, whereas in $F_{n}^m$ for finite $m$ it was necessary to allow
\pa\ to `reuse old nodes'. This makes the game easior as far as $\forall$ is concerned.) 
For a legal response, \pe\ must play a
$\lambda_0$-neat hypernetwork $M\equiv_k N$ where
$\nodes(M)=\nodes(N)\cup\set k$ and 
$M(f_0, f_{i-1}, k, f_{n-2})=b$.
Alternatively, \pa\ can play a \emph{transformation move} by picking a
previously played hypernetwork $N$ and a partial, finite surjection
$\theta:\omega\to\nodes(N)$, this move is denoted $(N, \theta)$.  \pe\
must respond with $N\theta$.  Finally, \pa\ can play an
\emph{amalgamation move} by picking previously played hypernetworks
$M, N$ such that $M\equiv^{\nodes(M)\cap\nodes(N)}N$ and
$\nodes(M)\cap\nodes(N)\neq \emptyset$.  
This move is denoted $(M,
N)$.  To make a legal response, \pe\ must play a $\lambda_0$-neat
hypernetwork $L$ extending $M$ and $N$, where
$\nodes(L)=\nodes(M)\cup\nodes(N)$.

Again, \pe\ wins $H_n(\alpha)$ if she responds legally in each of the
$\omega$ rounds, otherwise \pa\ wins. 

\item For $m< \omega$ the game $H_{m,n}(\alpha)$ is similar to $H_n(\alpha)$ but
play ends after $m$ rounds, so a play of $H_{m,n}(\alpha)$ could be
\[N_0, N_1, \ldots, N_m\]
If \pe\ responds legally in each of these
$m$ rounds she wins, otherwise \pa\ wins.
\end{itemize}

\end{definition}

\begin{definition}\label{def:hat}
For $m\geq 5$ and $\c C\in\CA_m$, if $\A\subseteq\Nr_n(\C)$ is an
atomic cylindric algebra and $N$ is an $\A$-network then we define
$\widehat N\in\C$ by
\[\widehat N =
 \prod_{i_0,\ldots i_{n-1}\in\nodes(N)}{\sf s}_{i_0, \ldots i_{n-1}}N(i_0\ldots i_{n-1})\]
$\widehat N\in\C$ depends
implicitly on $\C$.
\end{definition}
We write $\A\subseteq_c \B$ if $\A\in S_c\{\B\}$. 
\begin{lemma}\label{lem:atoms2}
Let $n<m$ and let $\A$ be an atomic $\CA_n$, 
$\A\subseteq_c\Nr_n\C$
for some $\C\in\CA_m$.  For all $x\in\C\setminus\set0$ and all $i_0, \ldots i_{n-1} < m$ there is $a\in\At(\A)$ such that
${\sf s}_{i_0\ldots i_{n-1}}a\;.\; x\neq 0$.
\end{lemma}
\begin{proof}
We can assume, see definition  \ref{def:string}, 
that ${\sf s}_{i_0,\ldots i_{n-1}}$ consists only of substitutions, since ${\sf c}_{m}\ldots {\sf c}_{m-1}\ldots 
{\sf c}_nx=x$ 
for every $x\in \A$.We have ${\sf s}^i_j$ is a
completely additive operator (any $i, j$), hence ${\sf s}_{i_0,\ldots i_{\mu-1}}$ 
is too  (see definition~\ref{def:string}).
So $\sum\set{{\sf s}_{i_0\ldots i_{n-1}}a:a\in\At(\A)}={\sf s}_{i_0\ldots i_{n-1}}
\sum\At(\A)={\sf s}_{i_0\ldots i_{n-1}}1=1$,
for any $i_0,\ldots i_{n-1}<n$.  Let $x\in\C\setminus\set0$.  It is impossible
that ${\sf s}_{i_0\ldots i_{n-1}}\;.\;x=0$ for all $a\in\At(\c A)$ because this would
imply that $1-x$ was an upper bound for $\set{{\sf s}_{i_0\ldots i_{n-1}}a:
a\in\At(\A)}$, contradicting $\sum\set{{\sf s}_{i_0\ldots i_{n-1}}a :a\in\At(\c A)}=1$.
\end{proof}

We now prove two Theorems relating neat embeddings
to the games we defined:

\begin{theorem}\label{thm:n}
Let $n<m$, and let $\A$ be a $\CA_m$.  
If $\A\in{\bf S_c}\Nr_{n}\CA_m, $
then \pe\ has a \ws\ in $F^m(\At\A)$. In particular if $\A$ is $CR$ then \pe has a \ws in $F^{\omega}(\At\A)$
\end{theorem}
\begin{proof}
If $\A\subseteq\Nr_n\C$ for some $\C\in\CA_m$ then \pe\ always
plays hypernetworks $N$ with $\nodes(N)\subseteq n$ such that
$\widehat N\neq 0$. In more detail, in the initial round , let $\forall$ play $a\in \At \cal A$.
$\exists$ play a network $N$ with $N(0, \ldots n-1)=a$. Then $\widehat N=a\neq 0$.
At a later stage suppose $\forall$ plays the cylindrifier move 
$(N, \langle f_0, \ldots f_{\mu-2}\rangle, k, b, l)$ 
by picking a
previously played hypernetwork $N$ and $f_i\in \nodes(N), \;l<\mu,  k\notin \{f_i: i<n-2\}$, 
and $b\leq {\sf c}_lN(f_0,\ldots  f_{i-1}, x, f_{n-2})$.
Let $\bar a=\langle f_0\ldots f_{l-1}, k\ldots f_{n-2}\rangle.$
Then ${\sf c}_k\widehat N\cdot {\sf s}_{\bar a}b\neq 0$.
Then there is a network  $M$ such that
$\widehat{M}.\widehat{{\sf c}_kN}\cdot {\sf s}_{\bar a}b\neq 0$. Hence 
$M(f_0,\dots  k, f_{n-2})=b.$
\end{proof}

\begin{theorem}\label{thm:RaC}
Let $\alpha$ be a countable 
$\CA_n$ atom structure.  If \pe\ has a \ws\ in $H_n(\alpha),$ then
there is a representable cylindric algebra $\C$ of
dimension $\omega$ such that $\Nr_n\c C$ is atomic 
and $\At \Nr_n\C\cong\alpha$.
\end{theorem}

\begin{proof} 
We shall construct a generalized atomic weak set algebra of dimension $\omega$ such that the atom 
structure of its full neat reduct is isomorphic to
the given atom structure. Suppose \pe\ has a \ws\ in $H_n(\alpha)$. Fix some $a\in\alpha$. We can define a
nested sequence $N_0\subseteq N_1\ldots$ of hypernetworks
where $N_0$ is \pe's response to the initial \pa-move $a$, requiring that
\begin{enumerate}
\item If $N_r$ is in the sequence and 
and $b\leq {\sf c}_lN_r(\langle f_0, f_{n-2}\rangle\ldots , x, f_{n-2})$.  
then there is $s\geq r$ and $d\in\nodes(N_s)$ such 
that $N_s(f_0, f_{i-1}, d, f_{n-2})=b$.
\item If $N_r$ is in the sequence and $\theta$ is any partial
isomorphism of $N_r$ then there is $s\geq r$ and a
partial isomorphism $\theta^+$ of $N_s$ extending $\theta$ such that
$\rng(\theta^+)\supseteq\nodes(N_r)$.
\end{enumerate}
We can schedule these requirements
to extend so that eventually, every requirement gets dealt with.
If we are required to find $k$ and $N_{r+1}\supset N_r$
such that 
$N_{r+1}(f_0, k, f_{n-2})=b$ then let $k\in \omega\setminus \nodes(N_r)$
where $k$ is the least possible for definiteness, 
and let $N_{r+1}$ be \pe's response using her \ws, 
to the \pa move $N_r, (f_0,\ldots f_{n-1}), k, b, l).$
For an extension of type 2, let $\tau$ be a partial isomorphism of $N_r$
and let $\theta$ be any finite surjection onto a partial isomorphism of $N_r$ such that 
$dom(\theta)\cap nodes(N_r)= dom\tau$. \pe's response to \pa's move $(N_r, \theta)$ is necessarily 
$N\theta.$ Let $N_{r+1}$ be her response, using her wining strategy, to the subsequent \pa 
move $(N_r, N_r\theta).$ 

Now let $N_a$ be the limit of this sequence.
This limit is well-defined since the hypernetworks are nested.  

Let $\theta$ be any finite partial isomorphism of $N_a$ and let $X$ be
any finite subset of $\nodes(N_a)$.  Since $\theta, X$ are finite, there is
$i<\omega$ such that $\nodes(N_i)\supseteq X\cup\dom(\theta)$. There
is a bijection $\theta^+\supseteq\theta$ onto $\nodes(N_i)$ and $j\geq
i$ such that $N_j\supseteq N_i, N_i\theta^+$.  Then $\theta^+$ is a
partial isomorphism of $N_j$ and $\rng(\theta^+)=\nodes(N_i)\supseteq
X$.  Hence, if $\theta$ is any finite partial isomorphism of $N_a$ and
$X$ is any finite subset of $\nodes(N_a)$ then
\begin{equation}\label{eq:theta}
\exists \mbox{ a partial isomorphism $\theta^+\supseteq \theta$ of $N_a$
 where $\rng(\theta^+)\supseteq X$}
\end{equation}
and by considering its inverse we can extend a partial isomorphism so
as to include an arbitrary finite subset of $\nodes(N_a)$ within its
domain.
Let $L$ be the signature with one $\mu$ -ary predicate symbol ($b$) for
each $b\in\alpha$, and one $k$-ary predicate symbol ($\lambda$) for
each $k$-ary hyperlabel $\lambda$.  

 
For fixed $f_a\in\;^\omega\!\nodes(N_a)$, let
$U_a=\set{f\in\;^\omega\!\nodes(N_a):\set{i<\omega:g(i)\neq
f_a(i)}\mbox{ is finite}}$.
Notice that $U_a$ is weak unit ( a set of sequences agrreing cofinitely with a fixed one)


We can make $U_a$ into the base of an $L$ relativized structure $\c N_a$.
Satisfiability for $L$ formluas at assignments $f\in U_a$ is defined the usual Tarskian way.

For $b\in\alpha,\;
l_0, \ldots l_{\mu-1}, i_0 \ldots, i_{k-1}<\omega$, \/ $k$-ary hyperlabels $\lambda$,
and all $L$-formulas $\phi, \psi$, let
\begin{eqnarray*}
\c N_a, f\models b(x_{l_0}\ldots  x_{n-1})&\iff&N_a(f(l_0),\ldots  f(l_{n-1}))=b\\
\c N_a, f\models\lambda(x_{i_0}, \ldots,x_{i_{k-1}})&\iff&  N_a(f(i_0), \ldots,f(i_{k-1}))=\lambda\\
\c N_a, f\models\neg\phi&\iff&\c N_a, f\not\models\phi\\
\c N_a, f\models (\phi\vee\psi)&\iff&\c N_a,  f\models\phi\mbox{ or }\c N_a, f\models\psi\\
\c N_a, f\models\exists x_i\phi&\iff& \c N_a, f[i/m]\models\phi, \mbox{ some }m\in\nodes(N_a)
\end{eqnarray*}
For any $L$-formula $\phi$, write $\phi^{\c N_a}$ for the set of assighnments satisfying it; that is
$\set{f\in\;^\omega\!\nodes(N_a): \c N_a, f\models\phi}$.  Let
$D_a = \set{\phi^{\c N_a}:\phi\mbox{ is an $L$-formula}}.$ 
Then this is the universe of the following weak set algebra 
\[\c D_a=(D_a,  \cup, \sim, {\sf D}_{ij}, {\sf C}_i)_{ i, j<\omega}\] 
then 
$\c D_a\in\RCA_\omega$. (Weak set algebras are representable).

Let $\phi(x_{i_0}, x_{i_1}, \ldots, x_{i_k})$ be an arbitrary
$L$-formula using only variables belonging to $\set{x_{i_0}, \ldots,
x_{i_k}}$.  Let $f, g\in U_a$ (some $a\in \alpha$) and suppose
is a partial isomorphism of $N_a$.  We can prove by induction over the
quantifier depth of $\phi$ and using (\ref{eq:theta}), that
\begin{equation}
\c N_a, f\models\phi\iff \c N_a,
g\models\phi\label{eq:bf}\end{equation} 

Let $\c C=\prod_{a\in
\alpha}\c \c D_a$.  Then  $\c C\in\RCA_\omega$, and $\c C$ is the desired generalized weak set algebra.
Note that unit of $\c C$ is the disjoint union of the weak spaces.
We set out to prove our claim. 
We shall show that $\alpha\cong \At\Nr_n\c C.$

An element $x$ of $\c C$ has the form
$(x_a:a\in\alpha)$, where $x_a\in\c D_a$.  For $b\in\alpha$ let
$\pi_b:\c C\to\c \c D_b$ be the projection defined by
$\pi_b(x_a:a\in\alpha) = x_b$.  Conversely, let $\iota_a:\c D_a\to \c
C$ be the embedding defined by $\iota_a(y)=(x_b:b\in\alpha)$, where
$x_a=y$ and $x_b=0$ for $b\neq a$.  Evidently $\pi_b(\iota_b(y))=y$
for $y\in\c D_b$ and $\pi_b(\iota_a(y))=0$ if $a\neq b$.

Suppose $x\in\Nr_{\mu}\c C\setminus\set0$.  Since $x\neq 0$, 
it must have a non-zero component  $\pi_a(x)\in\c D_a$, for some $a\in \alpha$.  
Say $\emptyset\neq\phi(x_{i_0}, \ldots, x_{i_k})^{\c
 D_a}= \pi_a(x)$ for some $L$-formula $\phi(x_{i_0},\ldots, x_{i_k})$.  We
 have $\phi(x_{i_0},\ldots, x_{i_k})^{\c D_a}\in\Nr_{\mu}\c D_a)$.  Pick
 $f\in \phi(x_{i_0},\ldots, x_{i_k})^{\c D_a}$ and let $b=N_a(f(0),
 f(1), \ldots f_{n-1})\in\alpha$.  We will show that 
$b(x_0, x_1, \ldots x_{n-1})^{\c D_a}\subseteq
 \phi(x_{i_0},\ldots, x_{i_k})^{\c D_a}$.  Take any $g\in
b(x_0, x_1\ldots x_{n-1})^{\c D_a}$, 
so $N_a(g(0), g(1)\ldots g(n-1))=b$.  The map $\set{(f(0),
g(0)), (f(1), g(1))\ldots (f(n-1), g(n-1))}$ is 
a partial isomorphism of $N_a$.  By
 (\ref{eq:theta}) this extends to a finite partial isomorphism
 $\theta$ of $N_a$ whose domain includes $f(i_0), \ldots, f(i_k)$. Let
 $g'\in U_a$ be defined by
\[ g'(i) =\left\{\begin{array}{ll}\theta(i)&\mbox{if }i\in\dom(\theta)\\
g(i)&\mbox{otherwise}\end{array}\right.\] By (\ref{eq:bf}), $\c N_a,
g'\models\phi(x_{i_0}, \ldots, x_{i_k})$. Observe that
$g'(0)=\theta(0)=g(0)$ and similarly $g'(n-1)=g(n-1)$, so $g$ is identical
to $g'$ over $\mu$ and it differs from $g'$ on only a finite
set of coordinates.  Since $\phi(x_{i_0}, \ldots, x_{i_k})^{\c
\ D_a}\in\Nr_{\mu}(\c C)$ we deduce $\c N_a, g \models \phi(x_{i_0}, \ldots,
x_{i_k})$, so $g\in\phi(x_{i_0}, \ldots, x_{i_k})^{\c D_a}$.  This
proves that $b(x_0, x_1\ldots x_{\mu-1})^{\c D_a}\subseteq\phi(x_{i_0},\ldots,
x_{i_k})^{\c D_a}=\pi_a(x)$, and so 
$$\iota_a(b(x_0, x_1,\ldots x_{n-1})^{\c \ D_a})\leq
\iota_a(\phi(x_{i_0},\ldots, x_{i_k})^{\c D_a})\leq x\in\c
C\setminus\set0.$$  Hence every non-zero element $x$ of $\Nr_{n}\c C$ 
is above
a an atom $\iota_a(b(x_0, x_1\ldots n_1)^{\c D_a})$ (some $a, b\in
\alpha$) of $\Nr_{n}\c C$.  So
$\Nr_{n}\c C$ is atomic and $\alpha\cong\At\Nr_{n}\c C$ --- the isomorphism
is $b \mapsto (b(x_0, x_1,\dots x_{n-1})^{\c D_a}:a\in A)$.
\end{proof}

We can use such games to show that for $n\geq 3$, there is a representable $\A\in \CA_n$ 
with atom structure $\alpha$ such that $\forall$ can win the game $F^{n+2}(\alpha)$.
However \pe\ has a \ws\ in $H_n(\alpha)$, for any $n<\omega$.
It will follow that there a countable cylindric algebra $\c A'$ such that $\c A'\equiv\c
A$ and \pe\ has a \ws\ in $H(\c A')$.
So let $K$ be any class such that $\Nr_n\CA_{\omega}\subseteq K\subseteq S_c\Nr_n\CA_{n+2}$.
$\c A'$ must belong to $\Nr_n(\RCA_\omega)$, hence $\c A'\in K$.  But $\c A\not\in K$
and $\c A\preceq\c A'$. Thus $K$ is not elementary. From this it easily follows that the class of completely representable cylindric algebras
is not elementary, and that the class $\Nr_n\CA_{n+k}$ for any $k\geq 0$ is not elementary either. 
Furthermore the constructions works for many variants of cylindric algebras
like Halmos' polyadic equality algebras and Pinter's substitution algebras.
Formally we shall prove:
\begin{theorem}\label{r} Let $3\leq n<\omega$. Then the following hold:
\begin{enumroman}
\item Any $K$ such that $\Nr_n\CA_{\omega}\subseteq K\subseteq S_c\Nr_n\CA_{n+2}$ is not elementary.
\item The inclusions $\Nr_n\CA_{\omega}\subseteq S_c\Nr_n\CA_{\omega}\subseteq S\Nr_n\CA_{\omega}$ are all proper
\end{enumroman}
\end{theorem}


\subsection*{Details of the above idea}

Fix finite $n>2$. We use a rainbow construction for cylindric algebras.
We shall construct a cylindric atom structure based on graphs.
A coloured graph is an undirected irreflexive graph $\Gamma$
such that every edge of $\Gamma$ is coloured by a unique 
edge colour and some $n-1$ tuples
have a unique colour too. $\Z$ denotes the set of integers.
Let $P$ be the set of partial order preserving functions $f:\Z\to \N$ with $|\dom(f)|\leq 2$.

The edge colours (or future atoms) are

\begin{itemize}

\item greens: $\g_i$ ($i<n-1)$, $\g_0^i$, $i\in \Z$.  

\item whites : $\w, \w_f: f\in P$
\item yellow : $\y$
\item black :  $\bb$
\item reds:  $\r_{ij}$ $(i,j\in \N)$,

\item shades of yellow : $\y_S: S\subseteq_{\omega} \N$ or $S=\N$

\end{itemize}

\begin{definition} 
Let $i\in \Z$, and let $\Gamma$ be a coloured graph  consisting of $n$ nodes
$x_0,\ldots  x_{n-2}, z$. We call $\Gamma$ an $i$ - cone if $\Gamma(x_0, z)=\g^0_i$   
and for every $1\leq j\leq n-2$ $\Gamma(x_j, z)=\g_j$, 
and no other edge of $\Gamma$
is coloured green.
$(x_0,\ldots x_{n-2})$ 
is called the centre of the cone, $z$ the apex of the cone
and $i$ the tint of the cone.
We define a class $\bold J$ consisting of coloured graphs 
with the following properties:

\begin{enumarab}

\item $\Gamma$ is a complete graph.

\item $\Gamma$ contains no triangles (called forbidden triples) 
of the following types:

\vspace{-.2in}
\begin{eqnarray}
&&\nonumber\\
(\g, \g^{'}, \g^{*}), (\g_i, \g_{i}, \w), 
&&\mbox{any }i\in n-1\;  \\
(\g^j_0, \y, \w_f)&&\mbox{unless }f\in P, i\in dom(f)\\
(\g^j_0, \g^k_0, \w_0)&&\mbox{ any } j, k\in \Z\\
\label{forb:pim}(\g^i_0, \g^j_0, \r_{kl})&&\mbox{unless } \set{(i, k), (j, l)}\mbox{ is an order-}\\
&&\mbox{ preserving partial function }\Z\to\N\nonumber\\
\label{forb: black}(\y,\y,\y), (\y,\y,\bb)\\ 
\label{forb:match}(\r_{ij}, \r_{j'k'}, \r_{i^*k^*})&&\mbox{unless }i=i^*,\; j=j'\mbox{ and }k'=k^*
\end{eqnarray}
and no other triple of atoms is forbidden.


\item If $a_0,\ldots   a_{n-2}\in \Gamma$ are distinct, and no edge $(a_i, a_j)$ $i<j<n$
is coloured green, then the sequence $(a_0, \ldots a_{n-2})$ 
is coloured a unique shade of yellow.
No other $(n-1)$ tuples are coloured shades of yellow.

\item If $D=\set{d_0,\ldots  d_{n-2}, \delta}\subseteq \Gamma$ and 
$\Gamma\upharpoonright D$ is an $i$ cone with apex $\delta$, inducing the order 
$d_0,\ldots  d_{n-2}$ on its base, and the tuple 
$(d_0,\ldots d_{n-2})$ is coloured by a unique shade 
$y_S$ then $i\in S.$

\end{enumarab}
\end{definition}
\begin{demo}{Proof} 
The proof is very simiar to Hirsch's proof for reation algebras, 
except that we lift the rainbow construction to cylindric algebras. This is highly non- trivial, 
it was first done in \cite{complete}. The idea is to use labelled graphs as atoms; that is algebras will be based on atom structures consisting of labelled 
graph. Cylindrfiers are stimulated by shades of yellow,
which has to do with the $n$ tuples, and their an additional complexity the presence of cones. 

We define a cylindric algebra of 
dimension $n$. We first specify its atom structure.
Let $$K=\{a: a \text { is a surjective map from $n$ onto some } \Gamma\in \bold J
\text { with nodes } \Gamma\subseteq \omega\}.$$
We write $\Gamma_a$ for the element of $K$ for which 
$a:\alpha\to \Gamma$ is a surjection.
Let $a, b\in K$ define the following equivalence relation: $a \sim b$ if and only if
\begin{itemize}
\item $a(i)=a(j)\text { and } b(i)=b(j)$

\item $\Gamma_a(a(i), a(j))=\Gamma_b(b(i), b(j))$ whenever defined

\item $\Gamma_a(a(k_0)\dots a(k_{n-2}))=\Gamma_b(b(k_0)\ldots b(k_{n-1}))$ whenever  
defined
\end{itemize}
Let $\mathfrak{C}$ be the set of equivalences classes. Then define
$$[a]\in E_{ij} \text { iff } a(i)=a(j)$$
$$[a]T_i[b] \text { iff }a\upharpoonright n\sim \{i\}=b\upharpoonright n\sim \{i\}.$$
This, as easily checked, defines a $\CA_n$ 
atom structure.
Let $3\leq n<\omega$. 
Let $\c C_{n}$ be the complex algebra over $\mathfrak{C}$. We will show that $\c C_{n}$ is 
not in $S_c\Nr_n\CA_{n+2}$ 
but an elementary extension of $\c A$ belongs to
$\Nr_{n}\CA_{\omega}.$
But first we translate our games to games on coloured graphs:
Let $N$ be an atomic $\c C_{n}$ network. 
Let $x,y$ be two distinct nodes occuring in the
$n$ tuple $\bar z$. $N(\bar z)$ is an atom of $\c C_{n}$ 
which defines an edge colour of 
$x,y$. Using the fact that the dimension is at least $3$, 
the edge colour depends only on $x$ and $y$
not on the other elements of 
$\bar z$ or the positions of $x$ and $y$ in $\bar z$.
Similarly $N$ defines shades of white for certain $(n-1)$ tuples.  In this way $N$ 
translates
into a coloured graph. 
This translation has an inverse. More precisely we have:
Let $\Gamma\in \bold J$ be arbitrary. Define $N_{\Gamma}$ 
whose nodes are those of $\Gamma$
as follows. For each $a_0,\ldots a_{n-1}\in \Gamma$, define 
$N_{\Gamma}(a_0,\ldots  a_{n-1})=[\alpha]$
where $\alpha: n\to \Gamma\upharpoonright \set{a_0,\ldots a_{n-1}}$ is given by 
$\alpha(i)=a_i$ for all $i<n$. 
Then, as easily checked,  $N_{\Gamma}$ is an atomic $\c C_{n}$ network.
Conversely , let $N$ be any non empty atomic $\c C_n$ network. 
Define a complete coloured graph $\Gamma_N$
whose nodes are the nodes of $N$ as follows:  
\begin{itemize}
\item For all distinct $x,y\in \Gamma_N$ and edge colours $\eta$, $\Gamma_N(x,y)=\eta$
if and only if  for some $\bar z\in ^nN$, $i,j<n$, and atom $[\alpha]$, we have
$N(\bar z)=[\alpha]$, $z_i=x$ $z_j=y$ and the edge $(\alpha(i), \alpha(j))$ 
is coloured $\eta$ in the graph $\alpha$.

\item For all $x_0,\ldots x_{n-2}\in {}^{n-1}\Gamma_N$ and all yellows $\y_S$, 
$\Gamma_N(x_0,\ldots x_{n-2})= \y_S$ if and only if 
for some $\bar z$ in $^nN$, $i_0, i_{n-2}<n$
and some atom $[\alpha]$, we have
$N(\bar z)=[\alpha]$, $z_{i_j}=x_j$ for each $j<n-1$ and the $n-1$ tuple 
$\langle \alpha(i_0),\ldots \alpha(i_{n-2})\rangle$ is coloured 
$\y_S.$ Then $\Gamma_N$ is well defined and is in $\bold J$.
\end{itemize}
The following is then, though tedious and long,  easy to  check:
For any $\Gamma\in \bold J$, we have  $\Gamma_{N_{\Gamma}}=\Gamma$, 
and for any $\c C_n$ network
$N$ $N_{{\Gamma}_N}=N.$
This translation makes the following equivalent formulation of the 
games $F^m(\At\c C_n),$ originally defined on networks. 
The new  game builds a nested sequence $\Gamma_0\subseteq \Gamma_1\subseteq \ldots $.
of coloured graphs. 
Let us start with the game $F^m(\alpha)$. 
\pa\ picks a graph $\Gamma_0\in \bold J$ with $\Gamma_0\subseteq m$ and 
$|\Gamma_0|=m$. $\exists$ makes no response
to this move. In a subsequent round, let the last graph built be $\Gamma_i$.
$\forall$ picks 
\begin{itemize}
\item a graph $\Phi\in \bold J$ with $|\Phi|=m$
\item a single node $k\in \Phi$
\item a coloured garph embedding $\theta:\Phi\sim \{k\}\to \Gamma_i$
Let $F=\phi\smallsetminus \{k\}$. Then $F$ is called a face. 
\pe\ must respond by amalgamating
$\Gamma_i$ and $\Phi$ with the embedding $\theta$. In other words she has to define a 
graph $\Gamma_{i+1}\in C$ and embeddings $\lambda:\Gamma_i\to \Gamma_{i+1}$
$\mu:\phi \to \Gamma_{i+1}$, such that $\lambda\circ \theta=\mu\upharpoonright F.$
\end{itemize} 
Let us consider the possibilities. There may be already a point $z\in \Gamma_i$ such that
the map $(k\to z)$ is an isomorphism over $F$.
In this case \ \pe does not need to extend the 
graph $\Gamma_i$, she can simply let $\Gamma_{i+1}=\Gamma_i$
$\lambda=Id_{\Gamma_i}$, and $\mu\upharpoonright F=Id_F$, $\mu(\alpha)=z$.
Otherwise, without loss of generality, 
let $F\subseteq \Gamma_i$, $k\notin \Gamma_i$.
Let ${\Gamma_i}^*$ be the colored graph with nodes $\nodes(\Gamma_i)\cup\{k\}$,
whose edges are the combined edges of $\Gamma_i$ and $\Phi$, 
such that for any $n-1$ tuple $\bar x$ of nodes of 
${\Gamma_i}^*$, the color ${\Gamma_i}^*(\bar x)$ is
\begin{itemize}
\item $\Gamma_i(\bar x)$ if the nodes of $x$ 
all lie in $\Gamma$ and $\Gamma_i(\bar x)$ is defined 
\item $\phi(\bar x)$ if the nodes of $\bar x$ all lie in 
$\phi$ and $\phi(\bar x)$ is defined
\item undefined, otherwise.
\end{itemize}
\pe\ has to complete the labeling of $\Gamma_i^*$ by adding 
all missing edges, colouring each edge $(\beta, k)$
for $\beta\in \Gamma_i\sim\Phi$ and then choosing a shade of 
white for every $n-1$ tuple $\bar a$ 
of distinct elements of ${\Gamma_i}^*$
not wholly contained in $\Gamma_i$ nor $\Phi$, 
if non of the edges in $\bar a$ is coloured green.
She must do this on such a way that the resulting graph belongs to $\bold J$.
If she survives each round, \pe\ has won the play.
Notice that \pe\ has a \ws in the in $F^m(\At(\C_n))$ if and only if 
and only if she has a wininng strategy in the 
graph games defined above. Th is tedious and rather long to verify but basically routine. 
We now show that the rainbow algebra $\c A$ (definition above) 
is not in ${\bf S_c}\Nr_{\mu}\CA_{n+2}$.

For that we show $\forall$ can win the game $F^{n+2}(\At(\C_n))$. In his zeroth move, $\forall$ plays a graph $\Gamma \in \bold J$ with
nodes $0, 1,\ldots, n-1$ and such that $\Gamma(i, j) = \w (i < j <
n-1), \Gamma(i, n-1) = \g_i ( i = 1,\ldots, n), \Gamma(0, n-1) =
\g^0_0$, and $ \Gamma(0, 1,\ldots, n-2) = \y_\omega $. This is a $0$-cone
with base $\{0,\ldots , n-2\}$. In the following moves, $\forall$
repeatedly chooses the face $(0, 1,\ldots, n-2)$ and demands a node (possibly used before)
$\alpha$ with $\Phi(i,\alpha) = \g_i (i = 1,\ldots,  n-2)$ and $\Phi(0, \alpha) = \g^\alpha_0$,
in the graph notation -- i.e., an $\alpha$-cone on the same base.
$\exists$, among other things, has to colour all the edges
connecting nodes. The idea is that by the rules of the game 
only permissible colours would be red. Using this, $\forall$ can force a
win eventually for else we are led to a a decreasing sequence in $\N$.
In more detail,
In the initial round $\forall$ plays a graph $\Gamma$ with nodes $0,1,\ldots n-1$ such that $\Gamma(i,j)=\w$ for $i<j<n-1$ 
and $\Gamma(i,n-1)=\g_i$
$(i=1, \ldots n-2)$, $\Gamma(0,n-1)=\g_0^0$ and $\Gamma(0,1\ldots n-2)=\y_{N}$.
$\exists$ must play a graph with $\Gamma_1(0,\ldots n-1)=\g_0$.
In the following move $\forall$ chooses the face $(0,\ldots n-2)$ and demands a node $n$ 
with $\Gamma_2(i,n)=\g_i$ and $\Gamma_2(0,n)=\g_0^{-1}$
$\exists$ must choose a label for the edge $(n,n-1)$ of $\Gamma_2$. It must be a red atom $r_{mn}$. Since $-1<0$ we have $m<n$.
In the next move $\forall$ plays the face $(0, \ldots n-2)$ and demands a node $n+1$ such that  $\Gamma_3(i,n+1)=\g_i^{-2}$.
Then $\Gamma_3(n+1,n)$ $\Gamma_3(n+1,n-1)$ both being red, the indices must match.
$\Gamma_3(n+1,n)=r_{ln}$ and $\Gamma_3(n+1, n-1)=r_{lm}$ with $l<m$.
In the next round $\forall$ plays $(0,1\ldots n-2)$ and reuses the node $n-2$ such that $\Gamma_4(0,n-2)=\g_0^{-3}$.
This time we have $\Gamma_4(n,n-1)=\r_{jl}$ for some $j<l\in N$.
Continuing in this manner leads to a decreasing sequence in $\N$.

Recall from definition~\ref{def:games} that $H_k(\alpha)$ 
is the hypernetwork game with $k$ rounds.
The translation of the games $H$ and $H_k$ to graphs is as follows.

\begin{itemize}
\item
Fix some hyperlabel $\lambda_0$.  $H_{\mu}(\alpha)$ is  a 
game the play of which consists of a sequence of
$\lambda_0$-neat hypernetworks 
$N_0, N_1,\ldots$ where $\nodes(N_i)$
is a finite subset of $\omega$, for each $i<\omega$.  
A neat hypernetwork, now, is a pair $(\Gamma, N^h)$ with $\Gamma$ a coloured 
graph. $N^h$ are the hyperlabels, these we forget for a while and identify the pair $(\Gamma, N^h)$ with $\Gamma.$
Hyperedges will be dealt with later, and we shall see that they are easior to deal with.
$\forall$ picks a graph $\Gamma_0\in \bold J$ with $\Gamma_0\subseteq_{\omega} \omega$ and 
here we do not require that $|\Gamma_0|=n$. 
$\exists$ make no response
to this move. In a subsequent round, let the last graph built be $\Gamma_i$.
$\forall$ picks 
\begin{itemize}
\item a graph $\Phi\in \bold J$ with $|\phi|=|\Gamma|$
\item a single node $k\in \phi$
\item a colored garph embedding $\theta:\Phi\sim\{k\}\to \Gamma_i$
Let $F=\phi\sim \{k\}$. Then $F$ is called a face. 
\pe must respond by amalgamating
$\Gamma_i$ and $\phi$ with the embedding $\theta$ as before. 
In other words she has to define a 
graph $\Gamma_{i+1}\in C$ and embeddings $\lambda:\Gamma_i\to \Gamma_{i+1}$
$\mu:\phi \to \Gamma_{i+1}$, such that $\lambda\circ \theta=\mu\upharpoonright F.$

\end{itemize} 
Now we may write $N_{\Gamma}$ or simply $N$ instead of $\Gamma$, but in all cases we are dealing with {\it coloured graphs} 
that is {\it the translation} of networks. That is when we writ $N$ then, $N$ will be viewed as a coloured graph.
Alternatively, \pa\ can play a \emph{transformation move} by picking a
previously played graph $N$ and a partial, finite surjection
$\theta:\omega\to\nodes(N)$, this move is denoted $(N, \theta)$.  \pe\
must respond with $N\theta$.  Finally, \pa\ can play an
\emph{amalgamation move} by picking previously played graphs
$M, N$ such that $M\equiv^{\nodes(M)\cap\nodes(N)}N$ and
$\nodes(M)\cap\nodes(N)\neq \emptyset$  
This move is denoted $(M,
N)$.  To make a legal response, \pe\ must play a $\lambda_0$-neat
hypernetwork $L$ extending $M$ and $N$, where
$\nodes(L)=\nodes(M)\cup\nodes(N)$.
Again, \pe\ wins $H(\alpha)$ if she responds legally in each of the
$\omega$ rounds, otherwise \pa\ wins.

It will simplify things a bit if we alter the rules of the game
$H(\alpha)$ slightly.
We impose certain restrictions on \pa.
First a piece of notation. A triangle move $\Gamma$ in the graph game will be denoted
by $(\Gamma, F, k, b, l),$  $\Phi=F\cup\{k\}$ and 
$b\leq {\sf c}_lN_{\Gamma}(f_0,\dots  f_{n-1})$.  
\begin{itemize}
\item
\pa\ is only allowed to play a triangle move $(\Gamma, F, k, b, l)$ if
there does not exist $l\in\nodes(\Gamma)$ such that 
$N_{\Gamma)}(f_0,\ldots f_i,l\ldots f_{n-2})=b.$ 

\item
\pa\ is only allowed to play transformation moves $(N, \theta)$ if $\theta$ is injective.

\item
\pa\ is only allowed to play an amalgamation move $(M, N)$ if for all
$m\in\nodes(M)\setminus\nodes(N)$ and all
$n\in\nodes(N)\setminus\nodes(M)$ the map $\set{(m, n)}\cup\set{(x,
x):x\in\nodes(M)\cap\nodes(N)}$ is not a partial isomorphism.  I.e. he
can only play $(M, N)$ if the amalgamated part is `as large as
possible'.
\end{itemize}  If, as a result of these restrictions, \pa\ cannot move at
some stage then he loses and the game halts.

It is easy to check that \pa\ has a \ws\ in $H(\alpha)$ iff he has a
\ws\ with these restrictions to his moves.  Also, if \pa\ plays with
these restrictions to his moves, if \pe\ has a \ws\ then she has a
\ws\ which only directs her to play strict hypernetworks.  The same
holds when we consider $H_n(\alpha)$.  
We will assume that \pa\ plays 
according to these restrictions.
\end{itemize}

Now we deal with hypernetworks.
Notice that in this game $\forall$ is not allowed to reuse nodes and so his winning strategy above does not work.
In a play of $N(\alpha)$ $\exists$ is required to play $\lambda_0$ neat hypernetworks, so she has no choice about the 
hyperedges for short edges, these are labelled by $\lambda_0$. In response to a cylindrifier move 
$(N, F, k,l)$ all long hyperedges not incident with $k$ necessarily keep the hyperlabel they had in $N$. 
All long hyperedges incident with $k$ in $N$
are given unique hyperlabels not occuring as the hyperlabel of any other hyperedge in $M$. 
We assume, without loss of generality, that we have infinite supply of hyperlabels of all finite arities so this is possible.
In reponse to an amalgamation move $(M,N)$ all long hyperedges whose range is contained in $\nodes(M)$ 
have hyperlabel determined by $M$, and those whose range is contained in nodes $N$ have hyperlabel determined
by $N$. If $\bar{x}$ is a long hyperedge of $\exists$'s response $L$ where 
$\rng(\bar{x})\nsubseteq \nodes(M)$, $\nodes(N)$ then $\bar{x}$ 
is given 
a new hyperabel, not used in any previously played hypernetwork and not used within $L$ as the label of any hyperedge other than $\bar{x}$.
This completeles her strategy for labelling hyperedges.

Now we give $\exists$'s strategy for edge labelling. We need some notation and terminology. 
Every irreflexive edge of any hypernetwork has an owner $\forall$ or $\exists$ namely the one who played this edge.
We call such edges $\forall$ edges or $\exists$ edges. Each long hyperedge $\bar{x}$ in a hypernetwork $N$ 
occuring in the play has an envelope $v_N(\bar{x})$ to be defined shortly. 
In the initial round of $\forall$ plays $a\in \alpha$ and $\exists$ plays $N_0$ 
then all irreflexive edges of $N_0$ belongs to $\forall$. There are no long hyperdeges in $N_0$. If in a later move, 
$\forall$ plays the transformation move $(N,\theta)$ 
and $\exists$ responds with $N\theta$ then owners and envelopes are inherited in the obvious way. 
If $\forall$ plays a cylindrifier move $(N, F, k, b,l)$ and $\exists$ responds with $M$ then the owner
in $M$ of an edge not incident with $k$ is the same as it was in $N$
and the envelope in $M$ of a long hyperedge not incident with $k$ is the same as that it was in $N$. 
The edges $(f,k), (k,f)$ belong to $\forall$ in $M$ all edges $(l,k)(k,l)$
for $l\in nodes(N)\sim F$ belong to $\exists$ in $M$.
if $\bar{x}$ is any long hyperedge of $M$ with $k\in \rng(\bar{x}$, then $v_M(\bar{x})=\nodes(M)$.
If $\forall$ plays the amalgamation move $(M,N)$ and $\exists$ responds with $L$ 
then for $m\neq n\in nodes(L)$ the owner in $L$ of a edge $(m,n)$ is $\forall$ if it belongs to
$\forall$ in either $M$ or $N$, in all other cases it belongs to $\exists$ in $L$. If $\bar{x}$ is a long hyperedge of $L$
then $v_L(\bar{x}=v_M(x)$ if $range(x)\subseteq nodes(M)$ $v_L(x)=v_N(x)$ and  $v_L(x)=\nodes(M)$ otherwise.
This completes the definition of owners and envelopes.
By induction on the number of rounds one can show

\begin{athm}{Claim} Let $M, N$ ocur in a play of $H_n(\alpha)$ in which $\exists$ uses default labelling
for hyperedges. Let $\bar{x}$ be a long hyperedge of $M$ and let $\bar{y}$ be a long hyperedge of $N$.
\begin{enumarab}
\item For any hyperedge $\bar{x}'$ with $\rng(\bar{x}')\subseteq _M(\bar{x})$, if $M(\bar{x}')=M(\bar{x})$
then $\bar{x}'=\bar{x}$.
\item if $\bar{x}$ is a long hyperedge of $M$ and $\bar{y}$ is a long hyperedge of $N$, and $M(\bar{x})=N(\bar{y})$ 
then there is a local isomorphism $\theta: v_M(\bar{x})\to v_N(\bar{y})$ such that
$\theta(x_i)=y_i$ for all $i<|x|$.
\item For any $x\in \nodes (M)\sim v_M(\bar{x})$ and $S\subseteq v_M(\bar{x})$, if $(x,s)$ belong to $\forall$ in $M$
for all $s\in S$, then $|S|\leq 2$.
\end{enumarab}
\end{athm}  
Now we define $\exists$'s straetegy for choosing the labels for edges and yelows colours for $n-1$ hyperedges.
Let $N_0, N_1,\ldots N_r$ be the start of a play of $H_k(\alpha)$ just before round $r+1$. 
$\exists$ computes partial functions $\rho_s:Z\to N$, for $s\leq r$. Inductively
for $s\leq r$ suppose

I. If $N_s(x,y)$ is green or yellow then $(x,y)$ belongs to $\forall$ in $N_s$.

II. $\rho_0\subseteq \ldots \rho_r$

III. $\dom(\rho_s)=\{i\in Z: \exists t\leq s, x,y\in \nodes(N_t), N_t(x,y)=\g_0^i\}$

IV. $\rho_s$ is order preserving: if $i<j$ then $\rho_s(i)<\rho_s(j)$. The range
of $\rho_s$ is widely spaced: if $i<j\in \dom\rho_s$ then $\rho_s(i)-\rho_s(j)\geq  3^{n-r}$, where $n-r$ 
is the number of rounds remaining in the game.

V. For $u,v,x,y\in \nodes(N_s)$, if $N_s(u,v)=\r_{\mu,\delta}$, $N_s(x,u)=\g_0^i$, $N_s(x,v)=\g_0^j$
$N_s(y,u)=N_s(y,v)=\y$ then

(a) if $N_s(x,y)\neq \w_f$ then $\rho_s(i)=\mu$ and $\rho_s(j)=\delta$

(b) If $N_s(x,y)=\w_f$ for some $f\in P$ , the $\mu=f(i)$ , $\delta=f(j)$.

VI. $N_s$ is a strict $\lambda_0$ neat hypernetwork.

To start with if $\forall$ plays $a$ in the intial round then $\nodes(N_0)=\{0,1,\ldots n-1\}$, the 
hyperedge labelling is defined by $N_0(0,1,\ldots n)=a$.

In reponse to a traingle move $(N_s, F, k, \g_0^p,l)$ by $\forall$, for some $s\leq r$ and some $p\in Z$, 
$\exists$ must extend $\rho_r$ to $\rho_{r+1}$ so that $p\in \dom(\rho_{r+1})$
and the gap between elements of its range is at least $3^{n-r-1}$. Inductively $\rho_r$ is order preserving and the gap betwen its elements is 
at least $3^{n-r}$, so this can be maintained in a further round.  
If $\forall$ chooses non green atoms, green atoms with the same suffix, or green atom whose suffices 
already belong to $\rho_r$, there would be fewer 
elements to add to the domain of $\rho_{r+1}$, which makes it easy for $\exists$ to define $\rho_{r+1}$.
Tis establishe properties $II-IV$ for round $r+1$.

Let us assume that  $\forall$ play the triangle move $(N_s, F, k, a, l)$ in round $r+1$. 
$\exists$ has to choose labels for $\{(x,k), (k,x)\}$ 
$x\in \nodes(N_s)\sim F$. She chooses labels for the edges
$(x,k)$ one at a time and then determines the reverse edges $(k,x)$ uniquely.
Property $I$ is clear since in all cases the only atoms $\exists$ chooses are white, black or red.

Now we distinguish between two case.

If $x$ and $k$ are both apexes of cones on $F$, then $\exists$ has no choice but to pick a red atom. The colour she chooses is uniquely defined.
Other wise, this is not the case, so for some $i<n-1$ there is no $f\in F$ such 
that $N_s(\beta, f), N_s(f,x)$ are both coloured $\g_i$ or if $i=0$, they are coloured
$\g_0^l$ and $\g_0^{l'}$ for some $l$ and $l'$.

Now we distinguish between several subcases:

\begin{enumarab}

\item Suppose that it is not the case that there exists distinct $i,j\in F$ $N_s(x,i)$ and $a$ are both green
and $N_s(x,j)$ and $a$ are both green.
Let $S=\{p\in Z: (N_s(x,i)=\g_0^p\land a =\y)\lor N_s(x,i)=\y \land a=\g_0^p)\}$
Then $|S|\leq 2$. $\exists$ lets $N_{s+1}(x,k)=\w_f$ for some $f$ with $\dom(f)=S$.

Suppose that there exists $i,j$ distinct in $F$ such that $N_s(i,j)=\r_{\beta, \mu}$, $N_s(x,i)=\g_p$, $N_s(x,j)=\g_q$ for some $p,q\in Z$. 
By property $(IV)$ $f=\{(p,\beta), (q,\mu)\}$ is order preserving.
$\exists$ lets $N_{s+1}(x,k)=\w_f$ in this case.

In all other cases: either there are $i,j\in F$ such that $N_s(i,j)$ is not red, or if it is 
then it is not the case that $N_s(x,i)$ $N_s(x,j)$ are both green, and it is not the case that $N_s(x,i)=N_s(x,j)=\y$, 
$\exists$ lets $f:S\to N$ an arbitrary order preserving function.
The only forbidden triangles involving $\w_f$ are avoided. Since $\exists$ does not change green or yellow atoms to label new edges
and $N_{r+1}(x,k)=\w_f$, all traingles involving the new edge $(x,k)$ are consistent in $N_{r+1}$.
Clearly propery $VI$ holds after round $r+1$.

\item Else if it is not the case that $N_s(x,i)=a$ and not the case that $N_s(x,j)=y$, $\exists$ lets $N_r(x,k)=\bb$. 
Property $V$ is not applicable in this case.
The only forbidden triple involving the atom $\bb$ is avoided, so all triangles 
$(x,y,k)$ are consistent in $N_{r+1}$ and property
$VI$ holds after round $r+1$lets $N_{r+1}(x,k)=a$.

\item If neither case above applies, then for all distinct $i,j\in F$ either $N_s(x,i)=\g_p$, $a=\g_q$ (for some $p$ and $q$)
and $N_s(x,j)=N_s(x,i)=a=\y$ and $N_s(x,j)=\g_p$, $b=\g_q$.
Assume the first alternative.  There are two subcases.
\begin{enumroman}
\item $N_s(i,j)\neq w_f$ for all $f\in P$. $\exists$ lets $\mu=\rho_{r+1}(p)$, $\delta=\rho_{r+1}(q)$,
maintaining property $Va$. The only forbidden triples of atoms involving $r_{\mu,\delta}$ are avoided. The triple
of atoms form a traingle $(x,y,k)$ will not be forbidden since the only green edge incident 
with $k$ is $(i,k)$ and since $\rho_{r+1}$ is order preserving. To check forbidden triple (16)
suppose $N_s(x,y), N_{r+1}(k)$ are both red for some $y\in \nodes(N_r).$
We have $y\notin \{i,j\}$ so $\exists$ chose the red label $N_{r+1}(y,k)$. By her strategy we have
$N_s(i,y)=\g_t$ and $N_s(j,y)=\y$ . By property $Va$ for $N_{r+1}$ we have 
$N_{r+1}(x,y)=r_{\rho_{r+1}(p)\rho_{r+1}(t)}$ and $N_{r=1}(y,k)=r$ 
The property $VI$ holds for $N_{r+1}$

\item $N_s(i,j)=\w_f$ for some $f\in F$. By consistency of $N_s$ and forbidden triple (14), we have $p\in dom(f)$and since $\forall$'s move 
we have $q\in \dom(f)$. $\exists$ lets $\mu=f(p)$ $\delta=f(q)$
maintaining property
$V$ for round $r+1$. As above, the only forbidden triples of atoms involving $r_{\mu,\delta}$ are (15) and (16).
Since $f$ is order preserving and since the only green edge incident
with $k$ is $(i,k)$ in $N_{r+1}$  tariangles involving the new edge $(x,k)$ cannot give a forbidden triple. 
For forbidden triple (16) let $y\in \nodes(N_s)$ and suppose
$N_{r+1}(x,y)$$ N_{r+1}(y,k)$ are both red. As above, by her strategy we must have $N_s(y,i)=\g_t$ for some $t$ and $N_s(y,j)=\y$.
By consistency of $N_s$ we have $t\in \dom(f)$ and the current part of her strategy she lets $N_{r+1}(y,k)=\r_{f(t),f(q)}.$ 
By property $Vb$ for $N_s$
we have $N_{r+1}(x,y)=\r_{f(p), f(t)}$. So the triple of atoms from the triangle $(x,y,k)$ is not forbidden by (16). This establishes propery $(VI)$ for
$N_{r+1}$
\end{enumroman}
\end{enumarab}
We have finished with triangle moves. Now we move to amalgamation moves.
Although our hypernetworks are all strict, it is not necessarily the case that hyperlabels label unique hyperedges - amalgmation moves can force that 
the same hyperlabel can label more than one hyperedge. However, within the envelope of a hyperdege $\bar{x}$, the hyperlabel $N(\bar{x})$ is unique.

We consider an amalgamation move $(N_s,N_t)$ chosen by $\forall$ in round $r+1$.
$\exists$ has to choose a label for each edge $(i,j)$ where $i\in \nodes(N_s)\sim \nodes(N_t)$ and $j\in \nodes(N_t)\sim (N_s)$.
This determines the label for the reverse edge.
Let $\bar{x}$ enumerate $\nodes(N_s)\cap \nodes(N_t)$
If $\bar{x}$ is short, then there are at most two nodes in the intersection 
and this case is similar to the trangle move. if not, that is if $\bar{x}$ is long in $N_s$, then by the claim 
there is a partial isomorphism $\theta: v_{N_s}(\bar{x})\to v_{N_t}(\bar{x})$ fixing
$\bar{x}$. We can assume that $v_{N_s}(\bar{x})=\nodes(N_s)\cap \nodes N_t)=rng(\bar{x})=v_{N_t}(bar{x})$.
It remains to label the edges $(i,j)\in N_{r+1}$ where $i\in \nodes(N_s)\sim \nodes N_t$ and $j\in \nodes(N_t)\sim \nodes(N_s)$. 
Her startegy is similar to the triangle move. if $i$ and $j$ are tints of the same cone she choose a red. If not she 
chooses  white atom if possible, else the black atom if possible, otherwise a red atom.
She never chooses a green atom, she lets $\rho_{r+1}=\rho_r$ and properies $II$, $III$, $IV$ remain true
in round $r+1$.

\begin{enumarab}  

\item There is no $x\in \nodes(N_s)\cap \nodes(N_t)$ such that $N_s(i,x)$ and $N_t(x,j)$ are both green. 
If there are nodes $u,v\in \nodes(N_s)\cap \nodes(N_t)$ such that
$N_s(u,v)=\r_{\beta,\mu}$, $N_s(i,u)=\g_p$, $N_s(i,v)=\g_q$, $N_t(u,j)=N_t(v,j)=y$ for some $\beta, \mu\in N$, $p,q\in Z$
or the roles of $i,j$ are swapped, she lets $f=\{(p,\beta), (q,\mu)\}$ and sets $N_{r+1}(i,j)=w_f$.
Since all edges labelled by green or yellow atoms belong to $\forall$, we can apply the above claim to 
show that the points $u,v$ are unique so $f$ is well defined.
This s also true if $\bar{x}$ is short, since in this csase there are only
two nodes
in $\nodes(N_s)\cap \nodes(N_t)$.

If there are no such $u,v$ as described then let $S=\{p\in Z: \exists y\in nodes(N_s)\cap Nodes(N_t), (N_s(i,y)=g_p\land N_t(y,j)=y)\lor 
(N_s(i,y)=y\land N_t(y,j=g_p)\}$. Then $|S|\leq 2$. 
Let $f$ be any order preserving function and $\exists$ let $N_{r+1}=w_f$.
Property $(VI)$ holds for $N_{r+1}$ as for traingle moves.

\item Otherwse if there is no such $x$, then she lets $N_r(i,j)=b$. As with traingle moves all properties are maintained.

\item Otherwise, there are $x,y\in \nodes(N_s)\cap \nodes(N_t)$ such that $N_s(i,x)=g_k$, $N_s(x,j)=g_l$ 
for some $k,l\in N$ and $N_s(i,y)=N_t(y,j)=y$. By claim 3 $x,y$ are unique. She labels $(i,j)$in $N_r$ with a red atom $r_{\beta,\mu}$ where
\begin{enumroman}
\item If $N_s(x,y)\neq \w_f$ for all $f\in P$, then $\beta=\rho_{r+1}(k)$, $\mu=\rho_{r+1}(l)$.
This maintains property $Va$.
\item Otherwise $N_s(x,y)=\w_f$ for some $f\in F$ and $\beta=f(k)$ $\mu=f(l)$.

\end{enumroman}
\end{enumarab}

Now we turn to coluring of $n-$ tuples. For each tuple $\bar{a}=a_0,\ldots a_{n-2}\in (\Gamma)^{n-1}$ with no edge 
$(a_i, a_j)$ coloured green, $\exists$ colours $\bar{a}$ by $y_S$, where
$S=\{i\in N: \text { there is an $i$ cone in $\Gamma$ with base $a_0,\ldots a_{n-2}$}\}$
We need to check that such labeling works.

Let us check that $(n-1)$ tuples are labeled correctly, by yellow colours. Let $D$ be  set of $n$ nodes, and supose that $\Gamma\upharpoonright D$ 
is an $i$ cone with apex
$\delta$ and base $\{d_0,\ldots d_{n-2}\}$ , and that the tuple $(d_0,\ldots d_{n-2})$ is labelled $y_S$ in $\Gamma$. 
We need to show that $i\in S$. If $D\subseteq \Gamma$, then inductively the graph
$\Gamma$ constructed so far is in $\bold J$, and therefore
$i\in S$. If $D\subseteq \Phi$ then as $\forall$ chose $\Phi$ in $\bold J$ we get also $i\in S$. If neither holds, then $D$ contains $\alpha$ 
and also some
$\beta\in \Gamma\sim \Phi$. $\exists$ chose the colour $\Gamma^+(\alpha,\beta)$ and her strategy ensures her that it is green. 
Hence neither $\alpha$ or $\beta$ can be the apex of the cone $\Gamma^+\upharpoonright D$, 
so they must both lie in the base $\bar{d}$. 
This implies that 
$\bar{d}$ is not yet labelled in $\Gamma^*$, so $\exists$ has applied her strategy to choose the colour $y_S$ to label $\bar{d}$ in $\Gamma^+$. 
But this strategy will have chosen $S$ containing $i$ since $\Gamma^*\upharpoonright D$ is already a cone
in $\Gamma^*$. Also $\exists$ never chooses a green edge, so all green edges of $\Gamma^+$ lie in $\Gamma^*$. 

That leaves one (hard) case, where there are two nodes $\beta, \beta',
\in \Gamma$, $\exists$ colours both $(\beta, \alpha)$ and $(\beta',
\alpha)$ red, and the old edge $(\beta, \beta')$ has already been
coloured red (earlier in the game).
If $(\beta, \beta')$ was coloured by $\exists$, then there is no problem. 
So suppose, for a contradiction, that $(\beta, \beta')$ was coloured by
$\forall$. Since $\exists$ chose red colours for $(\alpha, \beta)$
and $(\alpha, \beta')$, it must be the case that there are cones in
$\Gamma^*$ with apexes $\alpha, \beta, \beta'$ and the same base,
$F$, each inducing the same linear ordering $\bar{f} = (f_0,\ldots,
f_{n-2})$, say, on $F$. Of course, the tints of these cones may all
be different. Clearly, no edge in $F$ is labeled green, as no cone base can contain
green edges. It follows that $\bar{f}$  must be labeled by some
yellow colour, $\y_S$, say. Since $\Phi\in \bold J$, it obeys its definition, so
the tint $i$ (say) of the cone from $\alpha$ to $\bar{f}$ lies in
$S$. Suppose that $\lambda$ was the last node of $ F \cup \{ \beta,
\beta' \}$ to be created,as the game proceeded. As $ |F \cup \{
\beta, \beta' \}| = n + 1$, we see that $\exists$ must have chosen
the colour of at least one edge in this : say, $( \lambda, \mu )$.
Now all edges from $\beta$ into $F$ are green, and so chosen by
$\forall$, and the edge  $(\beta, \beta')$ was also chosen by him.
The same holds for edges from $\beta'$ to $F$. Hence $\lambda, \mu
\in F$. We can now see that it was $\exists$ who chose the colour $\y_S$ of
$\bar{f}$. For $\y_S$ was chosen in the round when $F$'s last node,
i.e., $\lambda$ was created. It could only have been chosen by
$\forall$ if he also picked the colour of every edge in $F$
involving $\lambda$. This is not so, as the edge $(\lambda, \mu)$
was coloured by $\exists$, and lies in $F$.
As $i \in S$, it follows from the definition of $\exists$'s strategy
that at the time when $\lambda$ was added, there was already an
$i$-cone with base $\bar{f}$, and apex $\gamma$ say.
We claim that $ F \cup \{ \alpha \}$ and $ F \cup \{ \gamma \}$ are
isomorphic over $F$. For this, note that the only $(n - 1)$-tuples of
either $ F \cup \{ \alpha \}$ or $ F \cup \{ \gamma \}$ with a
yellow colour are in $F$ ( since all others involve a green edge
). But this means that $\exists$ could have taken $\alpha = \gamma$ in
the current round, and not extended the graph. This is contrary to
our original assumption, and completes the proof.

Let $\c A$ be the rainbow algebra defined above.
There is a countable cylindric  algebra $\c A'$ such that $\c A'\equiv\c
A$ and \pe\ has a \ws\ in $H(\c A')$.
We have seen that for $n<\omega$ \pe\ has a \ws\ $\sigma_n$ in $H_n(\c A)$.  We can assume that $\sigma_n$ is deterministic.
Let $\c B$ be a non-principal ultrapower of $\c A$.  We can show that
\pe\ has a \ws\ $\sigma$ in $H(\c B)$ --- essentially she uses
$\sigma_n$ in the $n$'th component of the ultraproduct so that at each
round of $H(\c B)$ \pe\ is still winning in co-finitely many
components, this suffices to show she has still not lost.  
Now use an
elementary chain argument to construct countable elementary
subalgebras $\c A=\c A_0\preceq\c A_1\preceq\ldots\preceq\c B$.  For this,
let $\c A_{i+1}$ be a countable elementary subalgebra of $\c B$
containing $\c A_i$ and all elements of $\c B$ that $\sigma$ selects
in a play of $H_\omega(\c B)$ in which \pa\ only chooses elements from
$\c A_i$. Now let $\c A'=\bigcup_{i<\omega}\c A_i$.  This is a
countable elementary subalgebra of $\c B$ and \pe\ has a \ws\ in
$H(\c A')$.
\end{demo}

\section{The class of neat reducts}
 Here we study the class of neat reducts. 

\begin{theorem}
\begin{enumarab}
\item  Let $n>1$ and $\alpha\geq \omega$. Then the class $\Nr_n\CA_{\alpha}=\Nr_n\RCA_{\alpha}$ is pseudo-elementary, but is not elementary.
Furthermore, $El\Nr_n\CA_{\omega}\subset \RCA_n$, $EL\Nr_n\CA_{\omega}$ is recursively enumerable, and for $n>2$ is 
not finitely axiomatizable.
\item The class $\CCA_n$ of dimension $n\geq 3$, is also psuedo elementary and furthermore, its elementary closure
is  not finitely axiomatixable
\end{enumarab}
\end{theorem}

\begin{demo}{Proof}  We first show that for any infinite $\alpha$, $\Nr_n\CA_{\omega}=\Nr_n\CA_{\alpha}$. Let $\A\in \Nr_n\CA_{\omega}$, 
so that $\A=\Nr_n\B'$, $\B'\in \CA_{\omega}$. Let $\B=\Sg^{\B'}A$. Then $\B\in \Lf_{\omega}$, and $\A=\Nr_n\B$. 
But $\Lf_{\omega}=\Nr_{\omega}\Lf_{\alpha}$ and we are done.
To show that $\Nr_n\CA_{\omega}\subseteq \Nr_n\RCA_{\omega}$, let $\A\in \Nr_n\CA_{\omega}$, then by the above argument
there exists  
then $\B\in \Lf_{\omega}$ such that $\A=\Nr_n\B$. by $\Lf_{\omega}\subseteq \RCA_{\omega},$ we are done.

Hence It is known that class $\Nr_n\CA_{\omega}$ is not elementary. In fact, there is an algebra $\A\in \Nr_n\CA_{\omega}$ 
having a complete subalgebra $\B$, and $B\notin \Nr_n\CA_{n+1}$ 

To show that it is pseudo-elementary, we use a three sorted defining theory, with one sort for a cylindric algebra of dimension $n$ 
$(c)$, the second sort for the Boolean reduct of a cylindric algebra $(b)$
and the third sort for a set of dimensions $(\delta)$. We use superscripts $n,b,\delta$ for variables 
and functions to indicate that the variable, or the returned value of the function, 
is of the sort of the cylindric algebra of dimension $n$, the Boolean part of the cylindric algebra or the dimension set, respectively.
The signature includes dimension sort constants $i^{\delta}$ for each $i<\omega$ to represent the dimensions.
The defining theory for $\Nr_n\CA_{\omega}$ incudes sentences demanding that the consatnts $i^{\delta}$ for $i<\omega$ 
are distinct and that the last two sorts define
a cylindric algenra of dimension $\omega$. For example the sentence
$$\forall x^{\delta}, y^{\delta}, z^{\delta}(d^b(x^{\delta}, y^{\delta})=c^b(z^{\delta}, d^b(x^{\delta}, z^{\delta}). d^{b}(z^{\delta}, y^{\delta})))$$
represents the cylindric algebra axiom ${\sf d}_{ij}={\sf c}_k({\sf d}_{ik}.{\sf d}_{kj})$ for all $i,j,k<\omega$.
We have have a function $I^b$ from sort $c$ to sort $b$ and sentences requiring that $I^b$ be injective and to respect the $n$ dimensional 
cylindric operations as follows: for all $x^r$
$$I^b({\sf d}_{ij})=d^b(i^{\delta}, j^{\delta})$$
$$I^b({\sf c}_i x^r)= {\sf c}_i^b(I^b(x)).$$
Finally we require that $I^b$ maps onto the set of $n$ dimensional elements
$$\forall y^b((\forall z^{\delta}(z^{\delta}\neq 0^{\delta},\ldots (n-1)^{\delta}\rightarrow c^b(z^{\delta}, y^b)=y^b))\leftrightarrow \exists x^r(y^b=I^b(x^r))).$$
For ${\A}\in \CA_n,$ $\Rd_3{\A}$ denotes the $\CA_3$
obtained from $\A$ by discarding all operations indexed by indices in $n\sim 3.$
$\Df_n$ denotes the class of diagonal free cylindric algebras.
$\Rd_{df}{\A}$ denotes the $\Df_n$ obtained from $\A$
by deleting all diagonal elements. To prove the non-finite axiomatizability result we use Monk's algebras.
For $3\leq n,i<\omega$, with $n-1\leq i, {\C}_{n,i}$ denotes
the $\CA_n$ associated with the cylindric atom structure as defined on p. 95 of 
\cite{HMT1}.
Then by \cite[3.2.79]{HMT1}
for $3\leq n$, and $j<\omega$, 
$\Rd_3{\C}_{n,n+j}$ can be neatly embedded in a 
$\CA_{3+j+1}$.  (1)
By \cite[3.2.84]{HMT1}) we have for every $j\in \omega$, 
there is an $3\leq n$ such that $\Rd_{df}\Rd_{3}{\cal C}_{n,n+j}$
is a non-representable $\Df_3.$ (2)
Now suppose $m\in \omega$. By (2), 
choose $j\in \omega\sim 3$ so that $\Rd_{df}\Rd_3{\C}_{j,j+m+n-4}$
is a non-representable $\Df_3$. By (1) we have
$\Rd_{df}\Rd_3{\C}_{j,j+m+n-4}\subseteq \Nr_3{\B_m}$, for some 
${\B}\in \CA_{n+m}.$
Put ${\A}_m=\Nr_n\B_m$.  
$\Rd_{df}{\A}_m$ is not representable, 
a friotri, ${\A}_m\notin \RCA_{n},$ for else its $\Df$ reduct would be 
representable. Therefore $\A_m\notin EL\Nr_n\CA_{\omega}$.
Now let $\C_m$ be an algebra similar to $\CA_{\omega}$'s such that $\B_m=\Rd_{n+m}\C_m$.
Then $\A_m=\Nr_n\C_m$. Let $F$ be a non-principal ultrafilter on $\omega$. Then
$$\prod_{m\in \omega}\A_m/F=\prod_{m\in \omega}(\Nr_n\C_m)/F=\Nr_n(\prod_{m\in \omega}\C_m/F)$$
But $\prod_{m\in \omega}\C_m/F\in \CA_{\omega}$. Hence $\CA_n\sim El\Nr_n\CA_{\omega}$ is not closed under ultraproducts.
It follows that the latter class is not finitely axiomatizable.
In \cite{IGPL} it is proved that for $1<\alpha<\beta$, $El\Nr_{\alpha}\CA_{\beta}\subset S\Nr_{\alpha}\CA_{\beta}$.
\end{demo}
From the above proof it follows that
\begin{corollary} Let $K$ be any class such that $\Nr_n\CA_{\omega}\subseteq K\subseteq \RCA_n$. Then $ELK$ is not finitely axiomatizable
\end{corollary}

\subsection{Weakly but not strongly neat atom structures}

\begin{definition} An atom structure of dimension $\alpha$ is strongly neat, 
if any algebra having this atom structure is in $\Nr_{\alpha}\CA_{\alpha+\omega}$.
It is weakly neat if one algebra based on it is neat reduct, very weakly neat if no algebra based on it is a neat reduct.
\end{definition}

Since there are representable  algebras that are not neat reducts all algebras we are about to construct will be variation on themes
of representable algebras constructed in previous publications.
For completely representable algebras an atom structure is completely representable if it has at least one algebra having this atom structure
that is completely representable. This implies that all algebras based on its atom structure are comletely representable.

Here we show that this not the case with neat atom structures. For this puropse, we construct a weakly neat atom structure, that is not strongly neat.
Similar constructions were used to show that the class of 
neat reduct is not closed under forming subalgebras, and that there are isomorphic algebras that generate 
non -isomorphic ones in extra dimensiona, and that neat subreduct may not be full neat reducts, 
Such  construction, and several related ones, were also used to confirm unresolved  a conjecture of Tarski. 
The construction is joint with Istvan N\'emeti.

We put this construction to another new use,
and therefore we find it appropriate to give the details with several variants that  appeared in three  
publications of the author. 
The latter also shows in addition that there are representable cylindric algebras, satisfying the merry go round identies 
but cannat be a reduct of a $QEA$s.

\begin{theorem} For every ordinal $\alpha>1$, there exists a weakly neat atom structure of dimension $\alpha$, 
that is not strongly neat.
\end{theorem}
\begin{demo}{Proof}
Let $\alpha>1$ and $\F$ is field of characteristic $0$. 
Let 
$$V=\{s\in {}^{\alpha}\F: |\{i\in \alpha: s_i\neq 0\}|<\omega\},$$
Note that $V$is a vector space over the field $\F$. 
We will show that $V$ is a weakly neat atom structure that is not strongly neat.
Indeed $V$ is a concrete atom structure $\{s\}\equiv _i \{t\}$ if 
$s(j)=t(j)$ for all $j\neq i$, and
$\{s\}\equiv_{ij} \{t\}$ if $s\circ [i,j]=t$.
Let $y$ denote the following $\alpha$-ary relation:
$$y=\{s\in V: s_0+1=\sum_{i>0} s_i\}.$$
Note that the sum on the right hand side is a finite one, since only 
finitely many of the $s_i$'s involved 
are non-zero. 
For each $s\in y$, we let 
$y_s$ be the singleton containing $s$, i.e. $y_s=\{s\}.$ 
Define 
${\A}\in WQEAs_{\alpha}$ 
as follows:
$${\A}=\Sg^{\C}\{y,y_s:s\in y\}.$$
We shall prove that  
$$\Rd_{SC}\A\notin \Nr_{\alpha}SC_{\alpha+1}.$$ 
That is for no $\mathfrak{P}\in SC_{\alpha+1}$, it is the case that $\Sg^{\C}X$ exhausts the set of all $\alpha$ dimensional elements 
of $\mathfrak{P}$. 
So assume, seeking a contradiction,  that $\Rd_{SC}{\A}\in \Nr_{\alpha}SC_{\alpha+1}$. 
Let $X=\{y_s:s\in y\}$. 
Of course every element of $X,$ being a singleton, is an atom.
Next we show that $\A$ is atomic, i.e evey non-zero element contains a minimal non-zero element.  
Towards this end, let $s\in {}^{\alpha}\F^{(\bold 0)}$ be an arbitrary sequence. 
Then 
$$\langle s_0,s_0+1-\sum_{i>1}s_i,s_i\rangle_{i>1}$$ 
and 
$$\langle \sum_{0< i<\alpha}s_i-1,s_i\rangle_{i\geq 1}$$ are elements in y.
Since 
$$\{s\}={\sf c}_1\{\langle s_0,s_0+1-\sum_{i>1}s_i,s_i\rangle_{i>1}\}\cap
{\sf c}_0\{\langle \sum_{0\neq i<\alpha}s_i-1,s_i\rangle_{i\geq 1}\},$$ 
it follows
that $\{s\}\in A$. 
We have shown that $A$ contains all singletons, hence is atomic. Call an element rectangular if ${\sf c}_0x\cap {\sf c}_1x=x$.
As easily checked, every singleton is rectangular.
Also $y=\bigcup X$, and so $y=\sup X$ exists in ${\A}$.
Finally, let $x=\{s\}$ be an atom of ${\A}$.
Then 
$${\sf c}_0x\cap y=\{\langle \sum_{i>0}s_i-1,s_i\rangle_{i>0}\}$$ 
and
$${\sf c}_1x\cap y=\{\langle s_0,s_0+1-\sum_{i>1}s_i,s_i\rangle_{i>1}\};$$ 
which are singletons, hence atoms.
Let $$\tau(x,y) ={\sf c}_1({\sf c}_0x\cdot {\sf s}_1^0{\sf c}_1y)\cdot {\sf c}_1x\cdot {\sf c}_0y.$$
Let $$Y=\{\tau(a,b):a,b\in X\}.$$
We now show that, given that $\A$ is a neat reduct, $k=\sup Y$ exists in $\A$. 
Let $$\tau_{\alpha}(x,y) = {\sf c}_{\alpha}({\sf s}_{\alpha}^1{\sf c}_{\alpha}x\cdot {\sf s}_{\alpha}^0{\sf c}_{\alpha}y).$$
Then $\tau_{\alpha}(x,y)\leq \tau({\sf c}_{\alpha}x,{\sf c}_{\alpha}y)$
and $x,y$ are $\alpha$-closed and rectangular, then
$\tau_{\alpha}(x,y)=\tau(x,y)$.
Indeed, 
computing we get
$$\tau_{\alpha}(x,y)={\sf c}_{\alpha}
({\sf s}_{\alpha}^1{\sf c}_{\alpha}x\cap {\sf s}_{\alpha}^0{\sf c}_{\alpha}y)$$ 
$$\leq {\sf c}_{\alpha}({\sf s}_{\alpha}^1
({\sf c}_0{\sf c}_{\alpha}x\cap {\sf c}_1{\sf c}_{\alpha}x)\cap {\sf s}_{\alpha}^0
({\sf c}_0{\sf c}_{\alpha}y\cap {\sf c}_1{\sf c}_{\alpha}y))$$
(Here, equality holds if $x,y$ are $\alpha$-closed and rectangular)
$$={\sf c}_{\alpha}({\sf s}_{\alpha}^1{\sf c}_0{\sf c}_{\alpha}x\cap 
{\sf s}_{\alpha}^1{\sf c}_1{\sf c}_{\alpha}x\cap {\sf s}_{\alpha}^0
{\sf c}_0{\sf c}_{\alpha}y\cap {\sf s}_{\alpha}^0{\sf c}_1{\sf c}_{\alpha}y).$$
$$={\sf c}_{\alpha}({\sf s}_{\alpha}^1{\sf c}_0{\sf c}_{\alpha}x\cap {\sf c}_1{\sf c}_{\alpha}x\cap {\sf c}_0{\sf c}_{\alpha}y\cap {\sf s}_{\alpha}^0{\sf c}_1{\sf c}_{\alpha}y)$$
$$={\sf c}_{\alpha}({\sf s}_{\alpha}^1{\sf c}_0{\sf c}_{\alpha}x\cap {\sf s}_{\alpha}^0{\sf c}_1{\sf c}_{\alpha}y)\cap {\sf c}_1{\sf c}_{\alpha}x\cap {\sf c}_0{\sf c}_{\alpha}y$$
$$={\sf c}_{\alpha}({\sf s}_{\alpha}^1{\sf c}_0{\sf c}_{\alpha}x\cap {\sf s}_{\alpha}^1{\sf s}_1^0{\sf c}_1{\sf c}_{\alpha}y)\cap {\sf c}_1{\sf c}_{\alpha}x\cap {\sf c}_0{\sf c}_{\alpha}y$$
$$={\sf c}_{\alpha}{\sf s}_{\alpha}^1({\sf c}_0{\sf c}_{\alpha}x\cap {\sf s}_1^0{\sf c}_1{\sf c}_{\alpha}y)\cap {\sf c}_1{\sf c}_{\alpha}x\cap {\sf c}_0{\sf c}_{\alpha}y$$
$$={\sf c}_1{\sf s}_1^{\alpha}({\sf c}_0{\sf c}_{\alpha}x\cap {\sf s}_1^0{\sf c}_1{\sf c}_{\alpha}y)\cap {\sf c}_1
{\sf c}_{\alpha}x\cap {\sf c}_0{\sf c}_{\alpha}y$$
$$={\sf c}_1{\sf s}_1^{\alpha}{\sf c}_{\alpha}({\sf c}_0{\sf c}_{\alpha}x\cap {\sf s}_1^0{\sf c}_1{\sf c}_{\alpha}y)\cap {\sf c}_1
{\sf c}_{\alpha}x\cap {\sf c}_0{\sf c}_{\alpha}y.$$
$$={\sf c}_1{\sf c}_{\alpha}({\sf c}_0{\sf c}_{\alpha}x\cap {\sf s}_1^0{\sf c}_1{\sf c}_{\alpha}y)
\cap {\sf c}_1{\sf c}_{\alpha}x\cap {\sf c}_0{\sf c}_{\alpha}y$$
$$={\sf c}_1({\sf c}_0{\sf c}_{\alpha}x\cap {\sf s}_1^0{\sf c}_1{\sf c}_{\alpha}y)
\cap {\sf c}_1{\sf c}_{\alpha}x\cap {\sf c}_0{\sf c}_{\alpha}y$$
$$=\tau({\sf c}_{\alpha}x,{\sf c}_{\alpha}y).$$  
Let $k=\tau_{\alpha}^{\B}(y,y)$, 
where ${\B}$ is any algebra in $SC_{\alpha+1}$
such that $\Rd_{SC}{\A}=\Nr_{\alpha}\B$. Now $k\in A$.
Let $a,b\in X$, and $i\in \{0,1\}.$ 
SiSince ${\sf s}_{\alpha}^i$ is a boolean endomorphism for $i<\alpha$
we have $${\sf s}_{\alpha}^0a\leq {\sf s}_{\alpha}^0y\text { and }
{\sf s}_{\alpha}^1b\leq {\sf s}_{\alpha}^1y.$$
But $a,b$ are rectangular and together with $y$ are $\alpha$-closed,
we obtain that
$$\tau(a,b)={\sf c}_{\alpha}({\sf s}_{\alpha}^1a\cap {\sf s}_{\alpha}^0b)
\leq {\sf c}_{\alpha}({\sf s}_{\alpha}^1y\cap {\sf s}_{\alpha}^0y)=k.$$
This shows that $k\in B$, is an upper bound of $Y$.
We claim  that $k$ is in fact the least such upper bound, i.e. that $k=\sup Y.$
Towards this end, suppose that $t$ is an upper bound of $Y$ in $\A$.
We have to show that $k\leq t$.
Assume, to the contrary, that this is not the case, i.e. that $k-t\neq 0$.
Since ${\A}$ is atomic, there is an atom $z$ say, below $k-t$ 
i.e.  $$z\leq k\text { and }z\leq -t.$$                            
From $0\neq z\leq w={\sf c}_{\alpha}({\sf s}_{\alpha}^1y\cap {\sf s}_{\alpha}^0y)$, and $z$
is $\alpha$-closed, it readily follows that 
$$z\cap ({\sf s}_{\alpha}^1y\cap {\sf s}_{\alpha}^0y)\neq 0$$ thus 
$${\sf c}_0z\cap {\sf c}_1z\cap ({\sf s}_{\alpha}^1y\cap {\sf s}_{\alpha}^0y)\neq 0.$$
From ${\sf s}_{\alpha}^i{\sf c}_iz={\sf c}_iz$, and the fact that the ${\sf s}_{\alpha}^i$'s are
boolean endomorphisms for each $i<\alpha$, we get  
$${\sf s}_{\alpha}^1({\sf c}_1z\cap y)\cap {\sf s}_{\alpha}^0({\sf c}_0z\cap y)\neq 0.$$
Now let $a= {\sf c}_1z\cap y$. Then  
$a$ is a singleton, hence $a\in At{\C}$. Also $a\leq y$.
But $y=\sup X$ and $a\cap y=a\neq 0$, therefore $a\cap x\neq 0$, for some
atom $x\in X$, and so $a=x\in X$.
Analogously, if $b={\sf c}_0z\cap y$, then $b\in X$.
Being atoms, $a$ and $b$ are rectangular, and therefore we have from
$$\tau(a,b)={\sf c}_{\alpha}({\sf s}_{\alpha}^1({\sf c}_1z\cap y)\cap {\sf s}_{\alpha}^0({\sf c}_0z\cap y)).$$
and that  
$\tau(a,b) > 0.$

On the other hand, by invoking definitions and the fact that $z$ is 
rectangular a straightforward computation yields:
$$\tau(a,b)=\tau({\sf c}_1z\cap y, {\sf c}_0z\cap y)$$
$$= {\sf c}_1({\sf c}_0({\sf c}_1z\cap y)\cap {\sf s}_1^0{\sf c}_1({\sf c}_0z\cap y))\cap {\sf c}_1({\sf c}_1z\cap y)\cap 
{\sf c}_0({\sf c}_0z\cap y)).$$
$$={\sf c}_1({\sf c}_0({\sf c}_1z\cap y)\cap {\sf s}_1^0{\sf c}_1({\sf c}_0z\cap y))\cap 
({\sf c}_1z\cap {\sf c}_1y)\cap ({\sf c}_0z\cap {\sf c}_0y).$$
$$={\sf c}_1({\sf c}_0({\sf c}_1z\cap y)\cap {\sf s}_1^0{\sf c}_1({\sf c}_0z\cap y))\cap 
({\sf c}_0z\cap {\sf c}_1z)\cap ({\sf c}_0y\cap {\sf c}_1y).$$
$$={\sf c}_1({\sf c}_0({\sf c}_1z\cap y)\cap {\sf s}_1^0{\sf c}_1({\sf c}_0z\cap y))\cap z\cap ({\sf c}_0y\cap {\sf c}_1y).$$
That is, $0<\tau(a,b)\leq z$. But $z$ is an atom, and so $z=\tau(a,b)$.
Since $t$ is an upper bound for $Y$, we get $z\leq t$.
But we have $z\leq -t$. Hence $z=0$ which is impossible,
since $z$ is an atom. 
This means that $k\leq t$, and so $k=\sup Y$ as required. 
We will determine $k.$ 
We start by evaluating $\tau(a,b)$ for $a,b\in X$.
Towards this end, let $a,b\in X$. Assume that $a=y_r$ and $b=y_t$ 
with 
$$r=\langle r_i:i<\alpha\rangle\text { and }t=\langle t_i:i<\alpha\rangle$$ 
are elements in 
$y$.
$$\text { If } r_1\neq t_0\text { or } r_i\neq t_i \text { for some }i>1, \text { then }
\tau(y_r,y_s)=0.$$
Else, $r_1=t_0$ and $r_i=t_i$ for all $i>1$.
In this case, 
$$\tau(y_r,y_t)=\{s\}\text { where }s_0=r_0, \text { and } s_i=t_i \text { for all } i>0.$$
Moreover, an easy computation shows that 
$$s_0+2=s_1+2\sum_{i>1}s_i.$$
Now let $w$ be the set of all such elements, i.e.
$$w=\{s\in {}^{\alpha}\F^{(\bold 0)}: s_0+2=s_1+2\sum_{i>1}s_i\}.$$
Let $s\in w$. Then putting 
$$r=\langle s_0,s_0+1-\sum_{i>0}s_i,s_i\rangle_{i>1}$$ 
and 
$$t=\langle s_0+1-\sum_{i>0}s_i,s_1,s_i\rangle_{i>1},$$ 
we get $r,t\in y$ and $\tau(y_r,y_s)=\{s\}$.
We have shown that $$Y\sim \{0\}=\{\{s\}:s\in w\},$$ 
and so $$w=\bigcup Y.$$
Now clearly $w\leq k$. We will show that $k\leq w$.
But this easily follows from the fact that $A$ contains all singletons.
In more detail, let $z\in {}^{\alpha}\F^{(\bold 0)}\sim w$. 
Then $\{z\}$, and hence $^{\alpha}\F^{(\bold 0)}\sim \{z\}\in A$. 
Since $w\subseteq ^{\alpha}\F^{(\bold 0)}\sim \{z\}$, 
we have $\tau(a,b)\subseteq {}^{\alpha}\F^{(\bold 0)}\sim \{z\}$, 
for every $a,b\in X$, i.e. 
$^{\alpha}\F^{(\bold 0)}\sim \{z\}$ is an upper bound of of $Y$. 
But this means that $k\subseteq  V\sim \{z\}$, i.e. $z\notin k$.
We have proved that $k\leq w$.  
Next we proceed to show that, in fact, $w\notin A$. 
(That is $w$ is a new $\alpha$ dimensional element).
This contradiction will show that $\Rd_{SC}\A\notin \Nr_{\alpha}SC_{\beta}.$
Let
$$Pl=\{\{s\in {}^{\alpha}\F^{(\bold 0)}:t+\sum (r_is_i)=0\}:\{t,r_i:i<\alpha\}\subseteq \F\}.$$
$Pl$ consists of all hyperplanes of dimension $\alpha$. 
For $i<\alpha$, let 
$$q_i=\{s\in {}^{\alpha}\F^{(\bold 0) }:s_i+1=\sum_{j\neq i}s_j\}.$$
Let $i\in \alpha$ and $k,l\in \alpha$ be distinct.
Then if $i\in \{k,l\}$, $i=k$ say, then 
$${\sf s}_{kl}q_i=q_l.$$
Else, we have $i\notin \{k,l\}$, so that $i,k,l$ are pairwise distinct,
in which case we have 
$${\sf s}_{kl}q_i=q_i.$$ 
Now let 
$$Pl^S=\{q_i:i<\alpha\}.$$ 
Then by the above $Pl^S$ is closed under the operations
${\sf s}_{kl}$ for all $k,l\in \alpha$. 
Notice that $y=q_0\in Pl^S$. Thus $Pl^S$
consists of all hyperplanes obtained by interchanging the $k,l$ co-ordinates
of $y$, for all $k,l\in \alpha$, i.e. implementing the operation ${\sf s}_{kl}$. 
Here the superscript
$S$ is short for substitutions (corresponding to replacements) 
reflecting the fact that  $Pl^S$ consists of {\it substituted}  versions of $y$. Let
$$Pl^<=\{p\in Pl: {\sf c}_ip=p, \text { for some }i<\alpha\}.$$
If $p\in Pl^<$ and ${\sf c}_ip=p$, then 
$p$ is a hyperplane parallel to the $i$-th axis.

Note  that for $p\in Pl$, $p=\{s\in {}^{\alpha}\F^{(\bold 0)}:t+\sum_ir_is_i=0\}$ say, then
${\sf c}_ip=p$ (i.e. $p$ is parallel to the $i$-th axis) iff $r_i=0$. 
Note too, that if $p\in Pl^<$ and $k,l\in \alpha$ then ${\sf s}_{kl}p\in Pl^<$.

In the following we summarize the above, 
and state some other facts easy to check as well.
$$(1)\ \ y=q_0\in Pl^S.$$
$$(2)\  \ Pl^S\cup \{w,{\sf d}_{ij}:i,j\in \alpha\}\subseteq Pl.$$
$$(3)\ \text { If }q\in Pl^S (Pl^<) \text { and }i,j\in \alpha, 
\text { then } {\sf s}_{ij}q\in Pl^S(Pl^<).$$
$$(4)\  \ Pl^S\cap Pl^<=\emptyset, w\notin Pl^<\text { and }{}^{\alpha}\F^{(\bold 0)}\in Pl^<.$$
$$(5)\ \{{\sf d}_{ij}:i\neq j, i,j\in \alpha\}\subseteq Pl^<\text { iff }\alpha\geq 3.$$

For further use, we shall need:

\begin{athm}{Notation} 

\begin{enumroman}

\item Let $X$ be a set. Then $Z\subseteq_{\omega}X$ abbreviates 
``$Z$ is a finite subset $X$". $\wp_{\omega}X$ denotes all finite subsets of $X$, i.e.
$$\wp_{\omega}X=\{Z\in \wp(X): Z\subseteq_{\omega} X\}.$$

\item Let ${\A}\in K_{\alpha}$, $a\in A$ and $\Delta\subseteq_{\omega}\alpha$.
Then ${\sf c}_{(\Delta)}a={\sf c}_{{i_0}}\cdots {\sf c}_{i_{n-1}}a$, 
where $\Delta=\{i_0,\cdots i_{n-1}\}$. 
Because cylindrifications commute, 
the definition of ${\sf c}_{(\Delta)}$ does not depend on any
linear order defined on $\Delta$.  
When $n=0$, i.e. $\Delta$ is empty, then ${\sf c}_{(\Delta)}a={\sf c}_{\emptyset}a:=a$. 
\item To simplify notation, from now on we may write $-a$ 
instead of $^{\alpha}\F^{(\bold 0)}\sim a.$
It will be clear from context whether $-$ refers to the operation of substraction in $\F$
or the operation of complementation in $\A.$
We also write $1^{\A}$, or just $1,$ 
instead of the more cumbersome 
$^{\alpha}\F^{(\bold 0)}.$

\end{enumroman}

\end{athm}

{\bf Some more definitions.}
$$G=\{q,-q,\ p,-p,\ {\sf c}_{(\Delta)}\{\bold 0\},
-{\sf c}_{(\Delta)}\{\bold 0\}:\ q\in Pl^S,\ p\in Pl^<\cup \{{\sf d}_{01}\},\ \Delta\subseteq_{\omega} \alpha,\  0\in \Delta\}.$$
Forming finite intersections of elements in $G$, we let
$$G^*=\{\bigcap_{i\in n}g_i:n\in \omega,g_i\in G\},$$
and forming finite unions of elements in $G^*$, we let
$$G^{**}=\{\bigcup_{i\in n}g_i:n\in \omega, g_i\in G^*\}.$$
It is easy to see that $\{y,y_s:s\in y\}\subseteq G^{**}$, and $G^{**}$ 
is a boolean field of sets.
We start by proving that $w\notin G^{**}$. To this end, we set:
$$L =\{p\in Pl^<: {\sf c}_0p\neq p\}\text { and } P(0)=L\cup \{{\sf d}_{01}\}.$$
Notice that $L = P(0)$ iff $\alpha>2$.
$L$ stands for the set of ``lines" 
not parallel to the $0$th axis. 
Being in $Pl^<$, any such line is parallel to some other axis. 
When $\alpha=2$, i.e. in the  plane $\F\times \F$, these 
are precisely the lines that are parallel to the ``$y$ axis".
Next we define
$$G_1=\{g\in G^*:g\subseteq q,\text{ for some }q\in Pl^S\}$$ and
$$G_2=\{g\in G^*:g\not\subseteq q,\text{ for all }q\in Pl^S \text 
{ and }g\subseteq p, \text{ for some } p\in P(0)\}.$$ 
Note that $Pl^S\subseteq G_1$ and that $P(0)\subseteq G_2$.
Note too, that $G_1\cap G_2=\emptyset$.
Now let 
$$G_3=\{p_1\cap p_2\ldots \cap p_k:\ k\in \omega,\ 
\{p_1,p_2,\ldots, p_k\}\subseteq G\sim (Pl^S\cup P(0))\}.$$
It is easy to see that $G^*=G_1\cup G_2\cup G_3$.
To prove that $w\notin G^{**}$ we need:

\begin{athm}{Claim} If $g\in G_3$ and $0\neq g$, then $g\not\subseteq w$.
\end{athm}

{\bf Proof of Claim}

Assume that $g=p_1\cap p_2\ldots  \cap p_k$ say, with 
$p_i\in G$ and $p_i\notin (Pl^S\cup P(0))$
for $1\leq i\leq k$, and let $z\in g.$
Then for $1\leq i\leq k$, 
we have in addition that 
$$p_i\in \{p,-p, {\sf c}_{(\Delta)}\{\bold 0\}, -{\sf c}_{(\Delta)}\{\bold 0\}, -q: p\in Pl^<\cup \{{\sf d}_{01}\}, q\in Pl^S, \Delta\subseteq_{\omega}\alpha,
0\in \Delta\}.$$ 
Let $[{}]$ be the function from $G$ into $\wp(\F)$ defined as follows: 
$$[p] = \{1/r_0(-t-\sum_{0\neq i<\alpha} r_iz_i)\} \text { if }p=-\{s\in ^{\alpha}\F^{(\bold 0)}:t+\sum_{i<\alpha} r_is_i=0\}, 
r_0\neq 0.$$ 
Else  
$$[p]=0.$$
Let $$r\in \F\sim ((\bigcup_{1\leq i\leq k}[p_i])\cup [-w])$$ 
be arbitrary,
and let $$z_r^0=z\sim \{(0,z_0)\}\cup \{(0,r)\}.$$
By the choice of $r\notin [-w]$ we have
$$r\neq -2+z_1+2\sum_{i>1} z_i.$$ 
Hence $z_r^0\notin w$. If $p_i=-q$ with  $q\in Pl^S$, then similarly, $r\notin [-q]$ hence $z_r^0\notin p_i$.
Now assume that $p_i\in Pl^<$. Then we have $p_i\notin P(0)$ hence ${\sf c}_0p_i=p_i.$
Since $z\in p_i$ it follows that $z_r^0\in {\sf c}_0p_i.$ But $p_i\notin P(0)$ hence $z_r^0\in {\sf c}_0p_i=p_i.$
If $p_i=-p$ where $p\in Pl^<$, then $z\notin p$ and so $z_r^0\notin p$ since again ${\sf c}_0p=p$.
That is if $z_r^0\in p$ then $z$ differing from $z_r^0$ in atmost the $0$th place would also be in $p$.
Now assume that $p_i={\sf c}_{(\Delta)}\{\bold 0\}$. Then if $z\in p_i$ then $z_r^0\in p_i$ since ${\sf c}_0p_i=p_i$ by $0\in \Delta$.
Finally if $p_i=- {\sf c}_{(\Delta)}\{\bold 0\}$ then $z\notin {\sf c}_{(\Delta)}\{\bold 0\}$ iff $z_r^0\notin {\sf c}_{(\Delta)}\{\bold 0\}$ by $0\in \Delta$. 
We have shown that 
$$z_r^0\in g\sim w\text { i.e. }g\not\subseteq w.$$

We now proceed to show that $w\notin G^{**}$.
Assume to the contrary, that 
$$w=\bigcup\{g_i^1:i<n_1\}\cup\bigcup\{g_i^2:i<n_2\}\cup\bigcup \{g_i^3:i<n_3\}$$ 
where 
$$\{g_i^j:i<n_j\}\subseteq G_j\text { for }j\in \{1,2,3\}.$$
By the above, we have $g_i^3=0$, for all $i<n_3$.
From the definition of $G_1$ and $G_2$ , it follows that
$$w\subseteq \bigcup\{q_j:j<n_1\}\cup\bigcup \{p_j:j<n_2\}$$ 
where 
$$\{q_j:j<n_1\}\subseteq
Pl^S$$ 
and 
$$\{p_j:j<n_2\}\subseteq P(0).$$
Since $w\subseteq -{\sf d}_{01}$ for $\alpha=2$, and $P(0)=L$ for $\alpha>2$ 
we can, and will, assume that
$$\{p_j:j<n_2\}\subseteq L.$$
Before proceeding, let us take a special case as an illustration. Assume that $\alpha=3$.
Then each $p_j$ is a hyperplane that is determined by an equation.
Assume that the following equations determine the planes $p_j$ $j<n_2$.
$$(1)\ \ \ \ t_0+r_{00}x_0+r_{01}x_1+r_{02}x_2=0$$
$$\ \ \ \ t_1+r_{10}x_0+r_{11}x_1+r_{12}x_2=0$$
$$\cdots$$
$$\cdots$$
$$\ \ \ \ t_{n_2-1} +r_{n_2-1,0}x_0+ r_{n_2-1, 2}x_2=0.$$
In each of these equation one of the coefficients other than the zeroth is equal to $0$. The zeroth coefficient is {\it not} zero.
Therefore, these equations determine hyperplanes parallel to one of the axis
other than the zeroth axis.
Now consider the equations
$$(2)\ \ \ \ x_0+2=x_1+2x_2$$ 
and
$$(3)\ \ \ \ x_{\pi(0)}+1=x_{\pi(1)}+x_{\pi(2)}$$
where $\pi$ is a permutation of $\{0,1,2\}$.
$(3)$ consists of $6$ equations, only three of which are distinct (because of commutativity). Such equations 
represent the $q_i$'s.
Then it can be easily seen that there is an $s$ that satisfies 
(2) but does not satisfy (the equations in) $(1)$ and $(3)$.

And indeed in the general case, it can be seen by implementing easy linear algebraic 
arguments that, for every $n\in \omega$, 
for every $m\in \omega\cap (\alpha+1)$,
and for every  system
$$t_0+\sum(r_{0i}x_i)=0$$
$$\cdots$$
$$\cdots$$
$$t_n+\sum(r_{ni}x_i)=0,$$
of equations, such that for all $j\leq n$, there exists $i<\alpha$, such that
$$r_{ji}=0\text { and }r_{j0}\neq 0,$$
the equation 
$$x_0+2=x_1+2\sum x_j$$ 
has a solution $s$ in the weak space
$^{\alpha}\F^{(\bold 0)}$, such that
$s\notin q_i$ for every $i<m$, and such that $s$ is not a solution of
$$t_j+\sum (r_{ji}x_i)=0,$$ for every $j\leq n$.
Geometrically (and intuitively) 
a finite union of hyperplanes 
cannot cover a hyperplane, 
unless it is one of them.
By taking $n=n_1$ and $m=n_2$ we get 
the desired conclusion.
Let us prove this  formally.  It clearly suffices to show that for all positive integers $m$, $n$ the equation 
$$x_0+2=x_1+2\sum_{1< i \leq m} x_i$$
 has a solution which is not a solution of any of the following equations:
$$(4)\ \ \ \ t_k+\sum_{j\leq m}r_{kj}x_j=0,$$
$$(5)\ \ \ \ x_l+1=\sum_{1\neq i\leq m}x_i,$$
where $k=1,\ldots , n$, each $r_{k0}\neq 0$, for each $k=1,\ldots , n$ there is a $j$ with $0<j\leq m$ such that $r_{kj}=0$, and $0\leq l\leq m.$
To do this, first substitute $x_0=-2+x_1+2\sum_{1<i\leq m}x_i$ into each of the equations $(4)$ and $(5)$, obtaining
$$(6) \ \ \ \ t_k-2r_{k0}+(r_{k0}+r_{k1})x_1+\sum_{1<j\leq m}(2r_{k0}+r_{kj})x_j=0,$$
$$(7) \ \ \ \ -1+\sum_{1<j\leq m} x_j=0,$$
$$(8) \ \ \ \ 3+\sum_{1<j\leq m}(- 3 x_j)=0,$$
$$(9)\ \ \ \ 3-2x_1-x_k+\sum_{j\leq m, j\notin \{0,1,k\}}(-3x_j)=0\text { for } k>1.$$
Now we define $s_1,\ldots s_m$ by recursion. Choose $s_1$ so that $$t_k-2r_{k0}+(r_{k0}+r_{k1})s_1\neq 0$$
for each $k$ such that $r_{k0}+r_{k1}\neq 0$. Having defined $s_t$ with $t<m$ choose $s_{t+1}$ so that
$$t_k-2r_{k0}+(r_{k0}+r_{k1})s_1+\sum_{1<j\leq t+1}(2r_{k0}+r_{kj})s_j\neq 0$$
for each $k$ for which $2r_{k0}+r_{k,t+1}\neq 0$; also, if $t+1=m$, assure that the equations $(7)- (9)$ all fail.
Finally let $s_0=-2+s_1+2\sum_{1<i}s_i.$
Clearly the desired conclusion holds. 
We have proved that $w\notin G^{**}$. .

To show that 
$w\notin A$, we will  show that $G^{**}$ is closed under 
the polyadic set operations.
It only remains to show that $G^{**}$ 
is closed under cylindrifications and substitutions, 
since by definition, it is a boolean field of sets and 
contains the diagonal elements. ( Recall that for $\alpha>2$, ${\sf d}_{ij}\in Pl^<$.)

\begin{enumarab}

\item $G^{**}$ is closed under cylindrifications.

It is enough to show that (since the ${\sf c}_i$'s are additive), that
for $j\in \alpha$ and $g\in G^*$ arbitrary, we have ${\sf c}_jg\in G^{**}$.
For this purpose, put for every $p \in Pl$
$$p(j|0)={\sf c}_j\{s\in p: s_j=0\}\text { and } (-p)(j|0)=-p(j|0).$$  
Then it is not hard to see that 
$$p(j|0)=\{s\in {}^{\alpha}\F^{(\bold 0)}:t+\sum_{i\neq j}(r_is_i)=0\},$$
if $$p=\{s\in {}^{\alpha}\F^{(\bold 0)}:t+\sum_{i<\alpha}(r_is_i)=0\},$$ 
 
Indeed assume that $s_1\in {\sf c}_j\{s\in p: s_j=0\}$.
Then there exists $s_2$ in $p$ such that $s_1$ and  $s_2$ agree at all components except the $j$th where 
$s_2(j)=0$.
Then $$0=t+\sum_{i\neq j} r_is_1(i)=t+\sum_{i\neq j}r_is_2(i).$$
The other inclusion is analogous. Assume that $s_1\in p$ and $r_j\neq 0$ such that $t+\sum_{i\neq j}r_is_1(i)=0.$
Then $s_1(j)=0$. Hence $s_1\in {\sf c}_j\{s\in P: s_j=0\}.$ 
It follows thus that  
$$p(j|0)\in Pl^<\text { for every }p\in Pl.$$

Now let $j$ and $g$ be as indicated above. We can assume that
$$g=q_1\cdots \cap q_l\cap p_1\cap\cdots \cap p_n\cap  -Q_1\cdots \cap 
-Q_L$$
$$\cap -P_1\cdots \cap -P_m\cap y\cap -{\sf c}_{(\Delta_1)}
\{\bold 0\}\cdots\cap -{\sf c}_{(\Delta_N)}\{\bold 0\},$$ 
where
$$l, L, n, m, N\in \omega, q_i, Q_i\in Pl^S, 
p_i, P_i\in Pl^<\cup \{{\sf d}_{01}\},$$ 
$${\sf c}_jp_i\neq p_i,\  {\sf c}_jP_i\neq P_i,\ {\sf c}_jq_i\neq q_i,\  {\sf c}_jQ_i\neq Q_i,$$ 
$$y\in \{{\sf c}_{(\Delta)}\{\bold 0\},1:\Delta\in \wp_{\omega}\alpha,\ 0\in \Delta,\ j\notin \Delta\},$$ 
and 
$$\{\Delta_1,\cdots,\Delta_n\}\subseteq \{x\in \wp_{\omega}\alpha:\ j\notin x, 0\in x\}.$$
This is so because $${\sf c}_j(x\cap {\sf c}_jy)={\sf c}_jx\cap {\sf c}_jy,$$
$$\text { if } j\notin \Delta x\text { then }{\sf c}_jx=x$$ 
and $${\sf c}_{(\Delta)}\{\bold 0\}\cap {\sf c}_{(\Gamma)}\{\bold 0\}={\sf c}_{(\Delta \cap \Gamma)}\{\bold 0\}.$$

We distinguish between $2$ cases:

{\bf Case 1.}
$$y={\sf c}_{(\Delta)}\{\bold 0\}\text { and }j\notin \Delta.$$ 
Then a lenghly but routine computation gives
$${\sf c}_j(q_1\cdots \cap q_l\cap -Q_1\cdots \cap 
-Q_L\cap p_1\cdots\cap p_n\cap -P_1\cdots \cap -P_m$$
$$\cap {\sf c}_{(\Delta)}\{\bold 0\}\cap -{\sf c}_{(\Delta_1)}\{\bold 0\}\cdots \cap 
-{\sf c}_{(\Delta_N)}\{\bold 0\})$$
$$=q_1(j|0)\cap q_l(j|0)\cap -Q_1(j|0)\cdots \cap 
-Q_L(j|0)\cap p_1(j|0)\cdots $$
$$\cap p_n(j|0)\cap -P_1(j|0)\cdots \cap -P_m(j|0)$$
$$\cap {\sf c}_j
{\sf c}_{(\Delta)}\{\bold 0\}\cap -{\sf c}_j{\sf c}_{(\Delta_1)}\{\bold 0\}\cdots \cap -
{\sf c}_j{\sf c}_{(\Delta_N)}\{\bold 0\}.$$
Indeed let $$s\in 
{\sf c}_j(q_1\cdots \cap q_l\cap -Q_1\cdots \cap 
-Q_L\cap p_1\cdots\cap p_n\cap -P_1\cdots \cap -P_m$$
$$\cap {\sf c}_{(\Delta)}\{\bold 0\}\cap -{\sf c}_{(\Delta_1)}\{\bold 0\}\cdots \cap 
-{\sf c}_{(\Delta_N)}\{\bold 0\}).$$
For a sequence $t$, we write $t_u^j$ for the sequence that agrees with $t$ except at $j$ where $t_u^j(j)=u$.
Now, by definition of cylindrifications, there exists $u\in \F$ such that 
$$s_u^j\in 
(q_1\cdots \cap q_l\cap -Q_1\cdots \cap 
-Q_L\cap p_1\cdots\cap p_n\cap -P_1\cdots \cap -P_m$$
$$\cap {\sf c}_{(\Delta)}\{\bold 0\}\cap -{\sf c}_{(\Delta_1)}\{\bold 0\}\cdots \cap 
-{\sf c}_{(\Delta_N)}\{\bold 0\}).$$
Now since $s_u^j\in q_1$ and $s_u^j\in {\sf c}_{(\Delta)}\{\bold 0\}$ and $j\notin \Delta$ it readily follows that $u=0$.
But then $s\in q_1(j|0).$ Similarly $s\in q_i(j|0)$ for every $i\leq l$ and $s\in p_i(j|0)$ for each $i\leq n$. It is clear that $s\in {\sf c}_j{\sf c}_{(\Delta)}\{\bold 0\}$.
Also $s\notin Q_i(j|0)$ for any $i\leq L$ for else $s_0^j\in Q_i$
which is not the case.
Same reasoning gives $s\notin P_i(j|0)$ for all $i\leq m$. 
Finally $s\notin {\sf c}_j{\sf c}_{(\Delta_i)}\{\bold 0\}$ for $i\leq N$, because $s_u^j\notin {\sf c}_{(\Delta_j)}\{\bold 0\}$.
Now conversely if we start with 
$$s\in q_1(j|0)\cap q_l(j|0)\cap -Q_1(j|0)\cdots \cap 
-Q_L(j|0)\cap p_1(j|0)\cdots $$
$$\cap p_n(j|0)\cap -P_1(j|0)\cdots \cap -P_m(j|0)$$
$$\cap {\sf c}_j
{\sf c}_{(\Delta)}\{\bold 0\}\cap -{\sf c}_j{\sf c}_{(\Delta_1)}\{\bold 0\}\cdots \cap -
{\sf c}_j{\sf c}_{(\Delta_N)}\{\bold 0\}.$$
Then $$s_0^j\in 
(q_1\cdots \cap q_l\cap -Q_1\cdots \cap 
-Q_L\cap p_1\cdots\cap p_n\cap -P_1\cdots \cap -P_m$$
$$\cap {\sf c}_{(\Delta)}\{\bold 0\}\cap -{\sf c}_{(\Delta_1)}\{\bold 0\}\cdots \cap 
-{\sf c}_{(\Delta_N)}\{\bold 0\}.)$$

{\bf Case 2.}
$$y=1$$
This case is harder, so we start with some special cases before embarking on the most general case.
Consider the case when
$$g=q \cap p
\cap - P\cap  -{\sf c}_{(\Delta)}\{\bold 0\},$$ 
where $q\in Pl^S$, and $p,P\in Pl^<$. Then 
$${\sf c}_jg={\sf c}_j(q\cap p)\cap {\sf c}_j(q-P)\cap {\sf c}_j(q-{\sf c}_{(\Delta)}\{\bold 0\}).$$
$$\cap {\sf c}_j(p -P)\cap {\sf c}_j(p-{\sf c}_{(\Delta)}\{\bold 0\}).$$
Let us prove this special case to visualise matters. The general case will be a little bit more involved, but all the same an immediate generalization.
Let $$s\in {\sf c}_j(q\cap p)\cap {\sf c}_j(q-P)\cap {\sf c}_j(q-{\sf c}_{(\Gamma)}\{\bold 0\})$$
$$\cap {\sf c}_j(p -P)\cap {\sf c}_j(p-{\sf c}_{(\Delta)}\{\bold 0\}).$$
Then there exist $u_1, u_2, u_3, u_4, u_5 \in \F$ such that 
$$s^j_{u_1}\in q\cap p, s^j_{u_2}\in q-P, s^j_{u_3}\in q-{\sf c}_{(\Gamma)}\{\bold 0\},$$ 
$$s^j_{u_4}\in p-P\text {  and  }s^j_{u_5}\in p-{\sf c}_{(\Delta)}\{\bold 0\}.$$
Now ${\sf c}_jp\neq p$ it follows that $u_4=u_5$. 
Similarly $u_1$ and $u_2$ and $u_3$ are also equal because ${\sf c}_jq\neq q$.
But note that $s^j_{u_1}\in p\cap q$ and ${\sf c}_jq\neq q$, 
it follows that all of the $u_i$'s are in fact equal to $u$ say.
It readily follows that $$s^j_u\in q \cap p
\cap - P\cap  -{\sf c}_{(\Delta)}\{\bold 0\},$$
hence
$$s\in {\sf c}_j(q \cap p
\cap - P\cap  -{\sf c}_{(\Delta)}\{\bold 0\}).$$
The other inclusion is much easior, in fact it is absolutely straightforward.
Start with  $$s\in {\sf c}_j(q \cap p
\cap - P\cap  -{\sf c}_{(\Delta)}\{\bold 0\}).$$
Then $$s^j_u\in (q \cap p
\cap - P\cap  -{\sf c}_{(\Delta)}\{\bold 0\}).$$
It follows that 
$$s^j_{u}\in q\cap p, s^j_{u}\in q-P, s^j_{u}\in q-{\sf c}_{(\Gamma)}\{\bold 0\},$$ 
$$s^j_{u}\in p-P\text {  and  }s^j_{u}\in p-{\sf c}_{(\Delta)}\{\bold 0\}.$$
The required follows.
Now assume (still considering a special case) that 
$$g=q\cap p_1\cap\cdots \cap p_n\cap \cdots \cap -P_1\cdots \cap -P_m$$
$$\cap -{\sf c}_{(\Delta_1)}
\{\bold 0\}\cdots\cap -{\sf c}_{(\Delta_N)}\{\bold 0\},$$ 
then by the same token
 $${\sf c}_jg=\cap_{k\leq n}{\sf c}_j(q\cap p_k)\cap \cap_{i\leq m}{\sf c}_j(q-P_i)\cap_{i\leq N}{\sf c}_j(q-{\sf c}_{\Delta_N)}\{\bold 0\})$$
$$\cap_{k\leq n}((\cap_{i\leq n}{\sf c}_j(p_k\cap p_i)\cap_{i\leq m}{\sf c}_j(p_k-P_i)\cap_{i\leq N}{\sf c}_j(p_k- {\sf c}_{(\Gamma)}\{\bold 0\})).$$
To illustrate matters further, paving the way for the general case, consider another special case
that is essentially different than the one just considered.
$$g=-Q\cap p_1\cdots p_n
\cap -P_1\cdots \cap -P_m$$
$$\cap -{\sf c}_{(\Delta_1)}
\{\bold 0\}\cdots\cap -{\sf c}_{(\Delta_N)}\{\bold 0\}.$$ 
Then 
$${\sf c}_jg=\cap_{k\leq n}{\sf c}_j(p_k-Q)$$
$$\cap_{k\leq n}((\cap_{i\leq n}{\sf c}_j(p_k\cap p_i)\cap_{i\leq m}{\sf c}_j(p_k-P_i)\cap_{i\leq N}{\sf c}_j(p_k-{\sf c}_{(\Delta_i)}\{\bold 0\})).$$
Now consider the general case.
We assume for better readability that 
 $$g=p_1\cap\cdots \cap p_n\cap \cdots \cap -P_1\cdots \cap -P_m$$
$$\cap -{\sf c}_{(\Delta_1)}
\{\bold 0\}\cdots\cap -{\sf c}_{(\Delta_N)}\{\bold 0\},$$ 
where $p_i, P_i\in Pl^<\cup Pl^S$. Here we have collected the planes in $Pl^<$ and $Pl^S$ together to simplify matters.
Now we have by the same reasoning as above
$${\sf c}_jg=\cap_{k\leq n}((\cap_{i\leq n}{\sf c}_j(p_k\cap p_i)\cap \cap_{i\leq m}{\sf c}_j(p_k-P_i)$$
$$\cap_{i\leq N}{\sf c}_j(p_k-{\sf c}_{(\Delta_i)}\{\bold 0\})).$$
Now for every 
$p,q\in Pl$, there are $p',q',p''$ and $q''\in Pl^<$
such that (*)
$${\sf c}_j(p\cap q)=p'\cap q',$$ 
$${\sf c}_j(p\sim q)=p''\sim q''$$ 
and if  $j\in \Delta p\sim \Gamma,$ 
then (**)
$${\sf c}_j(p \sim {\sf c}_{(\Gamma)}\{\bold 0\})=
{}^{\alpha}\F^{(\bold 0)}\sim p(j|0)
\cup (p(j|0)\sim {\sf c}_j{\sf c}_{(\Gamma)}\{\bold 0\}).$$
Let us consider (*). We give an illustration when $\alpha=2$
and $\F=\R$. In this case elements in $Pl^<$ are either the whole plane or horizontal lines or vertical lines.
If $p$ and $q$ are any two straight lines that are not parallel, then they intersect
at a point. If we cylindrify (this point) with respect to one of the co-ordinates we get a line that is indeed parallel to one of the axis.
Next, if we have two straight lines $p$ and $q$, say, that intersect and neither is parallel to one of the axis, then $p-q$ will be $p$ ``minus" one point (the point of intersection). 
If we cylindrify with respect to one of the co-ordinates then we get the plane
minus the straight line that is parallel to one of the axis; call this straight line $p'$. 
But this is nothing more than $1-p'$.  On the other hand, if $p$ is parallel to one of the axis and we cylindrify $p$
in the direction of this axis, we get $p-q$ where $q$ is the line orthogonal to $p$ and passes by the missing point, hence $q$ is parallel to the other axis.
If $p$ and $q$ are parallel
then $p-q=p$ and so for $j\in \{0,1\}$, we have ${\sf c}_j(p-q)={\sf c}_jp\in Pl^<$.

Now we turn to (**). Assume that $p=\{s\in {}^0\F^{(\bold 0)}:  t+\sum_{i}r_is_i=0\}$. Let $j\in \Delta p\sim \Gamma$. 
Let $s\in {\sf c}_j(p-{\sf c}_{(\Gamma)}\{\bold 0\})$. Then there is some
$u\in \F$ such that $s_u^j\in p$ and $s_u^j\notin  {\sf c}_{(\Gamma)}\{\bold 0\}$.
Since $j\in \Delta p$ we have ${\sf c}_jp\neq p$ hence $r_j\neq 0$. Let $u=1/r_j(t+\sum_{j\neq i}r_is_i)$.
Then we have two cases. If $u\neq 0$, then $s\notin p(j|0)$ for else there exists $q\in p$ such that $q$ differs from $s$ in at most the $j$th component and $q(j)=0$.
But $q(j)=-1/r_0(t+\sum_{i\neq j}r_is_i)=u$. This would imply that $u=0$ which is a contradiction.
Else $u=0$. Then $s\in p(j|0)$ because $s_u^j\in p$ and $s_u^j(j)=0$.
Also $s\notin {\sf c}_j{\sf c}_{(\Gamma)}\{\bold 0\}$ for else $s_u^j\in {\sf c}_{(\Gamma)}\{\bold 0\}.$
We leave the other inclusion to the reader.
We have proved that ${\sf c}_jg\in G^{**}$.
\item $G^{**}$ is closed under substitutions.

It is enough to consider the ${\sf s}_{ij}$'s, since for all $i,j\in \alpha$
${\sf s}_i^j(x)={\sf c}_i(x\cap {\sf d}_{ij})$ i.e. 
the ${\sf s}_i^j$'s are term definable, $G^{**}$ is 
closed under the ${\sf c}_i$'s and ${\sf d}_{ij}\in G^{**}$.
For this purpose, let 
$$H= Pl^S\cup Pl^<\cup \{{\sf c}_{(\Delta)}{\bold \{\bold 0\}}:\Delta\subseteq_{\omega}\alpha\}.$$ 
Here we are not requiring that $0\in \Delta$,
so that $H$ is bigger than $G$.
However, since $G^{**}$ is closed under cylindrifications we still have
$G\subseteq H\subseteq G^{**}$. 
Now let $i,j\in \alpha$, and $x\in H$. 
Then it not hard to check that ${\sf s}_{ij}x\in G^{**}$. 
Indeed $Pl^S$ is closed under ${\sf s}_{ij}$ and so is $Pl^<$.
Assume that $y={\sf c}_{(\Delta)}\{\bold 0\}$, $\Delta\subseteq_{\omega} \alpha$,  and $i,j<\alpha$ are distinct. If $\{i,j\}\subseteq \Delta$ or $\{i,j\}\cap \Delta=\emptyset$, then
${\sf s}_{ij}{\sf c}_{(\Delta)}\{\bold 0\}={\sf c}_{(\Delta)}\{\bold 0\}$. 
Now suppose that $i\in \Delta$, $j\notin \Delta$ and that ${\sf c}_{(\Delta}\{\bold 0\}=V_1\times V_2\times V_l\ldots $ where $V_l\in \{\F, \{0\}\}$
and $V_l=\F$ only for finitely many $l$'s.
Assume further that $i<j$,  $V_i=\F$ and $V_j=\{0\}$ so that 
${\sf c}_{(\Delta}\{\bold 0\}=V_1\times\ldots  \F\times \ldots \{0\}\times \ldots.$
Then ${\sf s}_{ij}{\sf c}_{(\Delta)}\{\bold 0\}$ is obtained from ${\sf c}_{(\Delta)}\{\bold 0\}$ by interchanging
the $ith$ and $jth$ co-ordinates, that is it is equal to  $V_1\times\ldots  \{0\}\times \ldots \F\times \ldots.$, which is of the form ${\sf c}_{(\Gamma)}\{\bold 0\}$
for some $\Gamma\subseteq_{\omega}\alpha$. In fact $\Gamma=(\Delta\sim \{i\})\cup \{j\}.$
By noting that the ${\sf s}_{ij}$'s are boolean endomorphisms
we get that $G^{**}$, generated as a boolean algebra by $H$ 
(since $G\subseteq H$ is a set of generators),  
is closed under ${\sf s}_{ij}$. 
\end{enumarab}
It thus follows that $w\notin A$ and we are done.
\end{demo}

\section{The class of weakly neat atom structures is not elementary}

From the title, it is clear that there exists atom structures such that are very weakly neat. 

\begin{theorem} The class of weakly neat atom structures is not elementary, for finite and infinite dimensions.
\end{theorem}
\begin{proof} For the first part there is an atomic $\A\in \Nr_n\CA_{\omega}$ 
and an atomic $\B$ that is not a neat reduct, elementary equivalent to $\A$. $\At\A$ is srongly neat, but $\At\B$ is not.
But $\At\A\equiv \At\B$, 
because atom structure are first order interpretable in
the algebras, and we are done.
\end{proof}

The second part is harder. The construction is similar, but not identical to the one used in \cite{IGPL}. We prove our result for the harder infinite dimensional case.
The finite dimensional case can be easily destilled from the finite version of our algebras constructed in 
\cite{IGPL}.
Throughout this section $\alpha$ is a fixed arbitrary ordinal with 
$\alpha\geq \omega$ and $R$ is an uncountable set with $|R|>|\alpha|$.

\begin{athm}{Notation}
Let $s$ be an $\alpha$-ary sequence, and $i<\alpha$. Then $s(i|w)$ denotes the  
$\alpha$-ary sequence for which:
$s(i|w)(i)=w$ and $s(i|w)(x)=s(x)$ for $x\neq i$.
For a set $U$, $w\in W$ and $\beta$ an ordinal, 
we let $\bold w$ denote the $\beta$-ary sequence $\beta\times \{w\}$, i.e the constant 
sequence  for which $\bold w_i=w$, for all $i\in \beta$.
The arity of $\bold w$, we hope, will be clear from context.
\end{athm}
\begin{lemma}
There are a set $U$, $w\in U$, and $\alpha$-ary relations 
$C_r\subseteq {}^{\alpha}U^{(\bold w)}$ 
for all $r\in R$, such that conditions (i) and (ii) below hold:
\begin{enumroman}
\item For all $r\in R$, for all $i<\alpha$, for all $s\in {}^{\alpha}U^{(\bold w)}$, 
there exists $y\in U-Rgs$ such that $s(i|y)\in C_r$. 
\item The $C_r$'s ($r\in R$) are pairwise disjoint.
\end{enumroman}
\end{lemma}
\begin{athm}{Notation} Let $\bold U=\langle U, C_r\rangle_{r\in R}$, with 
$C_r\subseteq {}^{\alpha}U^{(\bold w)}$, 
be the structure described in Lemma 1*. 
Let $W=\cup_{i\in \alpha} U\times \{i\}$.
Then $W$ is simply $|\alpha|$ disjoint copies of $U$.
Let $s=\langle (w,0),\cdots,(w,j)\cdots\rangle_{j<\alpha}$. Then $s\in {}^{\alpha}W$
with  $s(i)\in W_i=U\times \{i\}$. 
(Recall that $\bold w=\langle w:i<\alpha \rangle$).
Let $V$ be the weak space ${}^{\alpha}W^{(s)}$ and let 
$C(\alpha) =\langle Sb(V),\cup,\cap,\-c_i,d_{ij}\rangle$ be the full $Ws_{\alpha}$ with unit $V$.
$S_{\alpha}$ denotes the set of one to one functions (equivalently permutations) 
which are almost everywhere identity, i.e, which are in $^{\alpha}\alpha^{(Id)}$.
\end{athm}

Let $u=\langle u_0,u_1,\cdots u_j\cdots \rangle_{j\in \alpha}\in S_{\alpha}$, and $r\in R$. Then 
$p(u,r)$ denotes the following $\alpha$-ary relation on $V$:
$$\{\langle (a_0,u_0),\cdots ,(a_j,u_j)\cdots \rangle_{j<\alpha} \in 
V:\quad\langle a_0,\cdots a_j,\cdots \rangle_{j<\alpha} \in C_r\}.$$
$$P(u)=\{p(u,r): r\in R\}.$$
Note that $P(u)$ consists of $|R|$ many $\alpha$ relations
on $V$. Let $N$ be a fixed countably infinite subset of  $R$.
Then $P_{\omega}(u)=\{p(u,r):r\in N\}$.
Again $P_{\omega}(u)$ countably infinite.
Now we define $\A(\alpha)\in Ws_{\alpha}$ and $\B(\alpha)\subseteq \A(\alpha)$ as follows:
$$\A(\alpha)=\Sg^{C(\alpha)}(\cup \{P(u):u\in S_{\alpha}\})$$
$$\B(\alpha)=\Sg^{A(\alpha)}((\cup \{P(u): u\in S_{\alpha}-\{Id\})\cup P_{\omega}(Id)).$$

\begin{theorem} $\A(\alpha)$ and $\B(\alpha)$ are atomic algebras with elementay equivalent atom structures.
$\At\B$ is not weakly neat, but $\At\A$ is.
\end{theorem}
\begin{proof}The thing to notice here, is that $\A(\alpha)$ is a disjoint union of algebras whose boolean part are isomorphic to the finite-cofinite
Boolean algebra on an uncountable set $R$. We have $^{\alpha}{\alpha}^{(Id)}$ many copies. Hence it is atomic.
Also $\B$ is atomic. We will show that that $\B\equiv \A$ , hence $\At\A\equiv \At\B$, and that any algebra having an atom 
structure isomorphic to $\At\B$ is not a 
neat reduct.
\end{proof}

We shall start by showing that $\A(\alpha)$ and $\B(\alpha)$ 
are elementary equivalent, so that their atom structures will also be elementary equivalent,
and that $\B(\alpha)$ is not a neat reduct; not only that, but any algebra based on its atom structure is not a neat reduct.

Then we shall impose further restrictions on 
the structure $\langle U, C_r\rangle_{r\in R}$, that will $force$ 
$A(\alpha)$ to be a neat reduct ( and not just a subneat reduct). 

\begin{athm}{Notation}
Let $u=\langle u_0,u_1,\cdots u_j\cdots \rangle_{j\in \alpha}\in {}^{\alpha}\alpha^{(Id)}$. 
Then $1_u$ denote the  $\alpha$-ary relation on $U$:
$$\{\langle (a_0,u_0),(a_1,u_1),\cdots (a_j,u_j)\cdots \rangle_{j<\alpha}\in V: \langle a_0,\cdots a_j\cdots 
\rangle_{j\in \alpha}\in {}^{\alpha}U^{(\bold w)}\}.$$
$L$ stands for $1_{Id}$.
\end{athm}  
\begin{theorem}
\begin{enumroman}
\item For all $u\in S_{\alpha}$, $1_u\in A(\alpha)$.
\item Let $Rl_LA(\alpha)=\{x\in A(\alpha):x\leq L\}$. Then $Rl_LA(\alpha)$ 
is atomic and its set of atoms is equal to $P(Id)$.
Furthermore, $Rl_LA(\alpha)$ is generated as a $BA$ by $P(Id)$.
\item For all non zero $a\in Rl_LA(\alpha)$, for all $i<\alpha$, $c_ia=c_iL$.
\item For all $a\in A(\alpha)$, for all $i<\alpha$, $c_ia\cap L\in \{0,L\}$.
\item Let $P=\cup \{P(u): u\in S_{\alpha}\}$. Let $p$ be a permutation of $P(Id)$. Then 
there exists an automorphism $f$ of $A(\alpha)$,
such that $p\subseteq f$ and $f(x)=x$, for each $x\in P-P(Id)$.
\end{enumroman}
\end{theorem}
\begin{demo}{Proof}
\begin{enumroman}
\item Let $u\in S_{\alpha}$. Let $r_1,r_2\in R$. Then as easily checked  
$c_0(p(u,r_1))\cap c_1(p(u,r_2))=1_{u}$.
\item Since the $C_r$'s are pairwise disjoint, we have $p(Id,r_1)\cap p(Id,r_2)=0$ for distinct $r_1$ and
$r_2$. Thus the set of atoms of $Rl_LA(\alpha)$ is precisely $P(Id)$. 
By $c_i(p(Id,r))\cap 1_{Id}=1_{Id}$ and $p(Id,r)\cap d_{ij}=0$, for distinct $i,j<\alpha$, it follows that $Rl_LA(\alpha)$
is indeed generated as a $BA$ by the set $P(Id)$. 
\item Let $a\in Rl_LA(\alpha)$ be non zero. By (ii) there exists an atom $p(Id,r)$ below $a$. 
Let $i<\alpha$. 
Then $c_i(p(Id,r))=$(by Lemma 1 (i)) $ \cup\{1_u:u\in {}^{\alpha}\alpha^{(Id)}$ and $u(j)=j$ for all 
$j\neq i\}= c_iL$. From $p(Id,r)\leq a$, (iii) readily follows. 
\item follows from (iii).
\item We first, like we did in the finite case, extend $p$ to a $BA$ automorphism $p'$, say,  
of $Rl_LA(\alpha)$. 
This is possible by (i). Then we define $f(x)=p'(x\cap L)\cup (x\cap -L)$.
The rest is completely analogous to the proof of fact 4.2.
\end{enumroman}
\end{demo}
\begin{theorem} $\B(\alpha)$ is an elementary subalgebra of $\A(\alpha)$.
\end{theorem}
\begin{proof} This can be proved using the Tarski-Vaught test.
\end{proof}

Now we show that $B(\alpha)$ is not a neat reduct. 
We proceed basically as we did in the finite case, the trick being 
``a(n) (infinite) cardinality twist" that first order logic cannot witness.

\begin{theorem} Let $\Rl_L\B(\alpha)=\{x\in B(\alpha): x\leq L\}$. 
Then $|\Rl_L\B(\alpha)|\leq \alpha$.
\end{theorem}
\begin{demo}{Proof}
First of all note that by 6.1 (i) we have $L\in B$. 
Now let $$X=\cup \{P(u):u\in S_{\alpha}-\{Id\}\}\cup P_{\omega}(Id),$$ 
$$Y=\{c_ia:a\in A\}\cup \{d_{ij}: i,j\in \alpha\}\text { , and }$$ 
$$D(\alpha)=Sg^{Bl{A}(\alpha)}(X\cup Y).$$
Then $|Y|=\alpha$. This can be seen by noting that for every $x\in Y$, if $x$ is non zero, 
and $x$ is not a diagonal element, then by Lemma 1*(ii) 
there is an $S$ with $S\subseteq {}^{\alpha}\alpha^{(Id)}$ such that $x=\cup \{1_u: u\in S\}$.
Also $D(\alpha)$ is a subuniverse of $A(\alpha)$ because of the following:
$D(\alpha)$ is closed under the boolean set operations , contains the 
diagonals, and for $a\in D(\alpha)$, 
and $i<\alpha$, we have $a\in A(\alpha)$ and so $c_ia\in D(\alpha)$, 
that is $D(\alpha)$ is closed under cylindrifications.  
Since $X\subseteq D(\alpha)$, $X$ generates $B(\alpha)$ and 
$D(\alpha)\in CA_{\alpha}$, we have $B(\alpha)\subseteq D(\alpha)$.
Let $rl_L(a)=a\cap L$. Then we have $$Rl_LD(\alpha)=\{x\in D(\alpha): x\leq L\}=
 Sg^{BlA(\alpha)}rl_L(X\cup Y).$$
But $|rl_L(X\cup Y)|=\alpha$, because $P_{\omega}(Id)$ is countable
and $|Y|=\alpha$.
Therefore $|rl_L(X\cup Y)|=\alpha$, and so $|Rl_LD(\alpha)|\leq \alpha$.
By $Rl_LB(\alpha)\subseteq Rl_LD(\alpha)$ we are done. 
\end{demo}
\begin{theorem}
Let $p[0,1]$ be the transposition on $\alpha$ 
interchanging $0$ and $1$. 
Let $v=p[0,1]\circ Id=\langle 1023\cdots \rangle$. Then $v \in S_{\alpha}$ and the 
following hold: 
\begin{enumroman}
\item $t^{B(\alpha)}(1_{v})=1_{Id}$
\item For all $\beta>\alpha$, $CA_{\beta}\models {}_{\alpha}s(0,1)c_{\alpha}x\leq t(c_{\alpha}x)$.
\item For all $\beta>\alpha$ and 
$A\in CA_{\beta}$, $_{\alpha}s(0,1)\in Ism(BlNr_{\alpha}A, BlNr_{\alpha}A)$.
\end{enumroman}
\end{theorem}
\begin{theorem}
$\B(\alpha)\notin \Nr_{\alpha}CA_{\alpha+1}$.
\end{theorem}
\begin{demo}{Proof}
If $B(\alpha)$ were a neat reduct then 
we would have $|Rl_L(B(\alpha)|=|R|$. 
But this is a contradiction.
\end{demo}

$\A(\alpha)$ constructed so far may or may not be a neat reduct. 
We now impose further restrictions on the $C_r$'s (or rather on the structure 
$\bold U=\langle U, C_r\rangle$), described in Lemma 1*, that {\it force} 
$A(\alpha)$ to be in $\cap_{k\in \omega}Nr_{\alpha}Ws_{\alpha+k}$.
These restrictions are described in Lemma 1** below.
Forming a ``limit" out of the $Ws_{\alpha+k}$'s ($k\in \omega$) , the $neat- \alpha$ 
reduct of which is $A(\alpha)$, we will show that $A(\alpha)\in Nr_{\alpha}Ws_{\alpha+\beta}$ for 
infinite $\beta$ as well.
This should be  done , of course, in such a way that does 
not interfere with what we have already 
established, namely:\par\noindent
$B(\alpha)$ is $not$ a neat reduct.\par\noindent
$A(\alpha)$ and $B(\alpha)$ are elementary equivalent.\par\noindent

To formulate Lemma 1**  we fix some needed notation:
\begin{athm}{Notation} Let $1\leq k<\omega$. Then $S(\alpha,\alpha+k)$ or 
simply 
$$S(\alpha,k)=\{i\in {}^{\alpha}(\alpha+k): \alpha+k-1\in Rgi\quad\text{and }
\{m\in \alpha:|\{i(m)\neq m\}|<\omega\}.$$ 
That is $S(\alpha,k)=\{i\in {}^{\alpha}(\alpha+k)^{(Id)}: \alpha+k-1\in Rgi\}$.
\par\noindent
$Cof^+R$ denotes the set of all non empty finite or cofinite subsets of $R$.
Let $C_r$ be an $\alpha$-ary relation symbol for all $r\in R$.
\par\noindent
For any $X\subseteq R$, $X$ finite we define the infinitary formulas:
$$\eta(X)=\lor\{C_r(x_0,\cdots x_j\cdots )_{j<\alpha}\, r\in X\},\quad\text{and }$$
$$\eta(R-X)=\land \{\neg C_r(x_0,\cdots x_j\cdots )_{j<\alpha}:r\in X\}.$$
\end{athm}
\par\noindent
$\eta(X)$ and $\eta(R-X)$ are {\it restricted} formulas, in the sense of [HMT2] sec 4.3. 
Satisfiability for such formulas (by $\alpha$-ary sequences) are defined the usual way. 
Below we shall have occasion to deal with (infinitary) formulas that are not restricted. 
These however can be obtained from restricted ones using  the (logical interpretation of the) 
$CA$ operations, i.e quantification on finitely many variables and equality. 
Now we are ready to formulate the infinite version of Lemma 1.

\begin{theorem}
There are a set $W$, an equivalence relation 
$E$ with $\alpha$-many blocks on $W$, $s\in {}^{\alpha}W$ with 
$s_i\in W_i$ the $i$th block, $i\in \alpha$, 
and $\alpha$-ary relations $C_r$ on $W$ for all $r\in R$,
such that conditions (i)-(v) below hold:
\begin{enumroman}
\item (Blocks) $C_r(w_0,w_1\cdots w_i\cdots )_{i<\alpha}$
implies  $D_E(w_0,w_1\cdots w_i\cdots )_{i<\alpha},$ for all $r\in R$ and $w_i\in W$.
Here $D_E(w_0, w_1\cdots w_i\cdots)_{i<j<\alpha}$ means that for any $i<j<\alpha$, 
$w_i,w_j$ are not $E$-equivalent i.e. they are in distinct blocks. 

\item (Symmetry) For all $f\in {}^{\alpha}W^{(s)}$ for all $r\in R$, for all permutations
$\pi\in ^{\alpha}\alpha^{(Id)}$, if $f\in C_r$ then $f\circ \pi\in C_r.$ 

\item (Bigness) For all $r\in R$, for all $i<\alpha$ 
and $v\in {}^{\alpha}W^{(s)}$ such that $D[v_i]_{i<\alpha}$
there exists $x\in W$ such that $v[i|x]\in C_r$.

\item (Saturation) For all $1\leq k<\omega$, for all 
$v\in {}^{\alpha+k-1}W^{(s)}$ one to one,  for all $x\in W$, for any
function $g:S(\alpha,k)\to Cof^+R$ 
for which $\{i\in S(\alpha,k):|\{g(i)\neq R\}|<\omega\}$, 
there is a $v_{\alpha+k-1}\in W\smallsetminus Rgv$ such that 
$xEv_{\alpha+k-1}$, i.e. $v_{\alpha+k-1}$ is in the same block as $x$,  
and  
$$\bigwedge \{D(v_{i_j})_{j<\alpha}\implies \eta(g(i))[\langle v_{i_j}\rangle]: 
i\in S(\alpha,k)\}.$$
\item (Disjointness) The $C_r$'s are pairwise disjoint.
\end{enumroman}
\end{theorem}
\begin{demo}{Proof}
Let $k\geq 1$. For brevity set 
$$Y(\alpha,R,k)=\{g\in {}^{S(\alpha,k)}Cof^+R: \{i\in S(\alpha,k):|\{g(i)\neq R\}|<\omega\}.$$
Let
$$Q=\cup\{I(^{\alpha+k-1}|R|)\times Y(\alpha,R,k):1\leq k<\omega\},$$
then $|Q|=|R|=\kappa$, say. Here $I(^{\alpha+k-1}|R|)$ stands for 
the set of all one to one functions
from $\alpha+k-1$ into $R$. 
Let $\rho$ be an enumeration of $Q$ such that:
for all $l<\kappa$, for all $q\in Q$, there exists $j$ with $l<j<\kappa$ such that $\rho(j)=q$.
Fix a well ordering $\prec$ of $R$.
let $l<\kappa$ and suppose that for all $\mu<l$ we have already defined the element
$x_{\mu}$, and the $\alpha$-ary relation 
$C_r^{\mu}\subseteq {}^{\alpha}W_{\mu}^{(s)}$, where $W_{\mu}=\{x_k:k<\mu\}$.
Assume that $\rho(l)=\langle \langle \beta_j \rangle_{j<\alpha+k-1},f\rangle$.
Then $\beta_j\in |R|$ for all $j<\alpha+k-1$, and $f\in Y(\alpha,R,k)$.
Let $x_l$ be an element not in $\cup \{W_{\mu}:\mu<l\}$.
If there exists $\mu<\alpha+k-1$ such that $l\leq \beta_{\mu}$, 
then for all $r\in R$ we define
$C_r^l=\cup \{C_r^{\mu}: \mu<l\}$. Else,  $l>\beta_{\mu}$ for all $\mu<\alpha+k-1$.
Let $x_{\mu}=x_{\beta_{\mu}}$ for $\mu<\alpha+k-1$ and let $x_{\alpha+k-1}=x_l$.
For all $r\in R$ and $ i\in S(\alpha,k)$ we let 
$$\langle u_{i_j}\rangle_ {j<\alpha}\in X_r^l\text { iff }r \text { is the }\prec-
\text { least element of }f( i)$$
and
$$C_r^l=\cup \{C_r^i:i<l\}\cup \{\langle t_{\pi(0)},t_{\pi(1)},\cdots \rangle:t\in X_r^l\text 
{ and }\pi\in S_{\alpha}\}.$$
Finally set 
$$W=\cup \{W_l:l<\mu\}\quad\text{and } C_r=\cup\{C_r^l:l<\mu\}.$$ 
It is not hard to check just like we did in the finite dimensional case,  that 
$\bold W=\langle W,C_r\rangle_{r\in R}$ is as desired.
\end{demo}

\par\noindent

\bigskip

Let $\bold W=\langle W, C_r\rangle$ and $s\in {}^{\alpha}W$ 
be as specified above
Then $W$ is a disjoint union of $\alpha$ copies. We let $W_i$ denote the $i$th copy.
We shall construct our algebras having unit $^{\alpha}W^{(s)}.$
Recall that $S_{\alpha}$ stands for the set of permutations in $^{\alpha}\alpha^{(Id)}$.
Let $u\in S_{\alpha}$ and $r\in R$. Then we set
$$p(u,r)=C_r\cap W_{u_0}\times W_{u_1}\times \cdots W_{u_i}\times\cdots \cap {}^{\alpha}W^{(s)}.$$
$$P(u)=\{p(u,r): r\in R\}.$$
Let $N$ be a fixed countable subset of $R.$
Let $$P_{\omega}(u)=\{p(u,r):r\in R\}.$$
With a slight abuse of notation we let $A(\alpha)$ and $B(\alpha)$ be the (new)  algebras 
consructed out of the (new) structure  $\bold W=\langle W, C_r\rangle_{r\in R}$, 
described in the previous lemma
That is $A(\alpha)\in Ws_{\alpha}$ and $B(\alpha)\subseteq A(\alpha)$ 
are defined as follows:
$$ A(\alpha)=Sg^C(\cup \{P(u): u\in S_{\alpha}\});$$
and
$$B(\alpha)=Sg^C((\cup \{P(u): u\in S_{\alpha}\smallsetminus \{Id\})\cup P_{\omega}(Id)).$$
Here $C$ is the full weak set algebra with greatest element $^{\alpha}W^{(s)}$, with
$s$ as specified in Lemma $1^{**}$. Recall that $s(i)\in W_i$, the $i$ th copy.
It is easy to check that  $A(\alpha)$ and $B(\alpha)$, so defined,
are (still) elementary equivalent (copy the proof of facts 6.1-2) 
and $B(\alpha)$ is (still) not a neat reduct (copy the proof of facts 6.3-6.5).  
Like in the finite case, we will use the new condition of 
``saturation" or  ``elimination of quantifiers" 
expressed in Lemma $1^{**}$ (iv) to  show that 
$A(\alpha)\in \cap_{k\in \omega}Nr_{\alpha}Ws_{\alpha+k}$. 
To ``lift" $A(\alpha)$ to arbitrary finite extra dimensions, we shall further need:

\subsubsection*{Some more Definitions.} Let $0\leq k<\omega$. Then we let
$W_k=W=\cup_{i<\alpha} W_i$. Fix $w\in W$. 
Let $s_k=s\cup \{\langle i,w\rangle: \alpha\leq i<(\alpha+k) \}$. 
Then $s_k\in {}^{\alpha+k}W$. (When $k=0$, then $s_k$ is just $s$).
Let $V_k$ be the weak space $^{\alpha+k}W^{(s_k)}$ and let 
$C(\alpha+k)$ be the full $Ws_{\alpha+k}$ with unit $V_k$. \par\noindent
Let $u\in {}^{\alpha+k}\alpha^{(Id)}$. 
Then 
$$F_k(u)=\{C_r(x_{i_0},\cdots x_{i_{j}}\cdots )_{j<\alpha}, \neg C_r(x_{i_0},\cdots x_{i_{j}}\cdots )_{j<\alpha}: 
r\in R,\  i\in {}^{\alpha}\alpha+k^{(Id)}$$
$$\text{and }
\langle u_{i_0},\cdots u_{i_j}\cdots \rangle_{j<\alpha}\quad\text{is one to one }\}\quad\cup Eq(u).$$
Recall that $Eq(u)$ stands for $\{x_i=x_j: x_i\neq x_j: (i,j)\in ker(u)\}$.  
Like the finite case we assume  that $\{T,F\}\in F_k(u)$. We put
$$ F_k(u)^*=\{\land J: J\subseteq_{\omega} F_k(u)\}\text { and } $$ 
$$F_k(u)^{**}=\{\lor J: J\subseteq_{\omega} F_k(u)^*\}.$$
Let $u\in {}^{\alpha+k}\alpha^{(Id)}$ and $\phi\in F_k(u)^{**}$. 
\par\noindent
Then we let $E(u,\phi)$ denote the following $(\alpha+k)$-ary relation on $V_k$
$$\{s\in V_k: s_j\in W_{u_j} \text { for all $j<\alpha+k$ and } \bold W\models 
\phi[\langle s_j\rangle_{j<\alpha+k}]\}.$$ 
As noted before, satisfiability is defined the usual way. 
In particular, when $k=0$, $u\in S_{\alpha}$ 
and $r\in R$, then $E(u,C_r(x_0,x_1,\cdots )$ is the same as $p(u,r)$.
\par\noindent
$ind(\phi)$ denotes the set $\{i\in \alpha+k:x_i\in var(\phi)\}$. 
\par\noindent
We let $G(k)=\{c_{(\Delta)}E(u,\phi), -c_{(\Delta)}E(u,\phi):\Delta\subseteq_{\omega} \alpha+k$, 
$u\in {}^{\alpha+k}\alpha^{(Id)},\phi\in F_k(u)^{**},\ \Delta\cap ind(\phi)=0\}$ 
$\cup\{d_{ij}^{V_k},-d_{ij}^{V_k}:i,j\in \alpha+k\}$.
\par\noindent
Here $c_{(\Delta)}$, the unary operation referred to as generalized 
cylindrification  in \cite{HMT1}.  
That is $c_{\emptyset}x=x$ and if $\Delta=\{k_0,\cdots k_{m-1}\}$  
is a non-empty finite subset of $\alpha+k$,  
then $c_{(\Delta)}x=c_{k_0}\cdots c_{k_{m-1}}x$.  
\par\noindent
For $i,j\in \alpha+k$, $d_{ij}^{V_k}$, 
or simply $d_{ij}$ is the diagonal element  $\{s\in V_k:\ s_i=s_j\}$.
\par\noindent
Forming finite intersections of elements in $G(k)$ we let
$$G(k)^*=\{\cap J: J\subseteq_{\omega} G(k)\}$$  
and forming finite unions of elements in $G(k)^*$, we finally let
$$G(k)^{**}=\{\cup J: J\subseteq_{\omega} G(k)^*\}.$$
We will show that $G(k)^{**}$ is a $Ws_{\alpha+k}$ (with unit $V_k$), that 
$G(0)^{**}=Nr_{\alpha}G(k)^{**}$ for all $k\in \omega$, and finally that
$G(0)^{**}=A(\alpha)$. This will show that  
$A(\alpha)\in \cap_{k\in \omega}INr_{\alpha}Ws_{\alpha+k}$.
To prove that $G(k)^{**}$ is a $Ws_{\alpha+k}$ we shall need:

\begin{athm}{Fact 6.6}
\begin{enumroman}
\item Let $u\in {}^{\alpha+k}\alpha^{(Id)}$, 
$\Delta \subseteq_{\omega} \alpha+k$ and $\phi\in F_k(u)^{**}$.
Then $c_{(\Delta)}E(u,\phi)\in G(k)^{**}$.
\item Let $g\in G(k)^{**}$ and $i,j\in \alpha+k$ be distinct.  Then $s_i^j(g)=c_j(g\cap d_{ij})\in G(k)^{**}$.
\item Let $H(k)=\{g\in G(k): g=c_{(\Delta)}E(u,\phi):\Delta\subseteq_{\omega}\alpha+k,u\in {}^{\alpha+k}\alpha^{(Id)}$ 
and $\phi\in F_k(u)^{**}\}$. Then $H(k)$ is closed under finite intersections.
\end{enumroman}
\end{athm}
\begin{demo}{Proof}
\begin{enumroman}
\item We first prove (i) when $\Delta$ consists of one element $j$, say.
That is we will show that for $u$ and $\phi$ as indicated above 
and $j<\alpha+k$ we have $c_jE(u,\phi)\in G(k)^{**}$.
We can and will assume that $\phi\in F_k(u)^*$ since 
$$c_jE(u,\phi_1 \lor\cdots \phi_n)=c_jE(u,\phi_1)\cup\cdots c_jE(u,\phi_n)$$
and $G(k)^{**}$ is closed under finite unions.
Then $\phi=\land J$, where $J\subseteq_{\omega} F_k(u)$.
Next, we  break up $\phi$ into two parts just like we did in the finite dimensional case; that is let
$J^+=\{\psi\in J: x_j\in var(\psi)\}$ and $J^-=\{\psi\in J:x_j\notin var(\psi)\}$.
For brevity set $\phi^+=\land J^+$ and  $\phi^-=\land J^-$,  where by the empty conjunction we understand the formula 
$T$. Now  $\phi=\phi^+\land \phi^-$, $\phi^-\in  F_k(u)^*$
and $x_j$ occurs only in $\phi^-$.
\begin{athm}{Notation} For $i,l\in \alpha+k$ and $\phi\in F_k(u)^{**}$ $\phi[i|l]$ 
stands for the formula obtained from $\phi$ by replacing all
(free and bound) occurences of $x_i$ in $\phi$ by $x_l$. 
\end{athm}
\noindent{\bf Case (a): }
$x_j=x_m\notin J^+$ and $x_j=x_m\notin J^+$ for all $m\in \alpha+k-\{j\}$.\par\noindent
Then either $\bold W\models \phi^+\equiv F$, in which case 
$\bold W\models \exists x_j \phi^+\equiv F$ , and so $c_jE(u,\phi)=c_jE(u,F)=0\in G(k)^{**}$.
Else  $\phi^+$ is satisfiable, in which case , by applying the saturation condition described
in  Lemma $1^{**}$, we get
$$\bold W\models (\exists x_{\alpha+k-1}\phi^+[j|\alpha+k-1])\equiv T,$$ 
which is the same as  $\bold W\models (\exists x_j\phi^+) \equiv T$, 
and so $$c_jE(u,\phi)=c_jE(u,(\exists x_j\phi^+) \land \phi^-)=c_jE(u,\phi^-).$$
The latter is in $G(k)^{**}$, since now  $x_j$ does not occur in $\phi^-$, 
and $\phi^-\in F_k(u)^*$. 
Note that $\phi^-\in F_k(v)^*$ for any $v$ that differs from $u$ in at most the $j$-th place;
and for any such $v$ we have $c_jE(u,\phi)=c_jE(v,\phi^-)$.
\par\noindent{\bf Case (b): }$x_j=x_m\in J^+$ or $x_m=x_j\in J^+$ for some $m\in \alpha+k$ and $m\neq j$.
\par\noindent
Assume that $x_j=x_m\in J^+$. Then , by definition of $F_k(u)$, we have $u_j=u_m$ . 
Also $\phi=(\land I)\land x_j=x_m$, for some $I\subseteq_{\omega} F_k(u)$.
For brevity let $\psi=\land I$. Then of course $\psi\in F_k(u)^*$. 
Computing we get  $$c_jE(u,\phi)=c_jE(u,\psi\land x_j=x_m)=c_jE(u,\psi[j|m]).$$
The latter is in  $G(k)^{**}$ since $x_j$ does not occur in $\psi[j|m]$, 
and $\psi[j|m]\in F_k(u)^*$ because $\psi \in F_k(u)^*$ and $u_j=u_m$.
In fact, it is not hard to see that $\psi[j|m]\in F_k(v)^{*}$, for any $v$ that differs from $u$ in at most the $j$ th place.
The case when $\Delta$ is an arbitrary finite subset of $\alpha+k$, 
follows  by using the above reasoning eliminating the variables occuring in $\phi$, 
with indices in $\Delta$, one by one. Alternatively one can use 
a straightforward induction on $|\Delta|$.
Note that we have actually proved the following infinite analogue of (**) in 
fact 3.1 way above: 
\par\noindent
\begin{athm}{(***)}Let $u\in {}^{\alpha}\alpha+k^{(Id)}$ and $\phi\in F_k(u)^{**}$. Let $\Delta=\{i_0,\cdots i_{n-1}\}$ be a finite subset 
of $\alpha+k$. Let $v\in {}^{\alpha}\alpha+k^{(Id)}$ be such that $v(j)=u(j)$ for all $j\notin \Delta$.
Then there exists $\psi\in F_k(v)^{**}$ such that $var(\psi)\cap \{x_{i_0},\cdots ,x_{i_{n-1}}\}=\emptyset$. 
i.e., $ind(\psi)\cap \Delta=\emptyset$, and $\bold W\models \psi\equiv \exists x_{i_0}\cdots \exists x_{i_{n-1}}\phi$.
In particular, $c_{(\Delta)}E(u,\phi)=c_{(\Delta)}E(v,\psi)$.
\item  Let $g\in G(k)^**$ and $i,j\in \alpha+k$ be distinct.
\end{athm}

{\bf Case (a): }$g=c_{(\Delta)}E(u,\phi)$ with $u\in {}^{\alpha+k}\alpha^{(Id)}$, $\phi\in F_k(u)^{**}$. 
\par\noindent
{\bf Subcase (a) : }
$j\in \Delta$.\par\noindent 
Then by \cite{HMT1} 1.5.8 (i) and 1.7.3 we have 
$$s_i^j(g)=s_i^jc_j(c_{(\Delta-\{j\})})E(u,\phi)=
c_j(c_{(\Delta-\{j\})})E(u,\phi)=g$$
\par\noindent
{\bf Subcase (b) : }
$j\notin \Delta$ and  $i\notin \Delta$. \par\noindent
Then by \cite{HMT1}, 1.5.8 (ii)  we have 
$$s_i^jg=s_i^jc_{(\Delta)}E(u,\phi)=c_{(\Delta)}s_i^jE(u,\phi).$$
Let $g^*=E(u,\phi)\cap d_{ij}$. Then $g^*$ and hence $s_j^ig=c_{(\Delta)}c_jg^*$, 
is equal to zero
if $u_i\neq u_j$.
Else, as easily checked,  $g^*=E(u,\phi \land x_i= x_j)$.  
By $u_i=u_j$ and $\phi\in F_k(u)^{**}$, 
we have $\phi \land x_i= x_j\in F_k(u)^{**}$.
By fact 6.6 (i)  we get 
that $$s_j^ig=c_{(\Delta)}c_jg^*=c_{(\Delta)}c_jE(u,\phi\land x_i=x_j)\in G(k)^{**}.$$
\par\noindent
{\bf Subcase (c) : }
$j\notin \Delta$ and $i\in \Delta$.\par\noindent
By \cite{HMT1} 1.7.3 and 1.5.8 (ii) we have 
 $$s_i^jc_{(\Delta)}E(u,\phi)=s_i^jc_{(\Delta-\{i\})}c_iE(u,\phi)
=c_{(\Delta-\{i\})}s_i^jc_iE(u,\phi).$$
Let $g^*=s_i^jc_iE(u,\phi)$. Then by \cite {HMT1} 
1.5.1 we have $g^*=c_j(c_i(E(u,\phi)\cap d_{ij})$.
Let $v=u_{u(j)}^{i}=u\circ [i|j]$. 
Here and elsewhere $[i|j]$ denotes the replacement that maps $i$ to $j$
and otherwise coincides with the identity. Then $v(j)=u(j)=v(i)$. Also, it is easy to see that 
$$c_i(E(u,\phi)\cap d_{ij})=E(v,\phi\land x_i=x_j).$$
Since $i\in \Delta$, we can assume by fact 6.6 (i)  that $x_i$ does not occur in $\phi$. 
Since $v$ differs from $u$ in at most the $ith$ 
place, we have $\phi\in F_k(v)^{**}$. But  $v(i)=v(j)$, and so  $x_i=x_j\in F_k(v)^{**}$, thus 
$\phi\land x_i=x_j\in F_k(v)^{**}$. By fact 6.6 (i)  we get that $g^*$ hence  $g$ is in $G(k)^{**}$.
\par\noindent
{\bf Case (b) : }$g=d_{kl}$, where $k,l\in \alpha+k.$ 
\par\noindent
By \cite{HMT1}, 1.5.4 we have $s_i^k(d_{kl})=d_{il}$ if $k\neq l$ and  
$s_i^jd_{kl}=d_{kl}$ if $j\notin \{k,l\}$. In either case we have $s_i^jg\in G(k)^{**}$.  
\par\noindent
{\bf Case (c) : }$g$ is an element in $G(k)^{**}$.
\par\noindent
Follows from the two previous cases, since the $s_i^j$'s ($i,j\in \alpha+k$) 
are $BA$ endomorphisms of the boolean algebra $G(k)^{**}$ (cf. \cite{HMT1} 
1.5.3 (i)). 
\item Let $g=c_{(\Delta_1)}E(u_1,\phi_1)\cap c_{(\Delta_2)}E(u_2,\phi_2)$
where $\Delta_i\cap ind(\phi_i)=0$ for $i\in \{1,2\}$.
Then $g=0$ if there exists $j\in (\alpha+k)-(\Delta_1\cup \Delta_2)$ such that $u_1(j)\neq u_2(j)$.
Else, it is not difficult to check that $g= c_{(\Delta_1\cap \Delta_2)}E(w,\phi_1\land \phi_2)$, 
where $w$ is defined as follows:
$$w(j)=u_1(j)=u_2(j)\text{ for }j\in \alpha+k-(\Delta_1\cup \Delta_2)$$
$$w(j)=u_1(j)\text{ if }j\in \Delta_2-\Delta_1$$
$$w(j)=u_2(j)\text{ if }j\in \Delta_1-\Delta_2,\text{ and }w\text{ is defined arbitrarily on }\Delta_1\cap \Delta_2.$$
Also by $ind(\phi_i)\cap \Delta_i=0$, for $i\in \{1,2\}$ and $w(j)=u_1(j)=u_2(j)$ for all 
$j\in \alpha+k-(\Delta_1\cup \Delta_2)$ we can assume by 6.6 (i) 
that $\phi_i\in F_k(w)^{**}$, 
hence $\phi_1\land \phi_2\in 
F_k(w)^{**}$. 
\end{enumroman}  
\end{demo}
To economise on notation we recall a few notions from \cite{HMT1}. 
A {\it generalized diagonal element} is a finite intersection of diagonal elements.
A {\it co-diagonal element} is the complement of a diagonal element. 
A {\it generalized co-diagonal element} is a finite intersection
of co-diagonal elements. In particular the unit of an algebra can be 
viewed as both a generalized diagonal and a 
generalized co-diagonal. 
\begin{athm}{Fact 6.7}
For all $0\leq k<\omega$, $G(k)^{**}\in Ws_{\alpha+k}$.
\end{athm}
\begin{demo}{Proof}
Clearly  $G(k)^{**}$ is a boolean field of sets with greatest element $V_k$. 
Also $G(k)^{**}$, by definition, contains all diagonal elements.
We are  thus left to check cylindrifications, by the additivity of which, 
it suffices  to show that for every  $g\in G(k)^*$ and $j<\alpha+k$ we have $c_jg\in G(k)^{**}$.
By fact 6.6 (iii) a typical element $g$ of $G(k)^*$ is of the form
\par\noindent
$a_0\cap b\cap c\cap -a_1\cap -a_2\cdots \cap -a_n$, where $n\in \omega$, 
$a_0,a_1,\cdots a_{n}\in H(k)$,
$b$ is a generalized diagonal element, and $c$ is a generalized co-diagonal element.
Now fix  $j<\alpha+k$ and let $g$ be as indicated above. 
\par\noindent
{\bf Case (a) : }
$b=d_{jl}\cap b'$ where $b'$ is a generalized diagonal element and $j\neq l$.
\par\noindent
Let $g^*=a_0\cap b'\cap c\cap -a_1\cdots \cap -a_n$. 
Then, of course $g^*\in G(k)^*$ and $c_jg=s_l^jg^*$.
\par\noindent
By fact 6.6 (i) the latter is in $G(k)^{**}$.
\par\noindent
{\bf Case (b) : }
$a_0=b=V_k$.
\par\noindent
In this case we have 
 $g= -c_{(\Delta_1)}E(u_1,\phi_1)\cdots \cap -c_{(\Delta_n)}E(u_n,\phi_n)
\cap -d_{jl_1}\cdots \cap -d_{jl_m}$, where $n,m\in \omega$, 
$\Delta_k\subseteq_{\omega}\alpha +k$ for $1\leq k\leq n$ 
and since $c_j(x\cap c_jy)=c_jx\cap c_jy$, we can assume 
that $j\notin \Delta_1\cup \cdots \Delta_n$. Assume that $g\neq 0$, 
for else there is nothing more to prove. In particular, $j,l_1,\cdots l_m$ are pairwise distinct. 
We claim that in this case we have   $c_jg=V_k$.
\par\noindent

Indeed, let $s=\langle s_l\rangle_{l\in \alpha+k}$ be an arbitrary element in $V_k$.
We will show that $s\in c_jg$, by which we will be done.
Choose $k_j\notin \{u_1(j),\cdots ,u_n(j)\}\cup \{i_{l_1},\cdots i_{l_m}\}$, and let  $u$
be an arbitrary element in $W_{k_j}.$
Let  $z=s(j|u)$. Assume that  $1\leq i\leq n$. 
Then $z\notin c_{(\Delta_i)}E(u_i,\phi_i)$ because
$z(j)=s(j|u)(j)=u\notin \cup_{i<n}W_{u_i(j)}$.
Assume now that  $1\leq i\leq m$. 
Then $z(j)\neq z(l_i)$, since they are in distinct copies by the choice of $k_j$,
i.e. $z\notin d_{jl_i}$.
We have shown that $z\in g$. Since $s$ differs fom $z$ in at most the $jth$ place we get
that $s\in c_jg$.  We have proved that $c_jg=V_k$.
\par\noindent
{\bf Case (c) : }
$$g=c_{(\Delta_0)}E(u_0,\phi_0)\cap -c_{(\Delta_1)}
E(u_1,\phi)\cap\cdots c_{(\Delta_n)}E(u_n,\phi_n)
\cap -d_{jl_1}\cdots\cap -d_{jl_m},$$ 
where $n,m\in \omega$, $\Delta_k\subseteq_{\omega}\alpha+k$ for 
$0\leq k\leq n$
and as in the previous case, we can assume that $j\notin \Delta_0\cup\cdots\Delta_n$ 
and that $j,l_1,\cdots l_m$ 
are pairwise distinct.\par\noindent
We can further assume that $\{l_1,\cdots l_m\}\subseteq \Delta_0$, for if $l_k\notin \Delta_0$ 
for some $1\leq k\leq n$,
then we can simply ``ignore" the co-diagonal element with indices $j$ and $l_k$, 
i.e. the element $-d_{jl_k}$,
because of the following:
$$c_{(\Delta_0)}E(u_0,\phi_0)\cap -d_{jl_{k}}=c_{(\Delta_0)}E(u_0,\phi_0\land x_j\neq x_{l_k})$$ 
if $u_0(j)=u_0(l_k)$,
and is equal to $c_{(\Delta_0)}E(u_0,\phi_0)$ otherwise.
Note too, that $\phi_0\land x_j\neq x_{l_k}\in F_k(u_0)^{**}$, 
whenever $\phi_0\in F_k(u_0)^{**}$ 
and $u_0(j)=u_0(l_k)$. In either case, $g$ can be written in the form:
$$c_{(\Delta_0)}E(u_0,\psi)\cap \cdots -c_{(\Delta_n)}E(u_n,\phi_n)\cap \cap \{-d_{jl}:l\in \{l_1,\cdots l_m\}-{l_k}\},$$ 
where $c_{(\Delta_0)}E(u_0,\psi)\in H(k)$.
\par\noindent
Now let 
$$\Gamma =\Delta_0-\{l_1,\cdots l_m\}\quad\text{and }\Gamma_1=\{l_1,\cdots l_m\}.$$ 
Then $\Delta_0=\Gamma\cup \Gamma_1$.
Let $S=\{j\}\times \Gamma_1=\{(j,l): l\in \Gamma_1\}$ and let  $d_{S}$  be the co-diagonal element:
$\cap \{-d_{jl}: (j,l)\in S\}$.\par\noindent 
Let $g*=-c_{(\Delta_1)}E(u_i,\phi_1)\cap\cdots -c_{(\Delta_n)}E(u_n,\phi_n)$,
$g_1= c_{(\Gamma)}E(u_0,\phi_0)\cap g*$ and 
$g_2= c_{(\Gamma_1)}E(u_0,\phi_0)\cap d_{S}\cap g*$.
Then by \cite {HMT1} 1.7.3, and 
the additivity of $c_j$ we have $c_jg=c_jg_1\cup c_jg_2$. 
\par\noindent
We shall show that $c_jg_1\in G(k)^{**}$ 
and $c_jg_2\in G(k)^{**}$, by which we will be done.
\par\noindent
{\it Proof of $c_jg_1\in G(k)^{**}$}.
\par
Write $g_1=c_{(\Gamma)}E(u_0,\phi_0)\cap\cdots -c_{(\Delta_n)}E(u_n,\phi_n)$.
\par\noindent
We can assume that for $1\leq k\leq n$ and $i\notin \Gamma \cup \Delta_k$, 
we have $u_0(i)=u_k(i)$, for if 
$u_0(i)\neq u_k(i)$ for some $1\leq k\leq n$, and some $i\notin \Gamma\cup \Delta_k$, then
$c_{(\Gamma)}E(u_0,\phi_0)\cap -c_{(\Delta_k)}E(u_k,\phi_k)=c_{(\Gamma)}E(u_0,\phi_0)$.
\par\noindent
For the time being assume that $\phi_1=\phi_2=\cdots \phi_n=T$.
\par\noindent
{\it Claim} 
\par\noindent
$c_jg_1=c_jc_{(\Gamma)}E(u_0,\phi_0)\cap -c_jc_{(\Delta_1)}E(u_1,T)\cap\cdots -c_jc_{(\Delta_n)}E(u,T)$.
\par\noindent
{\it Proof of claim}
\par
It is straightforward to show that r.h.s is contained in $c_jg$. Now for the opposite inclusion:

Let $s\in c_jg_1$. Then there exists $b\in W$ such that
$s(j|b)\in c_{(\Gamma)}E(u_0,\phi_0)$ and $s(j|b)\notin c_{(\Delta_k)}E(u_k,T)$ 
for all  $1\leq k\leq n$.
Thus $s\in c_jc_{(\Gamma)}E(u_0,\phi_0)$. 
We will show that $s\notin c_jc_{(\Delta_k)}E(u_k,T)$ for
all $1\leq k\leq n$, by which case we will be done. 
Towards this end, fix $1\leq k\leq n$ and assume to the contrary that
$s\in c_jc_{(\Delta_k)}E(u_k,T)$. Then there exists $c\in W$ 
such that $s(j|c)\in c_{(\Delta_k)}E(u_k,T)$.
By $s_c^j\in c_{(\Delta_k)}E(u_k,T)$ and $j\notin \Delta_k$
we get $c=u_k(j)$. Similarly, by $s_b^j\in c_{(\Gamma)}E(u_0,\phi_0)$ 
and $j\notin \Gamma$ we get $b=u_0(j)$.
Recall that we assumed that $u_0(j)=u_k(j)$ 
since $j\notin \Gamma\cup \Delta_k$, thus $b=c$.
Since $s(j|b)$ and $s(j|c)$ differ in at most the $jth$ place , 
$b=c$ and $s(j|c)\in c_{(\Delta_k)}E(u_k,T)$ we get
$s(j|b)\in c_{(\Delta_k)}E(u_k,T)$. Contradiction proving  that 
$s$ is as required.
\par\noindent
Now the general case follows from this by noting that 
$$-c_{(\Delta)}E(u,\phi)=c_{(\Delta)}E(u,\neg \phi)\cup -c_{(\Delta)}E(u,T),$$ 
and that $H(k)$ is closed 
under finite intersections.
\par\noindent
We now turn to showing that $c_jg_2\in G(k)^{**}$. 
This will be done by showing that 
$a=c_{(\Gamma_1)}E(u_0,\phi_0)\cap d_S$, hence $g_2=a\cap g*$, can be written 
as a finite  union of elements of the form 
$a_0\cap -a_1\cdots \cap -a_n$ where $a_i\in H(k)$.
Then by the additivity of $c_j$, and the previous case we get the desired.
We start off when $\Gamma_1$ consists of a single element $m$ say.
We compute $a=c_mE(u_0,\phi_0)\cap -d_{jm}$. 
For brevity set $u=u_0$. Let $v=u(m|u(j))=u\circ [m|j]$. 
Then $v(m)=v(j)=u(j)$.
Moreover it is not hard to check that 
$$a=(c_mE(u,\phi_0)\cap -E(v,T))\cup E(v,\phi_0\land x_j\neq x_m).$$
Since $x_m$ does not occur in $\phi_0$, and $v$ 
differs from $u$ in at most the $m$-th place, we infer from fact 6.6 (i) 
that $\phi_0\in F_k(v)^{**}$. By $v(m)=v(j)$ we get that $x_j\neq x_m$ is 
also in $F_k(v)$ hence in $F_k(v)^{**}$.
It follows that  $\phi_0\land x_j\neq x_m\in F_k(v)^{**}$. We have shown 
that $E(v,\phi_0\land x_j\neq x_m)\in H(k)$.
Thus $a$ is as desired, since  $c_mE(u,\phi_0)$ and $E(v,T)$ are in $H(k)$, too. 
Now for the general case. Assume that $\Gamma_1=\{l_1,\cdots l_m\}$. 
Having treated the case when $m=1$
we now assume that $m>1$. 
Then $$a=c_{(\Gamma_1)}E(u,\phi)\cap -d_{S}=\cap \{a_k: 1\leq k\leq m\}$$ where 
$a_k=c_{(\Gamma_1-\{\l_k\})}c_{l_k}E(u,\phi)\cap -d_{jl_k}$. 
Therefore  $a$, hence $g_2= a\cap g^*$, is indeed the finite union of 
elements of the $a_0\cap -a_1\cdots \cap -a_n$ where $a_i\in H(k)$.
Thus $c_jg_2$  is in $G(k)^{**}$.  
By this the proof of fact 6.7 is complete.
\end{demo}
\begin{athm}{Fact 6.8}
$G(0)^{**}\cong Nr_{\alpha}G(k)^{**}$, for all $0\leq k<\omega$.\footnote {We should point out that the complete 
analogue of fact 3.3 is true: 
That is for $0\leq l<k<\omega$ we have $G(l)^{**}\cong Nr_{\alpha+l}G(k)^{**}$. Fact 6.8
is the  special case when $l=0$, which is all what we need to show that 
${\bf A}(\alpha)\in Nr_{\alpha}Ws_{\alpha+k}$, for all $k<\omega$.}  
\end{athm}
\begin{demo}{Proof}
Let $k\geq 0$. Define $i(k)$ like in fact 3.3.
That is for $a\in G(0)^{**}$, let $i(k)(a)=\{t\in V_k :t\upharpoonright \alpha\in a\}$.
Then $i(k)\in Ism(A(\alpha), Nr_{\alpha}C(\alpha+k))$, where $C(\alpha+k)$ is the full 
$Ws_{\alpha+k}$ with unit $V_k$.
We will show that $i(k)G(0)^{**} \subseteq G(k)^{**}$.
For that it clearly suffices to show that for all $u\in {}^{\alpha}\alpha^{(Id)}$, and 
$\phi\in F_0(u)$, we have $i(k)E(u,\phi)\in G(k)^{**}$. Let $u$ and $\phi$ be as indicated:
Then $i(k)E(u,\phi)=\cup\{E(v,\phi):v\in {}^{\alpha+k}\alpha^{(Id)}:v\upharpoonright \alpha=u\}
=c_{\alpha}\cdots c_{\alpha+k-1}E(u*,\phi)$ where $u*=u\cup 
\{\langle i,0\rangle:\alpha\leq i< \alpha+k\}$. 
Since $F_0(u)\subseteq F_k(v)$ whenever $v\in {}^{\alpha+k}\alpha^{(Id)}$ is such that 
$v\upharpoonright \alpha=u$, 
we get $\phi\in F_k(u*)^{**}$, hence $i(k)E(u,\phi)\in G(k)^{**}$.
We have shown that $i(k)\in Ism(G(0)^{**}, Nr_{\alpha}G(k)^{**})$.
We will now show that $i(k)$ is actually onto $Nr_{\alpha}G(k)^{**}$.
Since the $c_i$'s are additive, it suffices to show that for all $g\in G(k)^*$,
there exists $a\in G(0)^{**}$, such that $i(k)a=c_{\alpha}\cdots c_{\alpha+k-1}g$.
\par\noindent
{\bf Case (a): }
$g=E(v,\phi)$, $v\in {}^{\alpha+k}\alpha^{(Id)}$, and $\phi\in F_k(v)^{**}$. 
\par\noindent
Then by fact 6.6 (i)   there exists $\psi\in F_k(v)^*$ such that 
$var(\psi)\cap \{x_{\alpha},\cdots x_{\alpha+k-1}\}=\emptyset$ and
$\bold W\models \exists x_{\alpha}\cdots \exists x_{\alpha+k-1}\phi\equiv \psi$.
Let $u=v\upharpoonright \alpha\in {}^{\alpha}\alpha^{(Id)}$.
By noting that $F_0(u)^*=\{\phi\in F_k(v)^*: var(\phi)\subseteq \alpha\}$, 
we get that $\psi\in F_0(u)^*$. Moreover we have  
$i(k)E(u,\psi)=c_{\alpha}\cdots c_{\alpha+k-1}E(u,\phi)$.
\par\noindent
{\bf Case (b): }
$g=a\cap c$, where $a=-a_0\cap \cdots -a_n$, $a_i\in H(k)$ and $c$ is a generalized codiagonal.\par\noindent
If $c_{\alpha}\cdots c_{\alpha+k-1}g=g$, then $g\in G(0)^{**}$. 
Else
$c_{\alpha}\cdots c_{\alpha+k-1}g$ is either $0$ or $V_k$, pending on whether $g=0$ or not.
The choice of $a\in G(0)^{**}$ in this case is also obvious.    
\par\noindent
{\bf Case (c): }
$g=d_0\cap -d_1\cap\cdots -d_n\cap b\cap c$, where $d_i\in H(k)$, $b$ is a generalized diagonal, and $c$ a
generalized codiagonal.\par\noindent
Assume that  $1\leq k<\omega$; else there is nothing to prove. \par\noindent
Let $\Gamma=\{\alpha,\cdots \alpha+k-1\}$.
By the proof of fact 6.7 ,  we have $c_{(\Gamma)}g$ 
is the finite union of elements of the form $a_0\cap- a_1\cdots \cap- a_n\cap b\cap c$ where
$a_0, \cdots a_n\in H(k)$, $b$ is a generalized diagonal and $c$ a generalized codiagonal.
Further we can assume that if $b=\cap \{d_{ij}^{V_k}:i,j\in \Delta_1\}$ and $c=\cap \{-d_{ij}^{V_k}:i,j\in \Delta_2, i\neq j\}$
then $\Delta_i\subseteq_{\omega} \alpha$, for $i\in \{1,2\}$. 
By the additivity of generalized cylindrifications (namely of $c_{(\Gamma)}$), we can and will assume that
$$c_{(\Gamma)}g=c_{(\Gamma_0)}E(v_0,\phi_0)\cap c_{(\Gamma_1)}E(v_1,\phi_1)\cdots 
\cap -c_{(\Gamma_n)}E(v_n,\phi_n)\cap b\cap c,$$
where for each $i\leq n$ we have  
$\Gamma\subseteq \Gamma_i\subseteq_{\omega}\alpha+k$,  $\phi_i\in F_k(v_i)^{**}$, $b$ is a 
generalized diagonal and $c$ is a generalized codiagonal as indicated above. 
By  fact 6.6 (i)   we can assume that $var(\phi_i)\cap \{x_{\alpha},\cdots x_{\alpha+k-1}\}=\emptyset$.
For $i\leq n$ let $u_i=v_i\upharpoonright \alpha$. 
Then $u_i\in {}^{\alpha}\alpha^{(Id)}$, and by fact 6.6 (i) $\phi_i\in F_0(u_i)^{**}$.
Let $b'=\cap \{d_{ij}^{V_0}:i,j\in \Delta_1\}$ and $c'=\cap \{d_{ij}^{V_0}:i,j\in \Delta_2, i\neq j\}$.
\par\noindent
Let $a= c_{(\Gamma_0-\Gamma)}E(u_0,\phi_0)\cap\cdots c_{(\Gamma_n-\Gamma)}E(u_n,\phi_n)\cap b'\cap c'$.
Then $a\in G(0)^{**}$ and an easy checked $i(k)a=c_{(\Gamma)}g$.
\end{demo}
\begin{athm}{Fact 6.9}
$G(0)^{*}=A(\alpha)$.
\end{athm}
\begin{demo}{Proof}
Let $u\in {}^{\alpha}\alpha^{(Id)}$ be one to one and $r\in R$. Then by definition we have
$$p(u,r)=E(u,C_r(x_0,x_1,\cdots )).$$
Hence $p(u,r)\in F_0(u)$.  Since $G(0)^{**}$ is a $Ws_{\alpha}$ 
containing the generators of $A(\alpha)$ 
it follows that
that $A(\alpha)\subseteq G(0)^{**}$.
For the other inclusion it is enough to show that for $u\in {}^{\alpha}\alpha^{(Id)}$ and $\phi\in F_0(u)$, we have
$E(u,\phi)\in A(\alpha)$. We start by showing that $E(u,T)\in A(\alpha)$ for all such $u$.
\par\noindent
{\bf Case (a): }
$u$ is one to one.\par\noindent 
Let $r\in R$ . Then $c_0(p(u,r))\cap c_1(p(u,r))=E(u,T)\in A(\alpha)$.
\par\noindent
{\bf Case (b): }
$u$ is not one to one.\par\noindent
Let $J$ be a finite subset of $\alpha$, such that $u(J)\subseteq J$ and $u(i)=i$ whenever $i\notin J$.
Such a $J$ exists, but of course is not unique. Fix one such $J$.  Let $u'=u|J$. Then $u':J\to J$. Let $n=|Rgu'|$. Then $n<|J|$, because $u$ hence $u'$ is not one to one.
We next procced as in the finite dimensional case; correlating a permuation $v$ to $u'$ as follows:
Let $y\in {}^nJ$ such that $y(0)<\cdots <y(n-1)$ and $\{u'_{y(0)},\cdots u'_{y(n-1)}\}=Rgu'$. 
\par\noindent
For $i<l$ put $m_{y(i)}=\min (u' {}^{-1}(u'_{y(i)}))$. Choose $v\in S_J$ ( a permutation on $J$) 
such that $v(m_{y(i)})=u'_{y(i)}$.
Such a $v$ exists. Let $I=J-\{m_{y(0)},\cdots ,m_{y(l-1)}\}$; suppose that $I=\{j_0,\cdots j_k\}$, $k\in \omega$ 
and $j_0<j_1\cdots <j_k$.
Let $t_i^l(x)=c_lx\cap d_{il}$. Let $v*=v\cup Id_{\alpha-J}$. Then $v*$ is a permutation on $\alpha$, and
like the finite case, we have
$$t_{j(k)}^{u_{j(k)}}\circ \cdots \circ t_{j(0)}^{u_{j(0)}}E(v*,T)=E(u,T).$$
From case (a) we get that  $E(u,T)\in A(\alpha)$. Now let $u\in {}^{\alpha}\alpha^{(Id)}$ 
and $i,j\in ker(u)$. Then $E(u,x_i=x_j)=E(u,T)\cap d_{ij}\in A(\alpha)$.
Also for all  $u,v\in S_{\alpha}$, we have by Lemma 1**(i), 
$E(u,C_r(x_0,x_1,\cdots ))=E(u,C_r(x_{v(0)},\cdots )=p(u,r)$.
Finally by noting that $E(u,\neg\phi)=E(u,T)-E(u,\phi)$ we get the desired. 
\end{demo}
\begin{athm}{Fact 6.10}
$A(\alpha)\in INr_{\alpha}Ws_{\beta}$, for all $\beta>\alpha$.
\end{athm}
\begin{demo}{Proof} 
By fact 6.9, it suffices to show that $A(\alpha)\in Nr_{\alpha}Ws_{\alpha+\beta}$ for infinite $\beta$.
So let $\beta\geq \omega$. Let $V_0={}^{\alpha}W^{(s)}$ be the unit of $A(\alpha)$. 
\par\noindent
Let $w\in W$. Let $s_{\beta}=s\cup \{(i,w):\alpha\leq i<\beta\}$. 
Then $s_{\beta}\in {}^{\beta}W$.
Let $V_{\beta}$ be the weak space ${}^{\alpha+\beta}W^{(s_{\beta})}$ and let 
$C(\alpha+\beta)$, or simply $C(\beta)$,  be the full $Ws_{\alpha+\beta}$ 
with unit $V_{\beta}$.
For $k<\omega$ let $i(k,\beta)$ be the function with domain $G(k)^{**}$, such that
$i(k,\beta)a=\{s\in V_{\beta}: s\upharpoonright \alpha+k\in a\}$, $a\in G(k)^{**}$.
Then $i(k,\beta):G(k)^{**}\to C(\beta)$.
Let $G(\beta)=Sg^{C(\beta)}\{i(k,\beta)G(k)^{**}: 0\leq k<\omega\}$. 
Then $G(\beta)\in Ws_{\alpha+\beta}$.
Moreover, it is easy to show  that $G(\beta)=\cup \{i(k,\beta)G(k)^{**}:0\leq k<\omega\}$.
From this together with facts 6.8 and 6.9, 
we get (using the same ideas of the proof of fact 3.4) 
that 
$A(\alpha)\cong Nr_{\alpha}G(\beta)$, thus 
$A(\alpha)\in INr_{\alpha}Ws_{\alpha+\beta}$.  
\end{demo}
Finally we show
\begin{theorem} The atom structure $\At\B$ is not weakly neat.That is for any atomic $\D\in \CA_{\alpha}$ if $\At\D\cong \At\B$, 
then $\D\notin \Nr_{\alpha}\CA_{\alpha+1}$.
\end{theorem}

\begin{proof}
Let $\D$ be an atomic algebra such that $\At\D\cong \At\B$. Then 
$\D$ will contain the term algebra which will be the disjoint copies indexed by $^{\alpha}\alpha^{(Id)}$ except that $P_{Id}$ is countable. 
The same cardinality trick implemented above works. 
Had it been a neat reduct that ${}_{\alpha}s(0,1)P_r=P_{Id}$ 
would be uncountable and this cannot happen.
\end{proof}


\end{document}